\documentclass[12pt]{amsart}
\usepackage{amsmath,amsfonts,amssymb,graphicx,hyperref, enumitem,amsthm,mathabx,amsrefs}
\hypersetup{colorlinks=true,citecolor=blue,linkcolor=blue,urlcolor=blue,pdfstartview=FitH, bookmarksdepth=3}
\usepackage{mathrsfs}
\usepackage{array}
\usepackage[noabbrev, nameinlink]{cleveref}

\setlist[enumerate,1]{label={\arabic*.}}

\newcommand{\paren}[1]{\ensuremath{\left( #1 \right)}}
\newcommand{\set}[1]{\ensuremath{\left\{ #1 \right\}}}
\newcommand{\setdiv}{\,\middle|\,}
\newcommand{\abs}[1]{\ensuremath{\left| #1 \right|}}
\newcommand{\norm}[1]{\ensuremath{\left\| #1 \right\|}}
\newcommand{\braces}[1]{\ensuremath{\left[ #1 \right]}}
\newcommand{\innerprod}[1]{\ensuremath{\left< #1 \right>}}

\newcommand{\Z}{\mathbb{Z}}
\newcommand{\R}{\mathbb{R}}
\newcommand{\N}{\mathbb{N}}
\newcommand{\Q}{\mathbb{Q}}
\newcommand{\C}{\mathbb{C}}

\newcommand{\Matrix}[1]{\begin{pmatrix}#1\end{pmatrix}}
\newcommand{\SmallMatrix}[1]{\left(\begin{smallmatrix}#1\end{smallmatrix}\right)}
\DeclareMathOperator{\diag}{diag}

\renewcommand{\Re}{{\mathop{\mathgroup\symoperators Re}}}
\renewcommand{\Im}{{\mathop{\mathgroup\symoperators Im}}}
\newcommand{\sgn}{{\mathop{\mathgroup\symoperators \,sgn}}}
\DeclareMathOperator*{\res}{res}
\DeclareMathOperator{\Tr}{Tr}

\newcommand{\e}[1]{e\paren{#1}}
\newcommand{\wbar}[1]{\overline{#1}}
\newcommand{\wtilde}[1]{\widetilde{#1}}
\newcommand{\what}[1]{\widehat{#1}}
\newcommand{\trans}[1]{{#1}^T}
\newcommand{\Max}[1]{\ensuremath{\max \set{#1}}}
\newcommand{\Min}[1]{\ensuremath{\min \set{#1}}}
\newcommand{\summod}[1]{\ensuremath{\,(\mathrm{mod}\,#1)}}

\newcommand{\pFqName}[2]{{_{#1}F_{#2}}}
\newcommand{\pFq}[5]{\pFqName{#1}{#2}\paren{\begin{matrix}#3;\\#4;\end{matrix}\,#5}}
\newcommand{\pFqStarName}[2]{{_{#1}F^*_{#2}}}
\newcommand{\pFqStar}[5]{\pFqStarName{#1}{#2}\paren{\begin{matrix}#3;\\#4;\end{matrix}\,#5}}

\newcommand{\piecewise}[1]{\left\{\begin{matrix}#1\end{matrix}\right.}
\newcommand{\If}{\mbox{if }}
\newcommand{\Otherwise}{\mbox{otherwise}}

\DeclareMathAlphabet{\mathcalligra}{T1}{calligra}{m}{n}
\newcommand{\WigDName}{\mathcal{D}}
\newcommand{\WigDMat}[1]{\WigDName^{#1}}
\newcommand{\WigD}[3]{\WigDMat{#1}_{#2,#3}}
\newcommand{\WigdName}{\mathcalligra{d}}

\newcommand{\tildek}[1]{\wtilde{k}\paren{#1}}
\newcommand{\Dtildek}[2]{\mathcal{R}^{#1}\paren{#2}}

\newcommand{\vpmpm}[1]{v_{_{#1}}}

\newcommand{\AdSq}{{\mathrm{Ad}}^2}

\newcommand{\specmu}{\mathbf{spec}}

\newcommand{\sinmu}{\mathbf{sin}}

\newcommand{\Weyl}{\frak{W}}

\newcommand{\Ai}{\operatorname{Ai}}

\newcommand{\envAi}{\operatorname{envAi}}

\newcommand{\envU}{\operatorname{envU}}
\newcommand{\envJ}{\operatorname{envJ}}
\newcommand{\arccosh}{\operatorname{arccosh}}

\theoremstyle{plain} 
\newtheorem{thm}{Theorem}
\newtheorem{cor}[thm]{Corollary}
\newtheorem{lem}[thm]{Lemma}
\newtheorem{prop}[thm]{Proposition}

\newcommand{\claim}{\par\noindent \textit{Claim: }}

\theoremstyle{definition}
\newtheorem*{defn*}{Definition}

\theoremstyle{remark}
\newtheorem*{rem*}{Remark}

\newtheorem{prob}{Problem}
\newtheorem*{prob*}{Problem}

\crefname{thm}{Theorem}{Theorems}
\crefname{lem}{Lemma}{Lemmas}
\crefname{prop}{Proposition}{Propositions}
\crefname{cor}{Corollary}{Corollaries}
\crefname{conj}{Conjecture}{Conjectures}
\crefname{defn}{Definition}{Definitions}
\crefname{prob}{Problem}{Problems}

\title{The arithmetic Kuznetsov formula on $GL(3)$, II:\\ The general case.}
\author{Jack Buttcane}
\date{16 March 2018}
\address{Department of Mathematics \& Statistics, 5752 Neville Hall, Orono, ME 04469, USA}
\email{jack.buttcane@maine.edu}
\thanks{During the time of this research, the author was supported by NSF grant DMS-1916598.}

\begin{document}

\begin{abstract}
We obtain the last of the standard Kuznetsov formulas for $SL(3,\Z)$.
In the previous paper, we were able to exploit the relationship between the positive-sign Bessel function and the Whittaker function to apply Wallach's Whittaker expansion; now we demonstrate the expansion of functions into Bessel functions for all four signs, generalizing Wallach's theorem for $SL(3)$.
As applications, we again consider the Kloosterman zeta functions and smooth sums of Kloosterman sums.
The new Kloosterman zeta functions pose the same difficulties as we saw with the positive-sign case, but for the negative-sign case, we obtain some analytic continuation of the unweighted zeta function and give a sort of reflection formula that exactly demonstrates the obstruction when the moduli are far apart.
The completion of the remaining sign cases means this work now both supersedes the author's thesis and completes the work started in the original paper of Bump, Friedberg and Goldfeld.
\end{abstract}


\maketitle

\section{Introduction}

In 1956, a paper of Selberg \cite{Sel01} introduced the study of trace formulas to the field of analytic number theory.
Selberg's trace formula attaches a geometric interpretation -- a sum over conjugacy classes of the discrete group -- to a reasonably arbitrary sum over the spectrum of a discrete quotient of a symmetric space.
In 1978, Bruggeman \cite{Brugg01} introduced a type of trace formula weighted on the spectral side by Fourier-Whittaker coefficients which has a more arithmetic interpretation as a sum over certain exponential sums of the type commonly attributed to Kloosterman.
In 1980, apparently independent of Bruggeman, Kuznetsov \cite{Kuz01} developed an identical formula and gave its inverse, expressing a reasonably arbitrary sum of Kloosterman sums as a sum over the spectrum of $PSL(2,\Z)\backslash PSL(2,\R)/SO(2,\R)$, weighted by the Fourier-Whittaker coefficients.
The formulas of Selberg and Kuznetsov have become a cornerstone of modern analytic number theory.
In particular, the inversion of Kuznetsov's formula is something which cannot readily be applied to the Selberg trace formula, and this is key to many \cite{BFKMM01,MeRizwan01} and varied \cite{Duke01,DFI02} results; it is this step that we achieve now for the generalization to $GL(3)$.

In the time since Kuznetsov's work, these trace formulas have been generalized in many different directions.
We now have the Arthur-Selberg trace formula which extends the Selberg formula to very general groups \cite{Arthur01}; multiple forms exist to handle certain difficulties related to the lack of absolute convergence in the naive formula and other technical considerations.
Of the generalizations of the spectral Kuznetsov formula, one should point out those of Miatello -- Wallach \cite{MiaWallach01} for rank-one groups and Li \cite{Gold01} for $SL(n,\Z)$; the Kloosterman sums occuring in these formulas are described by the Bruhat decomposition of the discrete group in question.

In the series of papers \cite{SpectralKuz,WeylI,WeylII,ArithKuzI}, we have considered the full generalization of Kuznetsov's formulas to $SL(3,\Z)$.
There are essentially seven (families of) such formulas:
When the test function is on the spectral side, one wishes to collect the forms by their minimal weight (i.e. according to the ramification at the Archimedean place), which yields a two parameter spectrum of weight-zero Maass forms treated in \cite{SpectralKuz}, a two parameter spectrum of weight-one Maass forms treated in \cite{WeylI}, and a sequence of one-parameter Maass forms of each weight $2 \le d \in \N$ treated in \cite{WeylII}.
When the test function is on the arithmetic (geometric) side, one should collect the long-element Kloosterman sums by their signs; since there are two moduli -- or, equivalently, two pairs of indices -- for the $SL(3)$ Kloosterman sums, this leads to four choices of signs, and the first of these was treated in \cite{ArithKuzI}.
The current paper completes this project by developing the remaining three arithmetic Kuznetsov formulas.
Of course, one could hope for two more arithmetic Kuznetsov formulas with the test function on the hyper-Kloosterman sum terms, but though it is obvious that such formulas exist (a simple modification of \eqref{eq:PreBesselExpand} shows integrals of the $w_5$ Bessel function are dense in the Schwartz-class functions on $\R^\times$), it is not clear that one can produce a version that is useful, and we will discuss this a bit more, below.

To apply the Kuznetsov formula, one needs an understanding of the Kloosterman sums, which is given in \cite{BFG,Stevens,DabFish, LargeSieve}, and an understanding of the integral transforms and special functions; the latter is typically the main obstruction (one may consider that the point of \cite{SpectralKuz} was the evaluation of an integral of real dimension 13).
The main effort of the papers \cite{SpectralKuz,WeylI,WeylII} was to show that the integral transforms can be written as kernel integral transforms and to provide some basic (but useable) integral representations of the kernel functions, which we call the $GL(3)$ Bessel functions.
(Now there are several more such representations, most notably in \cite{Subconv,GPSSubconv,MeFan01}.)
Obtaining the arithmetic Kuznetsov formulas then typically proceeds by an inversion formula.
The paper \cite{ArithKuzI} was able to accomplish this for the positive-sign Kloosterman sums by noting that in that case, the Bessel function is the spherical Whittaker function, and the inverse Whittaker transform has an inversion due to Wallach \cite{Wallach}.
For the general case, we have no such inversion formula, and constructing it will be the bulk of the work of the current paper.

We simultaneously prove the arithmetic Kuznetsov formula at all signs (including the positive-sign case) and the expansion of Schwartz-class functions on $(\R^\times)^2$ into the $GL(3)$ Bessel functions; the heart of the argument is that a certain complicated integral is really just the projection onto the span of a single Bessel function given by the usual inner product.
The Bessel expansion can be regarded as a new entry in the relatively thin book of theorems on rank-two hypergeometric functions, and relates to Wallach's theorem as follows:
If one regards the $K$-Bessel function as the Whittaker function of a spherical $GL(2)$ Maass form, then Wallach's theorem \cite[Theorem 15.9.3]{Wallach} generalizes Kontorovich-Lebedev inversion \cite{Lebedev01,KontLeb01} to Whittaker functions which are not necessarily spherical and to real reductive groups.
On the other hand, if one regards the $K$- and $J$-Bessel functions as the kernel functions occuring in the Kuznetsov formula, then the current theorem generalizes the Kontorovich-Lebedev and Sears-Titchmarsh \cite{SearsTitch01} inversions to $GL(3)$.
In this sense, one should regard Sears-Titchmarsh as the deeper theorem, as it must treat a certain point spectrum in addition to the continuous spectrum of Kontorovich-Lebedev; in a similar manner, the Bessel expansion of the current paper is much deeper than \cite[Theorem 3.1]{ArithKuzI}, as it requires the full strength of Wallach's theorem for $GL(3)$, while \cite[Theorem 3.1]{ArithKuzI} is just the spherical case of Wallach's theorem.
In particular, we must treat the entire spectrum of $GL(3)$, including both two-parameter spectra (for $d=0,1$) and the sequence of one-parameter spectra that are the generalized principal series representations.

Any arithmetic Kuznetsov formula worth its salt should immediately have something to say about cancellation in smooth sums of Kloosterman sums and analytic continuation of the Kloosterman zeta function -- in a different guise, these were Kuznetsov's original goals, and we give them here as basic applications.
As in the previous paper \cite{ArithKuzI}, we encounter difficulties when the moduli of the Kloosterman sums are far apart, so we must consider a lightly modified Kloosterman zeta function and we will see some extra error terms in the bounds for smooth sums.
On $SL(2)$, the nicest case for a Kloosterman zeta function is the positive-sign case (with the $J$-Bessel function), and for its analog on $SL(3)$ (the negative-sign case due to a difference in definition of the long Weyl element), we are able to give a sort of reflection formula which precisely quantifies these difficulties.

The proof of these applications will appear deceptively simple, but one should note that this is only \emph{after} the development of the Kuzetnsov formulas, where we are forearmed with detailed knowledge of the special functions and integral transforms.
If one steps back a bit to see that we are simultaneously treating the expansion of a Poincar\'e series over the full spectrum of $SL(3,\Z)\backslash PSL(3,\R)$ using deep methods (meaning the opposite of the so-called ``soft'' methods), the fact that this becomes so easy is an excellent argument for the strength of the techniques that led up to it.

\section{Some notation for $GL(3)$}
Let $G=PSL(3,\R) = GL(3,\R)/\R^\times$ and $\Gamma=SL(3,\Z)$.
The Iwasawa decomposition of $G$ is $G=U(\R) Y^+ K$ using the groups $K=SO(3,\R)$,
\[ U(R) = \set{\Matrix{1&x_2&x_3\\&1&x_1\\&&1} \setdiv x_i\in R}, \qquad R \in \set{\R,\Q,\Z}, \]
\[ Y^+ = \set{\diag\set{y_1 y_2,y_1,1} \setdiv y_1,y_2 > 0}. \]
The measure on the space $U(\R)$ is simply $dx := dx_1 \, dx_2 \, dx_3$, and the measure on $Y^+$ is
\[ dy := \frac{dy_1 \, dy_2}{(y_1 y_2)^3}, \]
so that the measure on $G$ is $dg :=dx \, dy \, dk$, where $dk$ is the Haar probability measure on $K$ (see \cite[section 2.2.1]{HWI}).
We generally identify elements of quotient spaces with their coset representatives, and in particular, we view $U(\R)$, $Y^+$, $K$ and $\Gamma$ as subsets of $G$.

Characters of $U(\R)$ are given by
\[ \psi_m(x) = \psi_{m_1,m_2}(x) = \e{m_1 x_1+m_2 x_2}, \qquad \e{t} = e^{2\pi i t}, \]
where $m\in\R^2$; we say $\psi=\psi_m$ is non-degenerate when $m_1 m_2 \ne 0$.
Characters of $Y^+$ are given by the power function on $3\times 3$ diagonal matrices, defined by
\[ p_\mu\paren{\diag\set{a_1,a_2,a_3}} = \abs{a_1}^{\mu_1} \abs{a_2}^{\mu_2} \abs{a_3}^{\mu_3}, \]
where $\mu\in\C^3$.
We assume $\mu_1+\mu_2+\mu_3=0$ so this is defined modulo $\R^\times$, renormalize by $\rho=(1,0,-1)$, and extend by the Iwasawa decomposition
\[ p_{\rho+\mu}\paren{x y k} = y_1^{1-\mu_3} y_2^{1+\mu_1}, \qquad x\in U(\R),y\in Y^+,k\in K. \]
Integrals in $\mu$ use the permutation-invariant measure $d\mu=d\mu_1 \, d\mu_2$.

The Weyl group $\Weyl$ of $G$ contains the six matrices
\begin{equation*}
	\begin{array}{rclcrclcrcl}
		I &=& \Matrix{1\\&1\\&&1}, && w_2 &=& -\Matrix{&1\\1\\&&1}, && w_3 &=& -\Matrix{1\\&&1\\&1}, \\
		w_4 &=& \Matrix{&1\\&&1\\1}, && w_5 &=& \Matrix{&&1\\1\\&1}, && w_l &=& -\Matrix{&&1\\&1\\1},
	\end{array}
\end{equation*}
with the relations $w_3 w_2=w_4$, $w_2 w_3=w_5$ and $w_2 w_3 w_2=w_3 w_2 w_3=w_l$.
The subgroup of order three is $\Weyl_3 = \set{I, w_4, w_5}$.
The Weyl group induces an action on the coordinates of $\mu$ by $p_{\mu^w}(a):=p_\mu(waw^{-1})$, and we denote the coordinates of the permuted parameters by $\mu^w_i := (\mu^w)_i$, $i=1,2,3$.
Explicity, the action is
\begin{align*}
	\mu^I =& \paren{\mu_1,\mu_2,\mu_3}, & \mu^{w_2} =& \paren{\mu_2,\mu_1,\mu_3}, & \mu^{w_3} =& \paren{\mu_1,\mu_3,\mu_2}, \\
	\mu^{w_4} =& \paren{\mu_3,\mu_1,\mu_2}, & \mu^{w_5} =& \paren{\mu_2,\mu_3,\mu_1}, & \mu^{w_l} =& \paren{\mu_3,\mu_2,\mu_1}.
\end{align*}

The group of diagonal, orthogonal matrices $V \subset G$ contains the four matrices $v_{\varepsilon_1,\varepsilon_2}=\diag\set{\varepsilon_1,\varepsilon_1 \varepsilon_2,\varepsilon_2}, \varepsilon\in\set{\pm 1}^2$, which we abbreviate $V = \set{\vpmpm{++}, \vpmpm{+-}, \vpmpm{-+}, \vpmpm{--}}$.
We write $Y=Y^+ V$ for the diagonal matrices in $G$ so that $Y \cap K = V$.
We tend not to distinguish between elements of $Y$ and pairs in $(\R^\times)^2$, as the multiplication is the same.

We will also require the Bruhat decomposition $G=U(\R) Y \Weyl U(\R)$.
When taking the Bruhat decomposition of an element $\gamma\in\Gamma$, we have $\gamma=bcvwb'$ with $b,b'\in U(\Q)$, $v\in V$, $w\in \Weyl$ and $c$ of the form
\begin{align}
\label{eq:cmat}
	\Matrix{\frac{1}{c_2}\\&\frac{c_2}{c_1}\\&&c_1}, \qquad c_1, c_2\in \N.
\end{align}

The unitary, irreducible representations of $K$, up to isomorphism, are given by the Wigner $\WigDName$-matrices $\WigDMat{d}:K\to GL(2d+1,\C)$ for $0 \le d \in \Z$.
We treat $\WigDMat{d}$ primarily as a matrix-valued function with the usual properties
\[ \WigDMat{d}(kk')=\WigDMat{d}(k)\WigDMat{d}(k'), \qquad \trans{\wbar{\WigDMat{d}(k)}} = \WigDMat{d}(k)^{-1} = \WigDMat{d}(k^{-1}). \]
The entries of the matrix-valued function $\WigDMat{d}$ are indexed from the center:
\[ \WigDMat{d} = \Matrix{\WigD{d}{-d}{-d}&\ldots&\WigD{d}{-d}{d}\\ \vdots&\ddots&\vdots\\ \WigD{d}{d}{-d}&\ldots&\WigD{d}{d}{d}}. \]
The entries, rows, and columns of of the derived matrix- and vector-valued functions (e.g. the Whittaker function \eqref{eq:JacWhittDef}) will be indexed similarly.
As the Wigner $\WigDName$-matrices exhaust the equivalence classes of unitary, irreducible representations of the compact group $K$, they give a basis of $L^2(K)$, as in \cite[section 2.2.1]{HWI}, by the Peter-Weyl theorem.
We tend to refer to the index $d$ as the ``weight'' of the Wigner $\WigDName$-matrix and any associated objects (e.g. the Whittaker function, Maass forms, etc.).

A Maass form (cuspidal or Eisenstein) of weight $d$ and spectral parameters $\mu$ for $\Gamma$ is a row vector-valued (or matrix-valued) smooth function $f:\Gamma\backslash G\to\C^{2d+1}$ that transforms under $K$ as $f(gk)=f(g)\WigDMat{d}(k)$, satisfies a moderate growth condition and is an eigenfunction of the Casimir operators with eigenvalues matching $p_{\rho+\mu}$ (see \cref{sect:UniformSE}).
Define
\begin{align*}
	\frak{a}^0 =& \frak{a}^1 = \set{\mu\in\C^3\setdiv\mu_1+\mu_2+\mu_3=0}, \\
	\frak{a}^d =& \set{\mu^d(r) \setdiv r \in \C}, d \ge 2, \\
	\frak{a}^d_\delta =& \set{\mu\in\frak{a}^d\setdiv\abs{\Re(\mu_i)}<\delta}, d=0,1, \\
	\frak{a}^d_\delta =& \set{\mu^d(r)\in\frak{a}^d \setdiv \abs{\Re(2r)}<\delta}, d \ge 2, \\
	\frak{a}^{0,c}_\delta =& \frak{a}^{1,c} = \set{(x+it,-x+it,-2it)\in\C^3\setdiv \abs{x} \le \delta, t\in\R},
\end{align*}
where (for $d \ge 2$)
\[ \mu^d(r) = \paren{\tfrac{d-1}{2}+r,-\tfrac{d-1}{2}+r,-2r}, \]
then the spectral parameters $\mu$ of a Maass form is known to lie in either $\frak{a}^d_0$ or the (conjecturally non-existent) complementary spectrum $\frak{a}^{d,c}_\theta$ with $0 \le \theta < \frac{1}{2}$ for $d=0,1$.
We refer to $\theta$ as the Ramanujan-Selberg parameter, and by \cite[Proposition 1 of appendix 2]{KS01} (see also \cite[Theorem 4]{WeylI}) we may assume $\theta \le \frac{5}{14}$.

The action of the Lie algebra of $G$ on Maass forms gives rise to five operators $Y^a$, $\abs{a} \le 2$ on vector-valued Maass forms (see \cref{sect:LieAlgebra}) that change the weight $d \mapsto d+a$, and a Maass form is said to have minimal weight if it is sent to zero by the lowering operators $Y^{-1}, Y^{-2}$ and an eigenfunction of the $Y^0$ operator.

Throughout the paper, we take the term ``smooth'', in reference to some function, to mean infinitely differentiable on the domain.
The letters $x,y,k,g, v$ and $w$ will generally refer to elements of $U(\R)$, $Y^+$, $K$, $G$, $V$, and $\Weyl$, respectively.
The letters $\chi$ and $\psi$ will generally refer to characters of $V$ and $U(\R)$, respectively, and $\mu$ will always refer to an element of $\C^3$ satisfying $\mu_1+\mu_2+\mu_3=0$, except for the final section, \cref{sect:WhittBds}, where it is used as the index of a classical Whittaker function (in keeping with the notation of \cite{DLMF}).
Vectors (resp. matrices) not directly associated with the Wigner $\WigDName$-matrices, e.g. elements $n \in \Z^2$ are indexed in the traditional manner from the left-most entry (resp. the top-left entry), e.g.  $n=(n_1,n_2)$.
We do not use the primed notation $F'$ for derivatives, but rather to distinguish functions and variables with similar purpose.

\section{Results}
\label{sect:Results}

The primary objects of study in this paper are the $GL(3)$ Bessel functions $K^d(y,\mu)$, $y\in Y$, of each weight $d$, which are introduced in \cref{sect:GL3Bessel}, and the long-element $SL(3,\Z)$ Kloosterman sum, defined in terms of the Bruhat decomposition as
\begin{align*}
	S_{w_l}(\psi_m,\psi_n;cv) =& \sum_{\substack{bcvw_l b'\in U(\Z)\backslash\Gamma/U(\Z)\\b,b'\in U(\Q)}} \psi_m(b) \psi_n(b'),
\end{align*}
with $c$ as in \eqref{eq:cmat}, and we write the sum in terms of modular arithmetic in \cref{sect:wlKloostermanSum}.
This sum occurs naturally in the Fourier expansion of an $SL(3,\Z)$ Poincar\'e series, and hence also in the $SL(3,\Z)$ spectral Kuznetsov formulas \cite{SpectralKuz,WeylI,WeylII}, where it forms the main off-diagonal term on the arithmetic/geometric side.
Up to factors that are small in summation, the sum $S_{w_l}(\psi_m,\psi_n;cv)$ has the square-root cancellation bound $(c_1 c_2)^{\frac{1}{2}+\epsilon}$, see \eqref{eq:SquareRoot}.

\subsection{The Bessel expansion}

In \cref{sect:BesselExpand}, we will use Wallach's theorem for $GL(3)$ (see \cref{sect:Wallach}) along with bounds on the Whittaker and Bessel functions (see \cref{sect:AnalyticPrelims}), and the uniqueness of spherical functions (see \cref{sect:Gode}) to show that nice functions of $Y$ can be expanded into (sums/integrals of) the $GL(3)$ Bessel functions.
Along the way, we show the generalization of \cite[Lemma 3]{Me01}, the Fourier transform of the spherical function, to the weight one case, but we eventually arrive at:
\begin{thm}
\label{thm:BesselExpand}
For $f:Y\to\C$ smooth and compactly supported,
\begin{align*}
	f(y) =& \frac{1}{4} \sum_{d\ge 0} \int_{\frak{a}^d_0} K^d_{w_l}(y,\mu) \int_Y f(y') \wbar{K^d_{w_l}(y',\mu)} dy' \, \sinmu^{d*}(\mu) d\mu.
\end{align*}
\end{thm}
The spectral weights $\sinmu^{d*}(\mu)$ are defined in \cref{sect:Wallach}.
(Compact support on $Y\cong (\R^\times)^2$ implies in particular that the support of $f(y)$ must be bounded away from the lines $y_i=0$.)
Note that \cref{thm:BesselExpand} utilizes the full $GL(3)$ spectrum, as compared to the spherical case in \cite[Theorem 3.1]{ArithKuzI}.

The factor $\frac{1}{4}$ in the theorem is exactly the difference between probability measure on $V$ (which we use in the Iwasawa decomposition $dg=dx\,dy\,dk$, see \cref{sect:BruhatDecomp}), and the counting measure on $V$ (which we use in $Y=V Y^+$).
Since we are ultimately interested in isolating the individual signs, we prefer the counting measure, and hence the factor $\frac{1}{4}$.
Note that we can achieve the expansion into Bessel functions at each sign by taking a test function $f_Y(vy) = \delta_{v=v^*} f_{Y^+}(y)$, and interestingly, this proves a sort of orthogonality of the Bessel transform at one sign to the inverse Bessel transform at any other sign (since $f_Y(vy)=0$ for $v\ne v^*$), provided one properly collects transforms/inverse transforms at every weight.


\subsection{The arithmetic Kuznetsov formula}

In \cref{sect:ArithKuzProof}, we construct the arithmetic Kuznetsov formula (for all signs) for $SL(3,\Z)$.
We use a construction of Cogdell and Piatetski-Shapiro (see \cref{sect:CPSmethod}) to dodge the spectral Kuznetsov formulas and the need to show the orthogonality of the Bessel transforms at $w = w_l$ to the inverse Bessel transforms at $w \ne w_l$, though the orthogonality of the various inverse transforms is an interesting question in its own right.
Their construction produces a very complicated spectral expansion, but comparing to the proof of the Bessel expansion, we can see that a certain sum over the matrix coefficients of a representation is, in fact, just the inner product of our test function with the relevant Bessel function.
Thus we arrive at the very concise formula \eqref{eq:PreArithKuz}, and with some renormalizations, plus assuming (as we may) that our basis consists of Hecke eigenfunctions, we have
\begin{thm}
\label{thm:ArithKuzHecke}
	For $f:Y \to \C$ smooth and compactly supported and non-degenerate characters $\psi_m,\psi_n$, $m,n\in\Z^2$, we consider the sum
	\begin{align}
	\label{eq:KLDef}
		\mathcal{K}_L(f) =& \sum_{c\in\N^2} \sum_{\varepsilon\in\set{\pm1}^2} \frac{S_{w_l}(\psi_m,\psi_{\varepsilon n},c)}{c_1 c_2} f\paren{\frac{\varepsilon_2 m_1 n_2 c_2}{c_1^2},\frac{\varepsilon_1 m_2 n_1 c_1}{c_2^2}}.
	\end{align}
	Then we have the spectral interpretation
	\begin{align}
	\label{eq:ArithKuzHecke}
		\mathcal{K}_L(f) = \sum_{d=0}^\infty \paren{\mathcal{C}^d(\wtilde{f}^d)+\mathcal{E}^d_\text{max}(\wtilde{f}^d)}+\mathcal{E}_\text{min}(\wtilde{f}^0),
	\end{align}
	where
	\begin{align}
	\label{eq:HeckeCuspTerm}
		\mathcal{C}^d(F) =& \frac{4}{3} \sum_{\Phi\in\mathcal{S}_3^{d*}} \frac{\lambda_\Phi(n)\wbar{\lambda_\Phi(m)}}{L(\AdSq \Phi,1)} F(\mu_\Phi), \\
	\label{eq:HeckeMaxEisenTerm}
		\mathcal{E}^d_\text{max}(F) =& 2 \sum_{\Phi\in\mathcal{S}_2^{d*}} \int_{\Re(r) = 0} \frac{\lambda_\Phi(n,r)\wbar{\lambda_\Phi(m,r)} F(\mu_\Phi+r,-\mu_\Phi+r,-2r)}{L(\AdSq \Phi,1) \, L(\Phi,1+3r)\,L(\Phi,1-3r)} \frac{dr}{2\pi i}, \\
	\label{eq:HeckeMinEisenTerm}
		\mathcal{E}_\text{min}(F) =& \frac{2}{3} \int_{\Re(\mu) = 0} \frac{\lambda_E(n,\mu)\wbar{\lambda_E(m,\mu)}}{\prod_{i\ne j} \zeta(1+\mu_i-\mu_j)} F(\mu) \frac{d\mu}{(2\pi i)^2}, 
	\end{align}
	using the integral transform
	\begin{align}
	\label{eq:yfBesselTrans}
		\wtilde{f}^d(\mu) =& \frac{1}{4} \int_Y \abs{y_1 y_2} \, f(y) \wbar{K^d_{w_l}(y,\mu)} dy.
	\end{align}
\end{thm}
Here we take $\mathcal{S}_n^{d*}$ to be an orthonormal basis of Maass cusp forms for $SL(n,\Z)$ of minimal weight $d$ which are eigenfunctions of all of the Hecke operators, and $\lambda_\Phi(n)$,$\lambda_\Phi(n,r)$, and $\lambda_E(n,\mu)$ refer to the $n$-th Hecke eigenvalue of the cusp form $\Phi \in \mathcal{S}_3^{d*}$, the maximal parabolic Eisenstein series with spectral parameter $r$ twisted by $\Phi \in \mathcal{S}_2^{d*}$, and the minimal parabolic Eisenstein series with spectral parameters $\mu$, respectively.
These objects are described in detail in \cref{sect:UniformSE,sect:HeckeConvert}.

Because the Bessel function $K^d(y,\mu)$ disappears for positive $y$ when $d \ge 2$, a test function supported on positive $y$ only sees the weight $d=0,1$ terms, while the other signs always see the full spectral expansion.
This matches the behavior of the $GL(2)$ Kuznetsov formulas \cite{Kuz01}, though the positive $y$ case here corresponds to the negative case there, since we have separated the $V$-element $\vpmpm{--} \in G$ from the long Weyl element $w_l$, something that is not possible for $\SmallMatrix{-1\\&1} \notin SL(2,\R)$.

We give a weaker hypothesis for the test function in \cref{thm:ArithKuzHecke}
\begin{prop}
\label{prop:ArithKuzConv}
Suppose $f:(\R^+)^2\to\C$ satisfies
\begin{enumerate}
\item $f(y)$ is bounded by $\abs{y_1 y_2}^{1+\epsilon}$ as $y_1 y_2 \to 0$, for some $\epsilon > 0$,
\item for $j_1,j_2\le 6,j_1+j_2\le 6$, the derivatives $\partial_{y_1}^{j_1} \partial_{y_2}^{j_2} f(y)$ exist and are continuous,
\item for $j_1,j_2\le 6,j_1+j_2\le 5$, the functions
\[ \paren{\abs{y_1}^{-\theta-\epsilon}+\abs{y_1}^{\frac{8}{3}+\epsilon}} \paren{\abs{y_2}^{-\theta-\epsilon}+\abs{y_2}^{\frac{8}{3}+\epsilon}} y_1^{j_1-1} y_2^{j_2-1} \partial_{y_1}^{j_1} \partial_{y_2}^{j_2} (y_1 y_2 f(y)) \]
are all zero on the boundaries $y_i \to 0, \pm\infty$,
\item $f(y), f^*_1(y),f^*_2(y) \ll \prod_{i=1}^2 (1+\abs{y_i})^{\frac{1}{4}-\frac{1}{2}\theta-\epsilon}(1+\abs{y_i}^{-1})^{-\theta-\epsilon}$ for all $y\in(\R^+)^2$, and some $\epsilon > 0$,
\end{enumerate}
where $\theta$ is the Ramanujan-Selberg parameter, and
\[ f^*_i(y) := (y_1 y_2)^{-1} (\wtilde{\Delta_i})^{4-i} (y_1 y_2 f(y)), \]
\begin{align}
\label{eq:tildeDelta1Def}
	\wtilde{\Delta_1} =& -y_1^2 \partial_{y_1}^2-y_2^2\partial_{y_2}^2+y_1 y_2 \partial_{y_1} \partial_{y_2}+4\pi^2(y_1+y_2), \\
\label{eq:tildeDelta2Def}
	\wtilde{\Delta_2} =& y_1^2 \partial_{y_1}^2-y_1^2 y_2 \partial_{y_1}^2 \partial_{y_2}+y_1 y_2^2 \partial_{y_1} \partial_{y_2}^2-y_2^2 \partial_{y_2}^2+4 \pi ^2 y_1 y_2 \partial_{y_2}-4 \pi ^2 y_1 y_2 \partial_{y_1}+4 \pi ^2 (y_2-y_1).
\end{align}
Then \eqref{eq:ArithKuzHecke} holds for the test function $f$.
\end{prop}
We note that (even though it's less complicated) this is necessarily weaker than the positive-sign case \cite[Proposition 2.2]{ArithKuzI} (note that $\wtilde{\Delta_1}$ there is the restriction to spherical Whittaker functions instead of the restriction to Bessel functions), as the Bessel functions at the other signs do not have exponential decay at infinity.
Using \cref{lem:DoubleBesselBound,lem:BadDoubleBesselBound,lem:BesselDervAsymp} in place of the asymptotics of the Whittaker function, the proof of this proposition is essentially identical to \cite[Proposition 2.2]{ArithKuzI}, and we omit it; the only new feature occurs when the Laplacian eigenvalue is small compared to $\norm{\mu}$, which is discussed in \cref{sect:WeylLaw}.

\subsection{Smooth sums of Kloosterman sums}
\label{sect:SmoothSumsThm}

Kuznetsov's original application for his formulas was to show good cancellation in moduli sums of Kloosterman sums, answering a conjecture of Linnik.
His result was for a sharp cutoff function, but this follows from the smoothed version, and we only consider the smoothed version (which is more useful, in practice).
The proof is simply to apply \cref{thm:ArithKuzHecke} to a test function of the form
\begin{align}
\label{eq:fXDef}
	f_X(y) := f\paren{\frac{X_1}{m_1 n_2} y_1,\frac{X_2}{m_2 n_1} y_2}
\end{align}
and bound the resulting spectral expansion using the Weyl law of \cref{sect:WeylLaw}.
We have
\begin{thm}
\label{thm:SmoothSums}
Let $X_1, X_2 > 0$ with $X_1^2 X_2, X_1 X_2^2 > 1$.
Suppose $m,n\in \Z^2$ such that $m_1 m_2 n_1 n_2 \ne 0$, and $f$ smooth and compactly supported on $Y$, then
\begin{align*}
	\sum_{c\in\N^2} \sum_{\varepsilon\in\set{\pm1}^2} \frac{S_{w_l}(\psi_m,\psi_{\varepsilon n},c)}{c_1 c_2} f\paren{\frac{\varepsilon_2 X_1 c_2}{c_1^2},\frac{\varepsilon_1 X_2 c_1}{c_2^2}} \ll_{m,n,f,\epsilon}& (X_1 X_2)^{\theta+\epsilon}+X_1^{-1-\epsilon}+X_2^{-1-\epsilon},
\end{align*}
where $\theta$ is the Ramanujan-Selberg parameter.
\end{thm}
The $\theta+\epsilon$ can actually be improved to $\Max{\theta,\epsilon}$, as the $\epsilon$ comes from logarithmic factors induced by the digamma function when the coordinates of the spectral parameters are not distinct, but this cannot occur on the complementary spectrum.
The proof is given in \cref{sect:SmoothSums}.

Excluding $\epsilon$ powers, \cref{thm:SmoothSums} is the limit of current technology.
Of course, we expect that $\theta=0$ is the truth, and perhaps the remaining factor $(X_1 X_2)^\epsilon$ should be replaced by logarithms, but the two terms $X_1^{-1-\epsilon}+X_2^{-1-\epsilon}$ are a bit more mysterious.
These terms are relevant when one of $X_i < 1$; as we have $X_1 \sim c_1^2/c_2$, $X_2 \sim c_2^2/c_1$, this occurs only when $c_1$ and $c_2$ are very far apart, i.e. $c_1 < c_2^{1/2}$ or visa versa.
In fact, the above bound loses to the trivial bound $(X_1 X_2)^{\frac{1}{2}+\epsilon}$ when, say, $X_1 < X_2^{-1/3}$ or equivalently, $c_1 < c_2^{1/5}$.
One might hope to tackle this case with $GL(2)$ methods, but this is an extremely difficult problem, see \cref{sect:WhatsNext}.

\cref{thm:SmoothSums} immediately implies
\begin{cor}
\label{cor:SpecKuz2CondConv}
	The long-element term of the naive weight two spectral Kuznetsov formula converges conditionally.
\end{cor}
In the weight two spectral Kuznetsov formula, if one takes the kernel function for the integral transform on the sums of long-element Kloosterman sums to be the $GL(3)$ long-element Bessel function as given in \cref{sect:GL3Bessel}, the sum of Kloosterman sums just fails to converge absolutely (assuming square-root cancellation is optimal for the Kloosterman sums), and we call this the naive formula.
The failure comes from a pole of the Mellin transform of $K^d_{w_l}(y,\mu^d(r))$ (more importantly, of Mellin transform of the corresponding Whittaker function) at ${(-\frac{1}{2}+r,-\frac{1}{2}-r)}$, and this leads to a sum of the form
\[ \sum_{c\in\N^2} \sum_{\varepsilon\in\set{\pm1}^2} \frac{S_{w_l}(\psi_m,\psi_{\varepsilon n},c)}{(c_1 c_2)^{3/2}} \wtilde{F}\paren{\frac{c_1}{c_2}}, \]
where $\wtilde{F}$ is smooth and $\wtilde{F}(x) \ll (x+x^{-1})^\delta$ for some $\delta > 0$.
By conditional convergence, we mean that the limit
\[ \lim_{R\to\infty} \sum_{C_1,C_2 \le R} \sum_{c\in\N^2} \omega\paren{\frac{c_1}{C_1},\frac{c_2}{C_2}} \sum_{\varepsilon\in\set{\pm1}^2} \frac{S_{w_l}(\psi_m,\psi_{\varepsilon n},c)}{(c_1 c_2)^{3/2}} \wtilde{F}\paren{\frac{c_1}{c_2}}, \]
converges where
\[ \sum_{C_1,C_2} \omega\paren{\frac{c_1}{C_1},\frac{c_2}{C_2}} = 1 \]
is a partition of unity with $\omega$ smooth and compactly supported on $(\R^+)^2$.

Then it is not to hard to see that the naive weight two spectral Kuznetsov formula holds with the long element term rearranged as above, and hence also that the Weyl law of \cite[Theorem 1]{WeylII} continues to hold at $d=2$.
The argument proceeds by analytic continuation as in the $d=3$ case of \cite[section 12]{WeylII}, but we do not include it here.

\subsection{The Kloosterman zeta function}
\label{sect:KloosZeta}

On $SL(2,\Z)$ in the positive sign case, the Kloosterman zeta function is an example of a Dirichlet series which has a functional equation (of sorts) and something resembling an Euler product (but isn't given by rational functions) for which the Riemann Hypothesis is known; the negative-sign case has a much more complicated meromorphic continuation.
On $SL(3,\Z)$, we encounter difficulties that exceed those of the negative-sign case on $SL(2,\Z)$ and, in fact, we are unable to show the meromorphic continuation of the Kloosterman zeta function itself, but must introduce some weights to reduce the contribution of the terms with the moduli very far apart; of course, this interfers with the equivalent of the Euler product for this multiple Dirichlet series.
The problem persists even in the best case of $\sgn(y)=(-,-)$, where it has rotated into the weight direction:
Even though the expected spectral interpretation converges exponentially at each $d$, the sum over $d$ itself does not appear to converge absolutely.
(If one hopes to cancel the difficulties by taking linear combinations, say like the sign-independent Kloosterman zeta function, this will also lead to disappointment.)

Our solution in the previous paper was to introduce an exponential decay factor; this is perhaps a bit excessive (a power $\prod_{i=1}^2 (1+ \abs{y_i})^{-t_i}$ for large $\Re(t_i)$ would be sufficient), but we use the same factor here, for continuity.
We apply \cref{prop:ArithKuzConv} to a function of the form
\begin{align}
\label{eq:KlZetaTestFun}
	f_{\varepsilon,s}(y) =& \delta_{\sgn(y)=\varepsilon} \abs{y_1}^{s_1} \abs{y_2}^{s_1} \exp\paren{-\abs{y_1}-\abs{y_2}},
\end{align}
and analyze the resulting spectral expansion to obtain
\begin{thm}
\label{thm:KloosZeta}
	Let $m,n\in \N^2$ and $\varepsilon\in\set{\pm1}^2$, then the Kloosterman zeta function
	\begin{align*}
		Z_{m,n}^{\varepsilon*}(s) :=& \sum_{c_1,c_2\in\N} \frac{S_{w_l}(\psi_m,\psi_n,\varepsilon c)}{c_1 c_2} f_{\varepsilon,s}\paren{\frac{m_1 n_2 c_2}{c_1^2},\frac{m_2 n_1 c_1}{c_2^2}}, \\
	\end{align*}
	initially convergent on $\Re(2s_1-s_2), \Re(2s_2-s_1) > \frac{1}{2}$, has meromorphic continuation to all of $s \in \C^2$ with potential poles whenever
	\begin{align*}
		s_1=&-\mu_i-\ell, &\text{or}&& s_2 =& \mu_i-\ell,
	\end{align*}
	 with $\mu=\mu_\varphi$ for some cusp form $\varphi$ of weight at most one, $i=1,2,3$ and $\ell \in \N_0=\N\cup\set{0}$, or
	\begin{align*}
		s_1=&-\tfrac{d-1}{2}-r-\ell, &\text{or}&& s_1=&2r-\ell &\text{or}&& s_2 =& -\tfrac{d-1}{2}+r-\ell, &\text{or}&& s_2=&-2r-\ell,
	\end{align*}
	 with $\mu^d(r)=\mu_\varphi$ for some cusp form $\varphi$ of weight at least two and $\ell \in \N_0=\N\cup\set{0}$, as well as potential poles coming from the Eisenstein series terms.
\end{thm}
The proof is essentially identical to that of \cite[Theorem 1.4]{ArithKuzI}.

\subsection{The reflection formula}
\label{sect:Reflect}

In the case of $\sgn(y)=(-,-)$, we wish to precisely identify the obstructions to optimal bounds for smooth sums and the meromorphic continuation of the zeta function, and we accomplish this by applying the spectral Kuznetsov formulas to the $d \ge 3$ terms of the arithmetic Kuznetsov formula.
The final formula is given by equation \eqref{eq:MainKlZetammCase} together with the expansions \eqref{eq:HIstarEval}, \eqref{eq:HwlstarEval} and \eqref{eq:Hw4starEval}.

For each $d \ge 0$, if $F(\mu)$ is Schwartz-class and holomorphic on $\frak{a}^d_{\frac{1}{2}+\delta}$ for some $\delta > 0$, the spectral Kuznetsov formula \cite[Theorem 4.1]{ArithKuzI}, \cite[Theorem 4]{WeylII} has the form
\[ \mathcal{C}^d(F)+\mathcal{E}^d_\text{max}(F)+\delta_{d=0} \mathcal{E}_\text{min}(F) = \sum_{w\in \set{I,w_4,w_5,w_l}} \sum_{v\in V} \sum_{c_1,c_2\in\N} \frac{S_w(\psi_m,\psi_n,cv)}{c_1 c_2} H_w^d\paren{F; mcvwn^{-1}w^{-1}}, \]
where
\begin{align*}
	H_w^d(F; y) =& \frac{1}{\abs{y_1 y_2}} \int_{\frak{a}^d_0} F(\mu) K^d_w(y, \mu) \specmu^d(\mu) d\mu,
\end{align*}
and for $d \ge 2$,
\begin{align}
\label{eq:specd}
	\specmu^d(\mu^d(r)) =& \frac{1}{16\pi^4 i}(d-1)\paren{\tfrac{d-1}{2}-3r}\paren{\tfrac{d-1}{2}+3r}.
\end{align}

If $f(y)=\delta_{\sgn(y)=(-,-)} f_{--}(-y)$ is Schwartz-class and supported on $\sgn(y)=(-,-)$, then $\wtilde{f}^d$ as in \cref{thm:ArithKuzHecke} meets the hypotheses of the spectral Kuznetsov formula, so for $3 \le D_1 \le D_2 \le \infty$, we have
\begin{gather}
\label{eq:MainKlZetammCase}
\begin{aligned}
	& \mathcal{K}_L(f) - \sum_{d\in [0,D_1)\cup(D_2,\infty)} \paren{\mathcal{C}^d(\wtilde{f}^d)+\mathcal{E}^d_\text{max}(\wtilde{f}^d)}-\mathcal{E}_\text{min}(\wtilde{f}^0) \\
	&= \sum_{w\in \set{I,w_4,w_5,w_l}} \sum_{v\in V} \sum_{c_1,c_2\in\N} \frac{S_{w_l}(\psi_m,\psi_n,cv)}{c_1 c_2} H_{w_l}^*\paren{mcvwn^{-1}w^{-1}},
\end{aligned}\\
	H_w^*(y) := \frac{1}{8\pi^3\abs{y_1 y_2}} \sum_{d=D_1}^{D_2} (d-1) \int_{\Re(r)=0} \wtilde{f}^d(\mu^d(r)) K^d_w(y, \mu^d(r)) \paren{\paren{\tfrac{d-1}{2}}^2-9r^2} \frac{dr}{2\pi i} \nonumber
\end{gather}
Note:  We specifically avoid applying the spectral Kuznetsov formula at $d=2$ for the reasons described in \cref{sect:SmoothSumsThm}.

We consider such a formula in the cases where either $f_{--}(y) = f_X(y)$ has the form \eqref{eq:fXDef} for smooth sums or $f_{--}(y) = f_M(y) = \prod_{i=1}^2 y_i^s \exp(1-y_i^{1/M})$ and take the limit as $M \to \infty$ for the \emph{unweighted} Kloosterman zeta function.
The analysis will be done in \cref{sect:Reflect}, but we describe the results here.

For the smooth sums, we will choose $D_1 = (X_1 X_2)^{\theta/2}$ and $D_2 = (X_1 X_2)^\epsilon \Max{X_1^{-1/2},X_2^{-1/2}}$ and we identify a pair of ``bad'' terms giving the primary asymptotics of $H_w^*(y)$ in this range; all other terms of \eqref{eq:MainKlZetammCase} are small compared to $(X_1 X_2)^{\theta+\epsilon}$, and in particular, if $X_1,X_2>1$ these bad terms disappear.
Ideally, one would then use a deeper analysis of these expressions in the case $X_1<1$ or $X_2 < 1$ to either show they are negligible or give an inversion formula which corrects for their presence (assuming they are not too close to the original function), but this beyond the scope of the current paper.

For the zeta function, we will choose $D_1=3$ and $D_2=\infty$.
We give only the first terms in a sequence of contour shifting and describe how one obtains the more general result; however, it should be clear that the formula \eqref{eq:MainKlZetammCase} can be used to continue the unweighted Kloosterman zeta function
\[ Z^{--}_{m,n}(s) := \sum_{c_1,c_2\in\N} \frac{S_{w_l}(\psi_m,\psi_n, (-c_1,-c_2))}{c_1 c_2} \paren{\frac{c_2}{c_1^2}}^{s_1} \paren{\frac{c_1}{c_2^2}}^{s_2}, \qquad m,n\in\N^2 \]
from its initial definition on $\Re(2s_1-s_2),\Re(2s_2-s_1)>\frac{1}{2}$ to the larger region $\Re(s_1),\Re(s_2)>\frac{1}{2}$.
The main obstruction comes from the same pair of terms as for the smooth sums, as the other terms tend to converge in left half-planes in $s$.
On the other hand, if we simply wrote down the spectral expansion of \cref{thm:ArithKuzHecke} for the test function $f_{--}(y) = y_1^{s_1} y_2^{s_2}$, we would see that the $d$ sum only converges for $\Re(s_1+s_2) < -1$ (so we cannot do this); in some sense we are demonstrating a type of conditional convergence for that sum on (a subregion of) $\Re(s_1+s_2) > -1$.

\section{Methods}
\subsection{The Bessel expansion}
In proving the Fourier expansion for functions on $[0,1]$, if one knows, by some other means, that the functions $\e{nx}$ form a basis, i.e. for any smooth function $f:[0,1]\to\C$, there are some coefficients $a_n$ so that
\[ f(x) = \sum_{n\in\Z} a_n \, \e{nx}, \]
where the series converges rapidly, then to find the $a_n$ it is sufficient to show the $\e{nx}$ are orthonormal on $[0,1]$.
It follows on taking the inner product of both sides of the previous display that
\[ \int_{[0,1]} f(x) \wbar{\e{nx}} dx = a_n. \]

The proof of the Bessel expansion in \cref{sect:BesselExpand} follows precisely along these lines.
Because the $GL(3)$ Bessel functions are defined in terms of an integral transform of the Whittaker function \eqref{eq:KdwDef}, it follows from Wallach's Whittaker expansion that the Bessel functions give a basis of the smooth, compactly supported functions $f:Y\to\C$.
The coefficients produced by Wallach's Whittaker expansion, in the form of \cref{thm:Wallach}, are quite unpleasant, and the work of the proof is to show the Bessel functions exhibit the required orthonormality, so these complicated coefficients are really just the inner product with the given basis element.
Of course, the proof relies on some delicate interchanges of integrals and some work to show everything converges, and we avoid the generalized principal series representations by analytic continuation, but in the end we arrive at \cref{thm:BesselExpand}.

The proof we give for the Bessel expansion should generalize very well in the group direction, provided one knows the relevant generalizations of Stade's formula, and assuming one can justify a certain troublesome interchange of integrals (see \cref{sect:InterchangeOfIntegrals}).
It is interesting to note that Wallach shows the expansion into Whittaker functions by a sort of Fourier transform on the left of Harish-Chandra's expansion into spherical functions, and we now establish the expansion into Bessel functions by the same sort of Fourier transform -- now on the right -- of Wallach's expansion.
(Of course, the theorem here is much less general.)

\subsection{The arithmetic Kuznetsov formula}
\label{sect:CPSmethod}
The simplest way to construct a Kuznetsov-type formula is to equate the spectral expansion and Bruhat decomposition of a Poincar\'e series (provided one already understands the convergence issues), and Cogdell and Piatetski-Shapiro \cite{CPS} have given a method of constructing arithmetic Kuznetsov formulas which avoids the use of inversion formulas.
Their method is simply to take the Fourier coefficient of a Poincar\'e series whose the kernel function is defined in terms of the Bruhat decomposition of the long-element Weyl cell.
The kernel function may easily be chosen to produce the desired test function on the sum of (long-element) Kloosterman sums, and since the long-element Weyl cell has full measure, the series will be well-defined in $L^2(\Gamma\backslash G)$ (provided the test function is sufficiently nice).
Then \cref{sect:UniformSE} gives a version of the spectral expansion which is nicely uniform with respect to the Fourier-Whittaker coefficients of Maass forms (cuspidal or Eisenstein) lying in the same irreducible representation (occuring in the Langlands decomposition of the right regular representation of $L^2(\Gamma\backslash G)$), and applying this to our newly-constructed Poincar\'e series gives a spectral interpretation as a sum of pairs of Fourier-Whittaker coefficients over all minimal-weight Maass forms (equivalently, over irreducible representations; again, both cuspidal and Eisenstein).
(Yangbo Ye \cite{Ye01} gave an adelic treatment of the method of Cogdell and Piatetski-Shapiro on $GL(n)$, but was unable to provide a spectral interpretation, as we do here for $SL(3)$.)

The weight functions in such a spectral expansion are very complicated and involve both an integral transform of the constructed kernel function and a sum over all weights (i.e. representations of $K$, or equivalently, vectors in the irreducible representation).
Fortunately, these unpleasant functions are precisely the complicated coefficients we encountered in proving the Bessel expansion, and we now know that these are just the inner product of our test function with the appropriate Bessel function, as expected.

One of the highlights of the method of Cogdell and Piatetski-Shapiro was that it replaced the Iwasawa decomposition and representation theory on $K$ with the Bruhat decomposition and representation theory on $U(\R)$ (i.e. the Kirilov model).
We keep the analysis on the Bruhat decomposition, but do not switch to the representation theory of $U(\R)$.
This is partially because we have a very explicit formulation of the Whittaker functions and the spectral expansion over the $K$-types in \cite{HWI,HWII}, but also because a key step in the argument relies on Godement's theorem \cite[Theorem 14]{Gode01} on the uniqueness of the spherical functions and the author is unaware of any similar results for $U(\R)$-invariant functions.
(The spherical functions are defined via representations of $K$ and involve a trace over a finite-dimensional vector space, while $U(\R)$-invariant functions would involve a trace over an $L^2$-space; the Bessel functions below are actually an example of such a function and there has been some work done by Baruch and Mao in this area, see \cref{sect:BaruchMao}.)

As always, the analysis contained here is specific to the place at infinity, and not the particular discrete group $SL(3,\Z)$.
A great deal of work is involved in analyzing the convergence of integrals to justify interchanges, but this work should see use in future generalizations in the group direction; the reader will notice that most of the work is being done on $SL(2)$ integrals and functions, after which the $SL(3)$ analysis follows fairly easily.
The exception here comes in proving the generalizations of Stade's formula, which are not (visibly) achieved by reducing to $SL(2)$.



\section{What's next?}
\label{sect:WhatsNext}
In the development of the spectral and arithmetic Kuznetsov formulas, we have encountered a single problem in several different facades:
\begin{prob}
\label{prob:wlKlDistantModuli}
Understand sums of the long-element Kloosterman sums when the moduli are far apart.
\end{prob}
\noindent Here, the outermost ``sums'' refers to either smooth moduli sums (as considered here) or bilinear forms (as considered in \cite{LargeSieve}).
If the moduli are $c_1,c_2\in\N$ and we set $y_1=\frac{c_2}{c_1^2}$, $y_2=\frac{c_1}{c_2^2}$, then ``far apart'' should mean at least $y_1 \ll y_2^{1-\epsilon}$ (or $y_2 \ll y_1^{1-\epsilon}$), but the true difficulties begin to occur in the range $y_1 \ll y_2^\epsilon$; to the extent of the techniques considered in this paper, the arithmetic Kuznetsov breaks down in the region $y_1 \ll y_2^{-1/3+\epsilon}$, see \cref{thm:SmoothSums}.
We take a moment to explore the different realizations of this problem.
We are not attempting to demonstrate equivalences, just to draw connections; futhermore, these connections are typically concrete in one direction and for the reverse direction, we rely on the nature of the Kuznetsov formulas as an isomorphism between spectral sums and (moduli) sums of Kloosterman sums.
We consider this problem in greater detail in \cref{sect:Reflect}.

It may well be the case that $GL(3)$ methods cannot work in the more extreme cases.
In particular, when one of the moduli is essentially fixed, we see from \cite[Theorem 2.2]{LargeSieve} that the $SL(3)$ long-element Kloosterman sum looks like a classical Kloosterman sum times an additive twist, and this should certainly be handled by $SL(2)$ methods.

\begin{prob}
Identify and remove the contribution of the maximal parabolic Eisenstein series from the $GL(3)$ spectral Kuznetsov formulas.
\end{prob}
\noindent In \cite{LargeSieve}, an unexpectedly large contribution arose in bilinear sums of Kloosterman sums when the moduli were far apart, and this directly impacted the strength of the resulting spectral large sieve inequalities.
Also in that paper, a sequence of coefficients (i.e. the input for the large sieve inequalites) was constructed such that the maximal parabolic Eisenstein terms in the spectral large sieve sum precisely attained this unusually large asymptotic.
That implies, at least, that improving the large sieve inequalities would require isolating and removing this contribution.
(One might hope that would extend to improvements in the corresponding subconvexity estimates \cite{Subconv,GPSSubconv}, but these involve \emph{smooth} index sums.)

\begin{prob}
Use $GL(3)$ spectral theory to study the symmetric squares of $GL(2)$ cusp forms.
\end{prob}
\noindent In the papers \cite{Subconv,GPSSubconv,LargeSieve}, the test functions were taken specifically to avoid the self-dual Maass forms as the analysis becomes much more difficult there (and higher moments must be considered, etc.).
This paper and the previous one \cite{ArithKuzI} demonstrate a direct connection between a class of Maass forms that are nearly self-dual (in the sense that one of the parameters is small $\abs{\mu_i} \ll \norm{\mu}^{\epsilon}$) and \cref{prob:wlKlDistantModuli} via either error terms in the bound for smooth sums or the (non-)convergence of the (unweighted) Kloosterman zeta function.

\begin{prob}
Use $GL(3)$ spectral theory to study the $\operatorname{Kl}_3$ hyper-Kloosterman sums.
\end{prob}
\noindent If we examine the Pl\"ucker coordinates in the Bruhat decomposition of $\Gamma$ in any Poincar\'e series, we see that the long-element Kloosterman sum with moduli $c_1, c_2 \in \N$ comes from $\gamma \in \Gamma$ with
\[ \gamma=\Matrix{*&*&*\\d&e&*\\c_1&b&*}, \qquad c_2=bd-c_1 e, \]
and the $\operatorname{Kl}_3$ hyper-Kloosterman sums with modulus $b$ (which are the $w_5$ Kloosterman sums) come from those $\gamma$ with $c_1=0$ and $c_2 = bd = b^2$ (when the indices are all 1).
In this sense, the hyper-Kloosterman sums are hidden in the limit of the long-element sums as the moduli become infinitely far apart; this behavior extends to other interesting exponential sums \cite{Welsh}.
In examining the spectral Kuznetsov formulas, one notices that a degree of freedom is dropped in the integral transforms for the $w_5$ Weyl element terms, which leads one to conjecture \cite{BFG} that a test function can be constructed such that the $w_5$ cell dominates the long-element cell (and for which the spectral expansion still converges), but this has yet to be achieved, though the current paper is a significant step towards the necessary understanding of the integral transforms.
(As mentioned above, it's obvious that integrals of the $w_5$ Bessel function are dense in the Schwartz functions on $\R^\times$, the key point would be to identify precisely to what extent the coefficients in the Bessel expansion are determined by the inner product with the corresponding Bessel functions.)

We do not anticipate a common solution to these problems; in fact, any progress at all on any individual problem should be considered a major achievement.

\section{Background}

\subsection{The Wigner $\WigDName$-matrices and the $V$ group}

If we describe elements $k=k(\alpha,\beta,\gamma)\in K$ in terms of the $Z$-$Y$-$Z$ Euler angles
\begin{align}
\label{eq:kabcDef}
	k(\alpha,\beta,\gamma) :=& k(\alpha,0,0) \, w_3 \, k(-\beta,0,0) \, w_3 \, k(\gamma,0,0), & k(\theta,0,0) := \SmallMatrix{\cos\theta&-\sin\theta&0\\ \sin\theta&\cos\theta&0\\ 0&0&1}, \\
\label{eq:tildekDef}
	\tildek{e^{i\alpha},e^{i\beta},e^{i\gamma}} :=& k(\alpha,\beta,\gamma)
\end{align}
then the Wigner $\WigDName$-matrix is primarily characterized by
\begin{align}
\label{eq:WigDPrimary}
	\WigDMat{d}(k(\theta,0,0)) = \Dtildek{d}{e^{i\theta}}, \qquad \Dtildek{d}{s} := \diag\set{s^d,\ldots,s^{-d}}, s\in\C.
\end{align}
The entries of the matrix $\WigDMat{d}(k(0,\beta,0)) = \WigDMat{d}(w_3) \WigDMat{d}(k(-\beta,0,0)) \WigDMat{d}(w_3)$ are known as the Wigner $\WigdName$-polynomials, but we will avoid the Wigner $\WigdName$-polynomials by treating $\WigDMat{d}(w_3)$ as a black box; that is, as some generic orthogonal matrix.
We use the facts (see \cite[section 2.2.2]{HWI})
\begin{align}
\label{eq:WigDv}
	\WigDMat{d}(v_{\varepsilon,+1}) =& \diag\set{\varepsilon^d,\ldots,\varepsilon^{-d}}, & \WigD{d}{m'}{m}(v_{\varepsilon,-1}) =& \piecewise{(-1)^d\varepsilon^{m'}&\If m'=-m,\\0&\Otherwise.}
\end{align}

A complete list of the characters of $V$, is given by $\chi_{\varepsilon_1,\varepsilon_2}, \varepsilon\in\set{\pm 1}^2$ which act on the generators by
\begin{align}
\label{eq:chipmpmdef}
	\chi_{\varepsilon_1,\varepsilon_2}(\vpmpm{-+})=&\varepsilon_1 & \chi_{\varepsilon_1,\varepsilon_2}(\vpmpm{+-}) =&\varepsilon_2.
\end{align}
These give rise to the projection operators
\begin{align}
\label{eq:SigmadchiDef}
	\Sigma^d_\chi = \frac{1}{4} \sum_{v\in V} \chi(v) \WigDMat{d}(v),
\end{align}
which are written out explicitly in \cite[section 2.2.2]{HWI} using the description \eqref{eq:WigDv}.
Using the abbreviation $\Sigma^d_{\varepsilon_1,\varepsilon_2} = \Sigma^d_{\chi_{\varepsilon_1,\varepsilon_2}}$, we define
\begin{align}
\label{eq:SigmadpdDef}
	\Sigma^{d'}_0 = \Sigma^{d'}_{++}, \qquad \Sigma^{d'}_1 = \Sigma^{d'}_{+-}, \qquad \Sigma^{d'}_d = \Sigma^{d'}_{(-1)^d,+}, d \ge 2.
\end{align}

\subsection{The Bruhat decomposition}
\label{sect:BruhatDecomp}
Up to a measure-zero subspace, we may identify $G$ with the long-element Weyl cell $U(\R)Yw_l U(\R)$.
The change of variables from $g\in G$ to $g=xyw_lx'$ with $x,x'\in U(\R)$ and $y \in Y$ is
\begin{align}
\label{eq:dgBruhat}
	dg = \frac{1}{8\pi^2} dx\, dy\, dx',
\end{align}
which follows from \cite[Lemma 4.3.2]{T02}:
If $d\bar{k}$ is Haar probability measure on $V\backslash K$, then
\begin{align}
\label{eq:dbark}
	d\bar{k} =& \frac{1}{2\pi^2} p_{2\rho}(w_l x') dx',
\end{align}
under the association between $V\backslash K$ and the $K$-part of $w_l x'$ as $x'$ runs through $U(\R)$; an identical expression holds with $w_l x'$ replaced by $\trans{x'}$.
The measure on $Y$ here is the natural measure
\[ \int_Y f(y) dy = \sum_{v\in V} \int_{Y^+} f(vy) dy, \]
and the factor of $4$ difference between \eqref{eq:dgBruhat} and \eqref{eq:dbark} is exactly the difference between counting measure and probability measure on $V$.

\subsection{The long-element Kloosterman sum}
\label{sect:wlKloostermanSum}
As in \cite{BFG}, we may write this explicitly as
\begin{align*}
	S_{w_l}(\psi_m,\psi_n;c \vpmpm{\varepsilon_1,\varepsilon_2}) =& S(-\varepsilon_1 n_2,-\varepsilon_2 n_1,m_1,m_2;c_1, c_2),
\end{align*}
where $S(m_1, m_2, n_1, n_2; D_1, D_2)$ is given by the rather complicated exponential sum
\begin{align*}
	&\sum_{\substack{B_1, C_1 \summod{D_1}\\B_2, C_2 \summod{D_2}\\ }} \e{\frac{m_1B_1 + n_1(Y_1 D_2 - Z_1 B_2)}{D_1} + \frac{m_2B_2 + n_2(Y_2 D_1 - Z_2B_1)}{D_2}},
\end{align*}
where the sum is restricted to
\[ D_1C_2 + B_1B_2 + D_2C_1 \equiv 0 \summod{D_1D_2}, \qquad (B_1, C_1, D_1) = (B_2, C_2, D_2) = 1, \]
and the $Y_i$ and $Z_i$ are defined by
\[ Y_1B_1 + Z_1C_1 \equiv 1 \pmod{D_1}, \qquad Y_2B_2 + Z_2C_2 \equiv 1 \pmod{D_2}. \]
We will not require in-depth knowledge about the sum itself, beyond the square-root cancellation bound (originally due to Stevens \cite{Stevens}, but later made global with explicit dependence on the indices in \cite{Me01}, and strengthened in \cite[(2.10)]{LargeSieve}):
\begin{align}
\label{eq:SquareRoot}
	S_{w_l}(\psi_m,\psi_n;c)^2 \ll (c_1 c_2)^{1+\epsilon} (c_1,c_2) (m_1 m_2 n_1 n_2, c_1, c_2)(m_1,n_2,c_1)(m_2,n_1,c_2).
\end{align}
The $\epsilon$ power can be improved to divisor functions.

\subsection{Whittaker functions}
\label{sect:WhittFuns}
From the Iwasawa decomposition, we define
\begin{align}
\label{eq:IdDef}
	I^d_\mu(xyk) = p_{\rho+\mu}(y) \WigDMat{d}(k).
\end{align}
The components of this matrix-valued function are essentially the elements of a principal series representation (which may be reducible).

The majority of work of the paper will be concerned with the $GL(3)$ special functions consisting of the Whittaker functions, the Bessel functions (\cref{sect:GL3Bessel}), and the spherical functions (\cref{sect:Gode}).
The long-element, matrix-valued Jacquet-Whittaker function at each $K$-type $\WigDMat{d}$ is defined by the integral
\begin{align}
\label{eq:JacWhittDef}
	W^d(g,\mu,\psi) := \int_{U(\R)} I^d_\mu(w_l u g) \wbar{\psi(u)} du.
\end{align}
It is easy to see that this satisfies
\begin{align}
\label{eq:WhittGFEs}
	W^d(xyk,\mu,\psi) &= \psi(x) W^d(y,\mu,\psi) \WigDMat{d}(k), & W^d(y,\mu,\psi_m) =& p_{\rho+\mu^{w_l}}(y) W^d(I,\mu,\psi_{ym}).
\end{align}

The functions $I^d_\mu(xyk)$ and $W^d(g,\mu,\psi)$, being defined in terms of a Wigner $\WigDName$-matrix, are again matrix-valued, and we index their rows $W^d_{m'}$, columns $W^d_{\cdot,m}$, and entries $W^d_{m',m}$ from the central entry, i.e. by the same convention as the Wigner $\WigDName$-matrices.

We will frequently use the fact that the rows $W^{d'}_{m'}(g,\mu^d(r),\psi_I)$ with $d \le m' \equiv d \pmod{2}$ are all identically zero, as are the rows $W^{d'}_{m'}(g,-\mu^d(r),\psi_I)$ with $-d < m' \equiv d \pmod{2}$.
This can be seen from the pole of the gamma function in the denominator of \eqref{eq:classWhittDef} in the integral representation \eqref{eq:WhittIntRepn}, below.
One particular consequence of this is that we may freely replace $\Sigma^{d'}_{(-1)^d,+} \mapsto \Sigma^{d'}_{(-1)^d,(-1)^d}$ in the product
\[ \Tr\paren{\Sigma^{d'}_d W^{d'}(y,\mu^d(r),\psi_I) \trans{\wbar{W^{d'}(y,-\wbar{\mu^d(r)},\psi_I)}}}, \]
which is precisely the effect of applying the $\mu \mapsto \mu^{w_2}$ functional equation (see below), and this continues to hold at $d=0,1$ for general $\mu$ as $\Sigma^{d'}_{+\pm}$ is $w_2$-invariant in those cases.

\subsubsection{The functional equations}
\label{sect:WhitFEs}
The Whittaker functions are entire in $\mu$ and satisfy the functional equations (\cite[Proposition 3.3]{HWI}),
\begin{align}
\label{eq:WhittFEs}
	W^d(g,\mu,\psi_I) = T^d(w,\mu) W^d(g,\mu^w,\psi_I), \qquad w \in \Weyl
\end{align}
generated by the matrices
\begin{align}
\label{eq:Tdw2}
	T^d(w_2,\mu) :=&\pi^{\mu_1-\mu_2} \Gamma^d_\mathcal{W}(\mu_2-\mu_1,+1), \\
\label{eq:Tdw3}
	T^d(w_3,\mu) :=& \pi^{\mu_2-\mu_3} \WigDMat{d}(\vpmpm{--}w_l) \Gamma^d_\mathcal{W}(\mu_3-\mu_2,+1) \WigDMat{d}(w_l\vpmpm{--}),
\end{align}
and $\Gamma^d_\mathcal{W}(u,\varepsilon)$ is a diagonal matrix coming from the functional equation of the classical Whittaker function \cite[(2.20)]{HWI}:
For $y > 0$, define $\mathcal{W}^d(y,u)$ to be the diagonal matrix-valued function with entries (see \cite[section 2.3.1]{HWI})
\begin{equation}
\label{eq:classWhittDef}
\begin{aligned}
	\mathcal{W}^d_{m,m}(y,u) =& \int_{-\infty}^\infty \paren{1+x^2}^{-\frac{1+u}{2}} \paren{\frac{1+ix}{\sqrt{1+x^2}}}^{-m} \e{-yx} dx \\
	=& \pi \frac{(\pi y)^{\frac{u-1}{2}}}{\Gamma\paren{\frac{1-\varepsilon m+u}{2}}} W_{-\frac{m}{2}, \frac{u}{2}}(4\pi y),
\end{aligned}
\end{equation}
(where $W_{\alpha,\beta}(y)$ is the classical Whittaker function), then we have the functional equations
\begin{align}
\label{eq:WhittGammas}
	\mathcal{W}^d(y,-u) =& (\pi y)^{-u} \Gamma^d_\mathcal{W}(u,1) \mathcal{W}^d(y,u) & \Gamma_{\mathcal{W},m,m}^d(u,+1) =& \frac{\Gamma\paren{\frac{1-m+u}{2}}}{\Gamma\paren{\frac{1-m-u}{2}}}.
\end{align}

In particular, we have
\begin{align}
\label{eq:Tdwl}
	T^d(w_l,\mu) =& \pi^{2(\mu_1-\mu_3)} \Gamma^{d'}_{\mathcal{W}}(\mu_2-\mu_1,+1) \WigDMat{d'}(\vpmpm{--}w_l) \Gamma^{d'}_{\mathcal{W}}(\mu_3-\mu_1,+1) \WigDMat{d'}(\vpmpm{--}w_l) \Gamma^{d'}_{\mathcal{W}}(\mu_3-\mu_2,+1).
\end{align}

The matrices $T^d(w,\mu)$ satisfy the orthogonality relation
\begin{align}
\label{eq:TdOrtho}
	\trans{\wbar{T^d(w,-\wbar{\mu})}} T^d(w,\mu) =& I,
\end{align}
and the commutation relation with $\Sigma^d_\chi$,
\begin{align}
\label{eq:TdSigmadComm}
	\Sigma^d_\chi T^d(w,\mu) = T^d(w,\mu) \Sigma^d_{\chi^w},
\end{align}
where $\chi^w(v)=\chi(wvw^{-1})$.
This corrects \cite[(3.26) and (3.27)]{HWI}, which have the inverse Weyl element on the wrong side.

\subsubsection{The minimal-weight Whittaker functions}
\label{sect:MinWtWhitt}
For $\alpha\in\set{0,1}^3$, $\beta,\eta\in\Z^3$, $\ell\in\Z^2$ and $s\in\C^2$ define
\begin{equation}
\begin{aligned}
	\Lambda^0(\mu) =& \pi^{-\frac{3}{2}+\mu_3-\mu_1} \Gamma\paren{\tfrac{1+\mu_1-\mu_2}{2}} \Gamma\paren{\tfrac{1+\mu_1-\mu_3}{2}} \Gamma\paren{\tfrac{1+\mu_2-\mu_3}{2}}, \\
	\Lambda^1(\mu) =& \sqrt{2} \pi^{-\frac{3}{2}+\mu_3-\mu_1} \Gamma\paren{\tfrac{1+\mu_1-\mu_2}{2}} \Gamma\paren{\tfrac{2+\mu_1-\mu_3}{2}} \Gamma\paren{\tfrac{2+\mu_2-\mu_3}{2}}, \\
	\Lambda^d(\mu^d(r)) =& 2 (-1)^d (2\pi)^{-\frac{d+1}{2}-3r} \Gamma(d) \Gamma\paren{\tfrac{d+1}{2}+3r},
\end{aligned}
\end{equation}
\begin{align}
	\wtilde{G}(d,\beta,\eta, s,\mu)= \frac{\prod_{i=1}^3 \Gamma\paren{\frac{\beta_i+s_1-\mu_i}{2}} \Gamma\paren{\frac{\eta_i+s_2+\mu_i}{2}}}{\Gamma\paren{\frac{s_1+s_2+\sum_i (\beta_i+\eta_i)-2d}{2}}},
\end{align}
\begin{equation}
\begin{aligned}
	\wtilde{G}^0(\ell, s,\mu) =& \wtilde{G}(0,0,0,s,\mu), \\
	\wtilde{G}^1(\ell,s,\mu) =& \wtilde{G}(1,(\ell_1,\ell_1,1-\ell_1),(\ell_2,\ell_2,1-\ell_2),s,\mu), \\
	\wtilde{G}^d(\ell,s,\mu) =& \wtilde{G}(d,(d,0,\ell_1),(0,d,\ell_2),s,\mu), \qquad d \ge 2.
\end{aligned}
\end{equation}
Now for $\abs{m'}\le d$, write $m'=\varepsilon m$ with $\varepsilon = \pm 1$ and $0 \le m \le d$, set
\begin{align}
	G^d_{m'}(s,\mu) =& \sqrt{\binom{2d}{d+m}} \sum_{\ell=0}^{m} \varepsilon^\ell \binom{m}{\ell} \wtilde{G}^d((d-m,\ell), s,\mu),
\end{align}
and take $G^d(s,\mu)$ to be the vector with coordinates $G^d_{m'}(s,\mu)$, $m'=-d,\ldots,d$.

We define the completed minimal-weight Whittaker function at each weight $d$ as
\begin{align}
	W^{d*}(y, \mu) =& \frac{1}{4\pi^2} \int_{\Re(s)=\mathfrak{s}} (\pi y_1)^{1-s_1} (\pi y_2)^{1-s_2} G^d\paren{s,\mu} \frac{ds}{(2\pi i)^2},
\end{align}
for any $\mathfrak{s} \in (\R^+)^2$.

\begin{thm}
\label{thm:MinWhitt}
	The Whittaker functions at the minimal $K$-types are
	\begin{align*}
		\Lambda^0(\mu) W^0(y, \mu, \psi_I) =& W^{0*}(y, \mu), \\
		\Lambda^1(\mu) W^1_0(y, \mu, \psi_I) =& W^{1*}(y, \mu), \\
		\Lambda^d(\mu^d(r)) W^d_{-d}(y, \mu^d(r), \psi_I) =& W^{d*}(y, \mu^d(r)), \qquad d \ge 2.
	\end{align*}
\end{thm}

For completeness, on $d \ge 2$, we also define
\[ \Lambda^d(-\mu^d(r)) = \Lambda^d(\mu^d(-r)) \frac{\pi^{d-1}}{\Gamma(d)}, \]
so that
\[ \Lambda^d(-\mu^d(r)) W^d_{-d}(y, -\mu^d(r), \psi_I) = W^{d*}(y, \mu^d(-r)), \]
and this derives from $-\mu^d(r)^{w_2} = \mu^d(-r)$.

\subsubsection{The Mellin transform in the general case.}
In \cite[section 3.3]{HWI}, we computed the Mellin transform of $W^d(y,\mu,\psi_I)$ in terms of a type of beta function, defined for $\Re(a),\Re(b)>0$,$\varepsilon\in\set{\pm1}$ by
\begin{align}
\label{eq:BetaFunDef}
	\mathcal{B}_{\varepsilon, m}(a,b) =& \int_0^\infty x^{a-1} (1+x^2)^{-\frac{b+a}{2}} \paren{\paren{\frac{1+ix}{\sqrt{1+x^2}}}^{-m}+\varepsilon\paren{\frac{1-ix}{\sqrt{1+x^2}}}^{-m}} dx.
\end{align}
This function satisfies
\begin{align}
\label{eq:BetaFunFE}
	\mathcal{B}_{\varepsilon, -m}(a,b) =& \varepsilon \mathcal{B}_{\varepsilon, m}(a,b),
\end{align}
and, for $m \ge 0$, may be computed as
\begin{align}
\label{eq:BexplicitEval}
	\mathcal{B}_{(-1)^\delta, m}(a,b) =& i^{\delta} \sum_{j=0}^{(m-\delta)/2} \binom{m}{2j+\delta} (-1)^j B\paren{\frac{\delta+a}{2}+j, \frac{m-\delta+b}{2}-j}.
\end{align}
If we collect these into a diagonal matrix $\mathcal{B}^d_\varepsilon = \diag(\mathcal{B}_{\varepsilon, -d}, \ldots, \mathcal{B}_{\varepsilon, d})$, then the Mellin transform of $W^d(y,\mu,\psi_I)$ in the form
\begin{align}
\label{eq:WhittMellin}
	W^d(y, \mu, \psi_I) =& \int_{\Re(s)=(\frac{2}{10},\frac{2}{10})} (\pi y_1)^{1-s_1} (\pi y_2)^{1-s_2} \what{W}^d(s,\mu) \frac{ds}{(2\pi i)^2},
\end{align}
is
\begin{align}
\label{eq:WhittMellinEval}
	\what{W}^d(s,\mu) =&2^{-s_1-s_2} (2\pi)^{\mu_1-\mu_3} \sum_{\delta\in\set{0,1}^3} i^{\delta_1-\delta_2+\delta_3-1} \sin\frac{\pi}{2}(\delta_2+s_2+\mu_3) \Gamma\paren{s_2+\mu_3} \\
	& \int_{\Re(t)=\frac{1}{10}} \sin\frac{\pi}{2}(\delta_1+t-\mu_1) \Gamma\paren{t-\mu_1} \sin\frac{\pi}{2}(\delta_3+ s_1-t) \Gamma\paren{s_1-t} \nonumber\\
	& \quad \mathcal{B}^d_{(-1)^{1-\delta_1}}\paren{1+\mu_1-t,t-\mu_2} \WigDMat{d}(w_4) \mathcal{B}^d_{(-1)^{1-\delta_3}}\paren{1-s_1+t, s_1-\mu_3} \nonumber\\
	& \quad \WigDMat{d}(w_3) \mathcal{B}^d_{(-1)^{\delta_2+\delta_3}}\paren{1-s_2-s_1-\mu_3+t, s_2-\mu_3-t} \frac{dt}{2\pi i}, \nonumber
\end{align}
for, say, $\abs{\Re(\mu_i)} < \frac{1}{10}$.
(This corrects the parity of $\delta_3$ and the power of $i$ in \cite[(3.19)]{HWI}.)
This was obtained by applying \cite[(3.17)]{HWI}
\begin{align}
\label{eq:PsiThetaInvMellin}
	\e{x} &= \sum_{\delta\in\set{0,1}} (i\sgn(x))^{1-\delta} \lim_{\theta\to\frac{\pi}{2}^-} \frac{1}{2\pi i} \int_{\Re(t) = c} \abs{2\pi x}^{-t} \Gamma\paren{t} \sin(\tfrac{\pi}{2}\delta +\theta t)\, dt,& x\ne0,c>0
\end{align}
to the exponential terms in definition of the Whittaker function, after some useful substititions (see \cite[(3.12)]{HWI}).

If instead we start with \cite[(3.22)]{HWI}, which we write in the form
\begin{align}
\label{eq:WhittIntRepn}
	W^{d'}(y,\mu,\psi_I) =& y_1^{1+\frac{\mu_3}{2}} y_2^{1+\mu_3} \int_{\R^2} (1+u_3^2)^{\frac{-1+\frac{3}{2}\mu_3}{2}} (1+u_2^2)^{\frac{-1+\frac{3}{2}\mu_3}{2}} \paren{y_1\frac{\sqrt{1+u_3^2}}{\sqrt{1+u_2^2}}}^{-\frac{\mu_1-\mu_2}{2}} \\
	& \times \mathcal{W}^{d'}\paren{y_1\frac{\sqrt{1+u_3^2}}{\sqrt{1+u_2^2}},\mu_1-\mu_2} \WigDMat{d'}(w_4) \Dtildek{d'}{\frac{1-i u_3}{\sqrt{1+u_3^2}}} \nonumber \\
	& \times \WigDMat{d'}(w_3) \Dtildek{d'}{\frac{1-i u_2}{\sqrt{1+u_2^2}}} \e{-y_1\frac{u_2 u_3}{\sqrt{1+u_2^2}}-y_2 u_2} du_2 \, du_3 \nonumber
\end{align}
(we used $\Dtildek{d}{-1}\WigDMat{d}(w_3)\Dtildek{d}{-i} = \WigDMat{d}(\vpmpm{-+}w_3\vpmpm{--}w_2) = \WigDMat{d}(w_4)$), the Mellin transform can be written as
\begin{align}
\label{eq:WhittMellinEval2}
	\what{W}^{d'}(s,\mu) =& 2^{2-s_1-s_2} (2\pi)^{-2-\frac{3}{2}\mu_3} \sum_{\delta_2,\delta_3\in\set{0,1}} i^{\delta_3-\delta_2} \sin\frac{\pi}{2}(\delta_2+s_2+\mu_3) \Gamma\paren{s_2+\mu_3} \\
	& \times \int_{\Re(t)=\frac{1}{10}} \sin\frac{\pi}{2}(\delta_3+s_1+\tfrac{\mu_3}{2}-t) \Gamma\paren{s_1+\tfrac{\mu_3}{2}-t} \nonumber \\
	& \times (2\pi)^t \what{\mathcal{W}}^{d'}\paren{t,\mu_1-\mu_2} \WigDMat{d'}(w_4) \mathcal{B}^{d'}_{(-1)^{1-\delta_3}}\paren{1-s_1-\tfrac{\mu_3}{2}+t,s_1-\mu_3} \nonumber \\
	& \times \WigDMat{d}(w_3) \mathcal{B}^{d'}_{(-1)^{\delta_2+\delta_3}}\paren{1-s_1-s_2-\tfrac{3}{2}\mu_3+t,s_2-\tfrac{\mu_3}{2}-t} \frac{dt}{2\pi i}. \nonumber
\end{align}
where
\[ \what{\mathcal{W}}^d(s,u) := \int_0^\infty y^{-\frac{\mu_1-\mu_2}{2}} \mathcal{W}^d(y,u) y^{s-1} dy. \]

\subsection{Stade's formula}

Consider the Mellin transform of a product of two Whittaker function of the same minimal weight, at the minimal weight:
For $\mu,\mu'\in\frak{a}^d_0$ and $\Re(t)>0$, define
\[ \Psi^d(\mu,\mu',t) = \int_{Y^+} W^{d*}(y,\mu) \trans{W^{d*}(y,\mu')} (y_1^2 y_2)^t \, dy. \]
This was evaluated in \cite[Theorem 1.1]{Stade02} (see \cite[(4.13)]{ArithKuzI}), \cite[Theorem 2]{HWI} and \cite[Theorem 2]{HWII}:
\begin{align*}
	\Psi^0=&\frac{1}{4\pi^{3t}\Gamma\paren{\frac{3t}{2}}}\prod_{i,j}\Gamma\paren{\tfrac{t+\mu_i+\mu'_j}{2}}, \\
	\Psi^1=&\frac{1}{2\pi^{3t}\Gamma\paren{\frac{3t}{2}}}\prod_{i,j}\Gamma\paren{\tfrac{c_{i,j}+t+\mu_i+\mu'_j}{2}}, \\
	\Psi^d=& 2^{4-d-4t-r-r'} \pi^{2-3t} \Gamma\paren{t+r+r'} \Gamma\paren{\tfrac{d-1}{2}+t+r-2r'} \Gamma\paren{\tfrac{d-1}{2}+t+r'-2r} \\
	& \times \Gamma(d-1+t+r+r') \Gamma\paren{\tfrac{t}{2}-r-r'} / \Gamma\paren{\tfrac{3t}{2}},
\end{align*}
where
\[ c_{i,j}=\piecewise{1&\If i=3\ne j \text{ or } j=3\ne i,\\0&\Otherwise,} \]
and $\mu=\mu^d(r)$, $\mu'=\mu^d(r')$ for $d \ge 2$.

In terms of the incomplete Whittaker function, this is
\begin{align}
\label{eq:tildePsid}
	\wtilde{\Psi}^d(\mu,\mu',t) :=& \int_{Y^+} \Tr\paren{\Sigma^d_d W^d(y,\mu,\psi_I) \trans{\wbar{W^d(y,-\wbar{\mu'},\psi_I)}}} (y_1^2 y_2)^t \,dy \\
	=& \frac{1}{\Lambda^d(\mu) \Lambda^d(-\mu')} \times \piecewise{\Psi^d(\mu,-\mu',t) & \If d=0,1,\\ \frac{1}{2} \Psi^d(\mu,-\mu'^{w_2},t) & \If d \ge 2.} \nonumber
\end{align}
The $\frac{1}{2}$ here and the $2$ in \eqref{eq:sinmustar} is precisely due to the fact that
\begin{align}
\label{eq:StadesOrtho}
	\Tr\paren{\Sigma^d_d W^d(y,\mu,\psi_I) \trans{\wbar{W^d(y,-\wbar{\mu'},\psi_I)}}} = \tfrac{1}{2} W^d_{-d}(y,\mu,\psi_I) \trans{\wbar{W^d_{-d}(y,-\wbar{\mu'},\psi_I)}},
\end{align}
for $d \ge 2$ and $\mu=\mu^d(r)$, $\mu'=\mu^d(r')$.

\subsection{The uniform spectral expansion}
\label{sect:UniformSE}
The papers \cite{HWI} and \cite{HWII} made explicit the continuous, residual parts and cuspidal parts of the Langlands spectral expansion for $L^2(\Gamma\backslash G)$, and we now recall that construction.

A vector- or matrix-valued Hecke-Maass form $\Phi^{d'}$ of weight $d'$ for $SL(3,\Z)$ is an eigenfunction of the Casimir operators for $G$ and the Hecke operators for $\Gamma$ which transforms as $\Phi^{d'}(gk) = \Phi^{d'}(g)\WigDMat{d'}(k)$ for $k \in K$.
(One must also include a moderate-growth condition as in \cite[(14)]{HWII}, but this is not relevant to the current discussion.)

To each $\mu$ and $d' \ge 0$, the paper \cite{HWI} associates a matrix-valued Eisenstein series
\[ E^{d'}(g,\mu) = \sum_{\gamma\in U(\Z)\backslash\Gamma} I^{d'}_\mu(\gamma g), \]
with $I^d_\mu(g)$ as in \eqref{eq:IdDef}; these are the lifts of the minimal-parabolic spherical $SL(3,\Z)$ Eisenstein series.

For each integer $d \ge 0$, we take a basis of $GL(2)$ Maass forms:
Let $\mathcal{S}_2^{0*}$ be an orthonormal basis of even, spherical Hecke-Maass cusp forms for $SL(2,\Z)$, to which we append the constant function $\sqrt{\tfrac{6}{\pi}}$.
Let $\mathcal{S}_2^{1*}$ be an orthonormal basis of odd, spherical Hecke-Maass cusp forms for $SL(2,\Z)$.
Let $\mathcal{S}_2^{d*}$ for $d \ge 2$ be an orthonormal basis of holomorphic Hecke modular forms of weight $d$; in particular, $\mathcal{S}_2^{d*}$ is empty for odd $d > 2$.
To each $d,d'\ge 0$, $r\in\C$ and $\Phi \in \mathcal{S}_2^{d*}$, the paper \cite{HWI} associates a matrix-valued maximal-parabolic Eisenstein series $E^{d'}(\cdot, \Phi, r)$.
We also associate to such data an Eisenstein series $E^{d'}(\cdot, \wtilde{\Phi}, r)$, where if $\Phi$ has $SL(2)$ Langlands spectral parameters $(\mu_1,-\mu_1)$, then $\wtilde{\Phi}$ has spectral parameters $(-\wbar{\mu_1},\wbar{\mu_1})$; we discuss this further in \cref{sect:Normalizations}.
The residual spectrum is spanned by the Eisenstein series with $\Phi$ the constant function.
The construction of these functions is sufficiently complicated that we exclude it, see \cite[section 5]{HWI}.

For each $d \ge 0$, let $\mathcal{S}_3^{d*}$ be an orthonormal basis of Hecke-Maass cusp forms for $SL(3,\Z)$ of minimal weight $d$.
To each $d,d' \ge 0$ and $\Phi\in\mathcal{S}_3^{d*}$, the paper \cite{HWII} associates a pair of matrix-valued cusp forms $\Phi^{d'}$ and $\wtilde{\Phi}^{d'}$:
\begin{align}
\label{eq:CuspFourierExpand}
	\Phi^{d'}(g) = \sum_{\gamma\in(U(\Z)V)\backslash SL(2,\Z)} \sum_{v\in V} \sum_{m \in \N^2} \frac{\rho_\Phi(m)}{p_\rho(m)} \Sigma^{d'}_d W^{d'}(mv\gamma g,\mu,\psi_I),
\end{align}
where $\gamma$ is embedded in the upper left copy of $SL(2,\Z)$ in $\Gamma$, and $\rho_\Phi(m)$ are the Fourier-Whittaker coefficients of $\Phi$ (it can be shown that the Fourier-Whittaker coefficients of $\Phi$ and the Fourier-Whittaker coefficients of $\Phi^{d'}$ below agree, so there is no ambiguity in the terminology/notation).
The form $\wtilde{\Phi}^{d'}$ is defined identically, but with $\mu$ replaced with $-\wbar{\mu}$.
This corrects a mistake in \cite[(27)]{HWII} which uses $W^{d'}(\gamma vg,\mu,\psi_m)$; if one used that definition then for, say, $d \ge 2$, the $\mu \mapsto \mu^{w_2}$ functional equation relating $\wtilde{\Phi}^{d'}$ and $\Phi^{d'}$ would actually introduce a twist by $p_{\mu^{w_2 w_l}-\mu^{w_l}}(m)$.

From this basis we have \cite[Theorem 1.1]{HWI} and \cite[Theorem 6]{HWII}, the uniform spectral expansion
\begin{thm}
\label{thm:UniformSpectralExpand}
	For $f:\Gamma\backslash G \to \C$ smooth and compactly supported, we have $f=f_c+f_0+f_1+f_2$ where
	\begin{align*}
		f_c(g) = \sum_{d=0}^\infty \sum_{\Phi\in\mathcal{S}_3^{d*}} \sum_{d'=0}^\infty (2d'+1) \Tr\paren{\Phi^{d'}(g)\int_{\Gamma\backslash G} f(g')\wbar{\trans{\wtilde{\Phi}^{d'}(g')}}dg'}.
	\end{align*}
	\begin{align*}
		f_0(g) =& \frac{1}{24} \sum_{d'=0}^\infty \frac{(2d'+1)}{(2\pi i)^2} \int_{\Re(\mu) = 0} \Tr\Bigl(E^{d'}(g, \mu) \int_{\Gamma\backslash G} f(g') \wbar{\trans{E^{d'}(g', \mu)}} dg' \Bigr) \, d\mu,
	\end{align*}
	\begin{align*}
		f_1(g) =& \sum_{d=0}^\infty \sum_{\Phi\in\mathcal{S}_2^{d*}} \sum_{d'=0}^\infty \frac{(2d'+1)}{2\pi i} \int_{\Re(r) = 0} \Tr\biggl(E^{d'}(g, \Phi, r) \int_{\Gamma\backslash G} f\paren{g'} \wbar{\trans{E^{d'}(g', \wtilde{\Phi}, r)}} \, dg' \biggr) dr,
	\end{align*}
	\begin{align*}
		f_2 =& \frac{1}{4/\zeta(3)} \int_{\Gamma\backslash G} f(g) dg.
	\end{align*}
\end{thm}

We will abbreviate the uniform spectral expansion by
\begin{align}
\label{eq:UniformSpectralExpand}
	f(g) = \sum_{d=0}^\infty \int_{\mathcal{B}^{d*}} \sum_{d'=0}^\infty (2d'+1) \Tr\paren{\Xi^{d'}(g) \int_{\Gamma\backslash G}f(g') \trans{\wbar{\wtilde{\Xi}^{d'}(g')}}} d\Xi.
\end{align}
We are using $\int_{\mathcal{B}^{d*}} \ldots d\Xi$ purely in a notational sense, but one can instead construct a measure on the sets $\mathcal{B}^{0*}$ where
\begin{align*}
	\mathcal{B}^{d*}\times\set{d'\ge0} =& \set{\Phi^{d'}\setdiv \Phi\in\mathcal{S}_3^{d*}}\cup\set{E^{d'}(\cdot, \Phi, r)\setdiv \Phi\in\mathcal{S}_2^{d*},\Re(r)=0} \\
	& \qquad \cup\set{E^{d'}(\cdot, \mu)\setdiv \Re(\mu)=0},
\end{align*}
unless $d=0$ in which case we append the constant function
\[ \Phi_0^{d'}:=\piecewise{\frac{1}{4/\zeta(3)} & \If d'=0, \\ 0&\Otherwise.} \]
The map $\Xi \mapsto \wtilde{\Xi}$ is not duality, instead it replaces $\mu$ with $-\wbar{\mu}$ in the Fourier expansion; we explain further presently.

The cusp forms are distinguished from the Eisenstein series by the fact that their degenerate Fourier coefficients are all zero:
\begin{align}
\label{eq:CuspCond}
	\int_{U(\Z)\backslash U(\R)} \Xi^{d'}(ug) \psi_n(u) du = 0 \qquad \text{whenever $\Xi\in\mathcal{S}^{d*}$, $n_1 n_2=0$, $n \in \Z^2$.}
\end{align}
We will not need to consider the degenerate Fourier coefficients of the Eisenstein series in this paper.

We define the non-degenerate Fourier-Whittaker coefficients of a matrix-valued Maass form $\Xi^{d'}$, $\Xi \in \mathcal{B}^{d*}$ of weight $d'$ and minimal weight $d$ with Langlands parameters $\mu_\Xi$ (see \cite[Theorem 3]{HWII}) by
\begin{align}
\label{eq:FWcoefDef}
	\int_{U(\Z)\backslash U(\R)} \Xi^{d'}(xyk) \wbar{\psi_m(x)} dx = \frac{\rho_\Xi(m)}{p_\rho(m)} \Sigma^{d'}_d W^{d'}(m yk, \mu_\Xi,\psi_I),
\end{align}
where $0 \ne m_1,m_2\in\Z$.
The construction of the basis of Eisenstein series and cusp forms is such that the coefficient $\rho_\Xi(m)$ is a scalar and independent of $d'$, and the Fourier-Whittaker coefficients of $\wtilde{\Xi}$ are also $\rho_\Xi(m)$, provided we define $\mu_{\wtilde{\Xi}}=-\wbar{\mu_\Xi}$.

\subsubsection{Normalizations}
\label{sect:Normalizations}
The difference between the forms $\Xi$ and $\wtilde{\Xi}$ is fairly minimal.
For $d=0,1$ the cusp forms have spectral parameters with either $\Re(\mu_\Xi)=0$ or $\mu_\Xi=(x+it,-x+it,-2it)$ with $0 < x < \frac{1}{2}$,$t\in\R$, and it is known (for the particular case $SL(3,\Z)$ \cite{Roelcke01}) that the minimal and maximal parabolic Eisenstein series have $\Re(\mu_\Xi)=0$.
For $d\ge 2$, the cusp forms and Eisenstein series have spectral parameters of the form $\mu_\Xi = \mu^d(it)$ for some $t \in \R$.
In the case $\Re(\mu_\Xi)=0$, this leaves $\mu_{\wtilde{\Xi}}=\mu_\Xi$ unaffected, but for the other two types we have $\mu_{\wtilde{\Xi}}=\mu_\Xi^{w_2}$.
As the construction of the basis elements is done via their Fourier-Whittaker expansion, this implies that $\Xi^{d'}$ and $\wtilde{\Xi}^{d'}$ differ at most by the diagonal matrix \eqref{eq:Tdw2}.
In fact, it is possible to complete the Whittaker function with respect to the $\mu \mapsto \mu^{w_2}$ functional equation (on the rows for which $W^{d'}(\cdot,-\wbar{\mu^d(r)},\psi_I)$ is non-zero), but we have not done so for aesthetic reasons.

For the cusp forms $\Phi\in\mathcal{S}_3^{d*}$, we follow the normalization of \cite[(166)]{HWII}:
At the minimal weight, i.e. $d'=d$, the rowspace of $\wtilde{\Phi}^d$ is spanned by a single row, which we denote by
\[ \wtilde{\phi} := \piecewise{\wtilde{\Phi}^d_0 & \If d=0,1,\\ \wtilde{\Phi}^d_{-d} & \If d \ge 2,} \]
and we take $\phi$ to be the corresponding row of $\Phi^d$.
Then $\Phi^d$ should be normalized according to
\[ \int_{\Gamma\backslash G} \phi(g) \wbar{\wtilde{\phi}(g)} dg=\piecewise{1 & \If d=0,1,\\ \frac{1}{2} & \If d \ge 2.} \]
(The factor $\frac{1}{2}$ is because for $d \ge 2$, the row $\wtilde{\phi}=\wtilde{\Phi}^d_{-d}=(-1)^d\wtilde{\Phi}^d_d$ appears twice in the trace.
Note that this matches the squared-norm of the vector $(\frac{\pm1}{2},0,\ldots,0,\frac{1}{2})$ in the rowspace of $\Sigma^d_d$; compare \cite[Theorem 3 and (27)]{HWII}, \eqref{eq:StadesOrtho} and \eqref{eq:WallachsOrtho}, below.)

For the maximal parabolic Eisenstein series on $d \ge 2$, we are assuming the $SL(2,\Z)$ Maass form (coming from a holomorphic modular form) is normalized in a similar manner, i.e.
\[ \int_{SL(2,\Z)\backslash SL(2,\R)} \Phi(g) \wbar{\wtilde{\Phi}(g)} dg=1, \]
where $\Phi$ has spectral parameters $(\frac{d-1}{2},-\frac{d-1}{2})$ and $\wtilde{\Phi}$ has the same Fourier expansion, but using the (incomplete) Whittaker function at spectral parameters $(-\frac{d-1}{2},\frac{d-1}{2})$.
This is a slight renormalization compared to \cite{HWI}; the Eisenstein series used in the continuous part of the spectral expansion in that paper are completed with respect to the $\mu \mapsto \mu^{w_2}$ functional equation.
The difference is limited to the gamma factors of $\Gamma^d_\Phi$ and $\what{\Gamma}^d_\Phi$ in \cite[section 5.3]{HWI}, and in particular, the Fourier-Whittaker coefficients of the Hecke-normalized $\Phi_H$ \cite[(5.15-19)]{HWI} are unaffected.

To avoid the different normalizations, it is generally preferable to work with the Hecke eigenvalues rather than the Fourier coefficients, and this conversion is given in the next section.

\subsection{Fourier coefficients vs. Hecke eigenvalues}
\label{sect:HeckeConvert}
The Hecke operators (see \cite[section 6.4]{Gold01}) are certain arithmetic operators that commute with the action of the Lie algebra (see \cref{sect:LieAlgebra}) and the sums defining the Eisenstein series, so we may assume the elements of our spectral basis are eigenfunctions of every Hecke operator.
Then we use the conversion to Hecke eigenvalues from \cite[section 4.1]{ArithKuzI}, \cite[section 9.2]{WeylI} and \cite[section 11]{WeylII}:
For $\Phi\in\mathcal{S}_3^{d*}$,
\begin{align}
\label{eq:CuspHeckConv}
	\rho_\Phi(n)\wbar{\rho_\Phi(m)} =& \frac{4}{3} \frac{\lambda_\Phi(n)\wbar{\lambda_\Phi(m)}}{L(\AdSq \Phi,1)}.
\end{align}
When $\Xi=E^d(\cdot,\Phi,r)$ for $\Phi \in \mathcal{S}_2^{d*}$, $\Re(r)=0$,
\begin{align}
	\rho_\Xi(n)\wbar{\rho_\Xi(m)} =& 2 \frac{\lambda_\Phi(n,r)\wbar{\lambda_\Phi(m,r)}}{L(\AdSq \Phi,1) \, L(\Phi,1+3r)\,L(\Phi,1-3r)}.
\end{align}
When $\Xi=E^d(\cdot,\mu)$, $\Re(\mu)=0$,
\begin{align}
	\rho_\Xi(n)\wbar{\rho_\Xi(m)} =& 16 \frac{\lambda_E(n,\mu)\wbar{\lambda_E(m,\mu)}}{\prod_{i\ne j} \zeta(1+\mu_i-\mu_j)} 
\end{align}
(Keeping in mind the Fourier-Whittaker coefficients of the cited papers/sections are with respect to the completed $d=0,1$ Whittaker functions, the normalizations in \cref{sect:Normalizations}, and the $d=1$ completions compare as $\Lambda^1(\mu)=\sqrt{2}\Lambda_{(0,1,1)}(\mu)$; \eqref{eq:CuspHeckConv} may be computed directly from
\[ 2\pi^{-3/2}\Gamma(\tfrac{3}{2}) \Psi^d(\mu_\varphi,\wbar{\mu_\varphi},1) L(1,\AdSq \varphi) = \frac{2}{3} \frac{\Lambda^d(\mu_\Phi) \wbar{\Lambda^d(-\wbar{\mu_\Phi})}}{\rho_\Phi(1)\wbar{\rho_\Phi(1)}} \times \piecewise{1&\If d=0,1,\\ 2 & \If d \ge 2,} \]
see \cite[section 9]{WeylI}.)

The Hecke eigenvalues for the Eisenstein series may be collected from \cite[(4.3) and (5.13)]{HWI}:
\begin{align}
\label{eq:MinEisenHecke}
	\lambda_E((p^\alpha,p^\beta),\mu) =& p^{-\mu_3(2\alpha+\beta)} S_{\alpha,\beta}(p^{\mu_3-\mu_1},p^{\mu_3-\mu_2}), \\
\label{eq:MaxEisenHecke}
	\lambda_\Phi((p^\alpha,p^\beta),\Phi,r) =& p^{2r(2\alpha+\beta)}S_{\alpha,\beta}(p^{a_\Phi(p)-3r},p^{b_\Phi(p)-3r}),
\end{align}
using the Satake parameters $\lambda_\Phi(p)=p^{a_\Phi(p)}+p^{b_\Phi(p)}$, $a_\Phi(p)+b_\Phi(p)=0$ for the  $SL(2,\Z)$ Maass form $\Phi$ and the Schur polynomials
\[ S_{\alpha,\beta}(a,b) := \frac{\det\Matrix{1&b^{\alpha+\beta+2}&a^{\alpha+\beta+2}\\1&b^{\alpha+1}&a^{\alpha+1}\\1&1&1}}{\det\Matrix{1&b^2&a^2\\1&b&a\\1&1&1}}. \]
(The $\lambda_E$ and $\lambda_\Phi$ in \cite[(4.3) and (5.13)]{HWI} are not truly the Hecke eigenvalues, even though they satisfy the Hecke relations, as they include factors $p_{\mu^{w_l}}(n)$ and $n_1^{-r} n_2^r$, respectively; we have dropped these factors in \eqref{eq:MinEisenHecke}-\eqref{eq:MaxEisenHecke}.
Further, this corrects the factor $p^{3\mu_1(\alpha-\beta)}$ in \cite[(5.13)]{HWI} to $p^{3\mu_1(\alpha+\beta)}$.)

\subsection{The Lie algebra}
\label{sect:LieAlgebra}
The main work of the paper \cite{HWII} was to determine precisely the action of the Lie algebra of $G$ on smooth functions.
Let $\mathcal{A}^d$ for each $d \ge 0$ be the space of smooth, row-vector valued functions which transform as $f(gk)=f(g)\WigDMat{d}(k)$ for $k \in K$.
(One should also include the moderate growth criterion described in \cref{sect:Gode}.)
Then \cite[(82) or (83)]{HWII} defines an operator $Y^a:\mathcal{A}^d\to\mathcal{A}^{d+a}$ for each $d \ge 0$ (to be understood from context) and $\abs{a} \le 2$.
We will readily apply the $Y^a$ operators to matrix-valued functions by operating on the rows of the matrix.

We will not require very detailed knowledge about these operators, except:
First, they are invariant under left translation by $G$ (by \cite[(83)]{HWII}).
Second, they completely describe the action of the Lie algebra of $G$ on smooth functions by \cite[Proposition 9]{HWII}; that is, the action of any element of the Lie algebra can be written in terms of the components of the $Y^a$ operators and elements of the Lie algebra of $K$.
Lastly, they preserve $V$-characters; that is, the rows of $Y^a \Sigma^d_\chi I^d_\mu$, with $I^d_\mu$ as in \eqref{eq:IdDef}, lie in the rowspace of $\Sigma^{d+a}_\chi I^{d+a}_\mu$.

There are three more relevant operators on the smooth functions of $G$: $\Delta_K$ the Laplacian on $K$, and $\Delta_1$ and $\Delta_2$ the Casimir operators on $G$.
By considering vector-valued functions, we have essentially abstracted over the action of the Lie algebra of $K$, as its elements act on the entries of the vector-valued functions in very predictable ways, see \cite[sections 5.1-3, esp. (33),(35),(49)]{HWII}.
For a function $f \in \mathcal{A}^d$, we have 
\begin{align}
\label{eq:DeltaKAct}
	\Delta_K f = d(d+1) f.
\end{align}
The Langlands parameters $\mu$ parameterize the eigenvalues of the Casimir operators by
\begin{align}
\label{eq:CasimirAct}
	\Delta_1 p_{\rho+\mu} =& \paren{1-\tfrac{\mu_1^2+\mu_2^2+\mu_3^2}{2}}p_{\rho+\mu} & \Delta_2 p_{\rho+\mu} = \mu_1 \mu_2 \mu_3 p_{\rho+\mu},
\end{align}
say $\Delta_i p_{\rho+\mu} = \lambda_i(\mu) p_{\rho+\mu}$, and this extends to the $I^d_\mu$ function as $\Delta_i I^d_\mu = \lambda_i(\mu) I^d_\mu$.

The $K$-Laplacian is bi-$K$ invariant and the Casimir operators are bi-$G$ invariant which implies, among other things, that the matrix-valued Whittaker function $W^d(\cdot,\mu,\psi_I)$ is an eigenfunction of all three with eigenvalues as in \eqref{eq:DeltaKAct} and \eqref{eq:CasimirAct}.

As degree-two elements of the universal enveloping algebra (with real coefficients), the Laplacians $\Delta_1$ and $\Delta_K$ are symmetric with respect to $dg$ and $\Delta_2$ is anti-symmetric.
On spherical Maass cusp forms, it is known that $\Delta_1$ is a positive operator \cite{CHJT}, but this does not necessarily hold for $d \ge 1$, and in fact for $d > 3$, $\lambda_1(\mu^d(it))$ is positive or negative as $3t^2$ is larger or smaller than $\paren{\frac{d-1}{2}}^2-1$.
The symmetric squares of holomorphic modular forms have spectral parameters $\mu^d(0)$ (for even $d$), so their Laplacian eigenvalues are negative for $d \ge 4$, and the Weyl law \cite[Theorem 1]{WeylII} implies that forms with positive eigenvalues exist for all $d$.
(On $GL(2)$, the Laplace eigenvalue corresponding to a holomorphic modular form is also negative, while a spherical Maass form necessarily has a positive eigenvalue.)

So there may exist Maass forms with Laplacian eigenvalue equal to zero in weight $d \ge 3$.
On the other hand, we can see that when $\lambda_1(\mu)$ and $\lambda_2(\mu)$ are both bounded, $\norm{\mu}$ is also bounded, so the Whittaker transform of a nice function has rapid decay in both $\norm{\mu}$ (so also in $d$ when $\mu=\mu^d(r)$) and $d'$:
\begin{lem}
\label{lem:WhittTransSuperPoly}
	Let $f:G\to\C$ satisfying $f(xyk) = \psi_I(x) f(yk)$ be such that $f(yk)$ is smooth and compactly supported on $Y^+K$.
	Then for $n_1,n_2,n_3 \ge 0$,
	\begin{align*}
		&\int_{U(\R)\backslash G} f(g) \trans{\wbar{W^{d'}(g,\mu,\psi_I)}} dg \\
		&= \int_{U(\R)\backslash G} \paren{\paren{\frac{\Delta_1}{\lambda_1(\wbar{\mu})}}^{n_1} \paren{\frac{-\Delta_2}{\lambda_2(\wbar{\mu})}}^{n_2} \paren{\frac{1+\Delta_K}{1+d'(d'+1)}}^{n_3} f(g)} \trans{\wbar{W^{d'}(g,\mu,\psi_I)}} dg,
	\end{align*}
	provided $n_i=0$ when $\lambda_i(\mu)=0$ for $i=1,2$.
\end{lem}
The proof is trivial, but we have collected the statement into a lemma as we will use it frequently.

From the Casimir operators, one may construct the operator
\begin{align}
\label{eq:LambdaX}
	\Lambda_X = 27\Lambda_2^2+4(\Delta_1+X^2-1)(\Delta_1+4X^2-1)^2
\end{align}
for $X \in \C$.
When $\Re(\mu)=0$, we have $\mu=(it_1,it_2,it_3)$, $t_3=-t_1-t_2$ for some $t_1,t_2\in\R$ and
\begin{align}
\label{eq:LambdaXEigen1}
	\Lambda_X p_{\rho+\mu} = \paren{\prod_{i<j} \paren{(t_i-t_j)^2+4X^2}} p_{\rho+\mu},
\end{align}
and when $\mu=(x+it,-x+it,-2it)$,
\begin{align}
\label{eq:LambdaXEigen2}
	\Lambda_X p_{\rho+\mu} = 4(X-x)(X+x)\paren{4t^2+(2X-x)^2}\paren{9t^2+(2X+x)^2}p_{\rho+\mu},
\end{align}
see \cite[section 11.1]{HWII}.
The nullspace of the operator $\Lambda_{\frac{d-1}{2}}$, when applied to the Maass forms of $K$-type $\WigDMat{d}$ is precisely those forms whose minimal $K$-type is $\WigDMat{d}$, and we will apply this operator in a similar manner to show orthogonality of the Bessel functions coming from different $K$-types in \cref{sect:BesselExpand}.

\subsection{The Weyl Law}
\label{sect:WeylLaw}
The Weyl laws are given in \cite{Val01}*{Theorem 1}, \cite{WeylI}*{Theorem 1} and \cite{WeylII}*{Theorem 1}, and we use some simple cases here.
For some $f(d,\mu)$, we consider the convergence of a spectral sum of the form
\begin{align}
\label{eq:WeylLawSum}
	\sum_{d=0}^\infty \int_{\mathcal{B}^{d*}} \rho_\Xi(n)\wbar{\rho_\Xi(m)} f(d,\mu_\Xi) d\Xi.
\end{align}
We will need to know when such a sum converges (absolutely) in general, but we would like to discuss two regions of the spectrum in particular.
The first is the complementary spectrum $\frak{a}^{d,c}_\theta$, $d=0,1$.
The second are the nearly self-dual forms:
We say a Maass form is nearly self-dual if one of its spectral parameters is small in the sense $\min_i \abs{\mu_i}\le \norm{\mu}^{1-\epsilon}$ for some $\epsilon>0$.

Lastly, we will sometimes wish to express the convergence in terms of the eigenvalues $\lambda_i(\mu)$ of the Casimir operators:
For $\mu$ in $\bigcup_d \frak{a}^d_0$ or the complementary spectrum, we have $\lambda_1(\mu) \asymp 1+\norm{\mu}^2$ unless $\mu=\mu^d(it)$ with $t=\frac{d}{2\sqrt{3}}+o(d)$, and in that case $\mu$, we have $\lambda_2(\mu) \asymp 1+\norm{\mu}^3$.

Then the Weyl laws give us the following:
\begin{enumerate}
\item If $f(\cdot,\mu)$ is the characteristic function of $\norm{\mu} \le T$, then \eqref{eq:WeylLawSum} is $\ll_{m,n} T^5$.\\

\item If $f(\cdot,\mu)$ is the characteristic function of $\norm{\mu} \le T$ in the complementary spectrum, then \eqref{eq:WeylLawSum} is $\ll_{m,n} T^{3+\epsilon}$.\\

\item If $f(\cdot,\mu)$ is the characteristic function of $\min_i \abs{\mu_i} \le M < T$ inside $\norm{\mu} \le T$, then \eqref{eq:WeylLawSum} is $\ll_{m,n} T^{4+\epsilon} M$.\\

\item If $f(d,\mu)$ is the characteristic function of $d=d'$, $\abs{r} \le T$ inside $\mu=\mu^{d'}(r)\in\frak{a}^{d'}_0$ for some $d' \ge 3$, then \eqref{eq:WeylLawSum} is $\ll_{m,n} d' (1+T) (d'+T)^{2+\epsilon}$. \\

\item If $f(d,\mu)$ is the characteristic function of $d=2$, $\abs{r} \le T$ inside $\mu=\mu^{d'}(r)\in\frak{a}^{d'}_0$, then only have the upper bound $\ll_{m,n} (1+T)^4$ for \eqref{eq:WeylLawSum} by \cite[Proposition 12]{WeylII}. \\

\item If $f(\cdot,\mu) \ll \norm{\mu}^{\frac{1}{2}-\epsilon} \Min{\abs{\lambda_1(\mu)}^{-3},\abs{\lambda_2(\mu)}^{-2}}$ for some $\epsilon > 0$, then \eqref{eq:WeylLawSum} converges.
\end{enumerate}

\subsection{Wallach's Whittaker expansion}
\label{sect:Wallach}
We need to collect several spectral measures (weights) that appear rather often:
\begin{align}
	\sinmu^0(\mu) :=& \frac{1}{192\pi^5} \prod_{i<j} (\mu_i-\mu_j)\sin\frac{\pi}{2}(\mu_i-\mu_j), \\
	\sinmu^1(\mu) :=& \frac{1}{16\pi^5}(\mu_1-\mu_2)\sin\frac{\pi}{2}(\mu_1-\mu_2) \cos\frac{\pi}{2}(\mu_1-\mu_3) \cos\frac{\pi}{2}(\mu_2-\mu_3), \\
	\sinmu^d(\mu^d(r)) :=& \frac{1}{2^{5-d} \pi^3 i \Gamma(d-1) \Gamma\paren{\tfrac{d-1}{2}-3r} \Gamma\paren{\tfrac{d-1}{2}+3r}}, \qquad d \ge 2,
\end{align}
and lastly, we set
\begin{align}
\label{eq:sinmustar}
	\sinmu^{d*}(\mu) := \Lambda^d(\mu)\Lambda^d(-\mu)\sinmu^d(\mu)\times\piecewise{1&\If d=0,1,\\ 2& \Otherwise.},
\end{align}
for $\mu\in\frak{a}^d$.
Note that
\begin{align}
\label{eq:specmu}
	\sinmu^{d*}(\mu) = \frac{2}{\pi} \specmu^d(\mu) \times \piecewise{1 & \If d=0,1,\\ 4\pi^2 & \Otherwise,}
\end{align}
using the spectral measures (weights) $\specmu^d(\mu)$ given in \cite[(4.1),(4.2)]{ArithKuzI} and \cite[section 3]{WeylII}.

Let $C^\infty_c(U(\R)\backslash G, \psi_I)$ be the space of smooth functions $f:G\to\C$ that satisfy $f(xg) = \psi_I(x)$ for $x\in U(\R)$ such that $f(yk)$ is compactly supported on $Y^+ K \cong U(\R)\backslash G$.
Then \cite[Theorem 15.9.3]{Wallach} implies
\begin{thm}
\label{thm:Wallach}
For $f \in C^\infty_c(U(\R)\backslash G, \psi_I)$, we have
\begin{align*}
	f(g') =& \sum_{d\ge 0} \sum_{d'\ge 0} (2d'+1) \int_{\frak{a}^d_0} \Tr\paren{\Sigma^{d'}_d W^{d'}(g',\mu, \psi_I) \int_{U(\R)\backslash G} f(g) \trans{\wbar{W^{d'}(g,-\wbar{\mu}, \psi_I)}} dg} \sinmu^{d*}(\mu) d\mu.
\end{align*}
\end{thm}

The translation from Wallach's notation is not so easy:
Essentially, there are two isomorphism classes of parabolics; for the minimal parabolic $d=0$ corresponds to the trivial representation of the $1,1,1$ Levy, while $d=1$ corresponds to a non-trivial character (of which there is only one, up to permutation), for the maximal parabolic, $d \ge 2$ corresponds to a character (of that weight) of $SO(2,\R)$ in the $2,1$ Levy.
The $d'$ sum and the matrix elements involved in the trace are summing over a basis of the corresponding induced representation (from the Levy up to $G$).

On the other hand, if one believes the Whittaker functions give a spanning set, then the spectral measure can be derived from contour shifting with Stade's formula \eqref{eq:tildePsid} (see \cite[section 6]{WeylI}, \cite[section 6]{WeylII} and \cite[section 4.1]{ArithKuzI}; the method is due to \cite{GoldKont}) and \cite[Proposition 20]{HWII} with a proof similar to \cref{sect:BesselExpand}, below:
For $d \ge 2$, let $\Sigma^{d'*}_d$ be the matrix obtained from $\Sigma^{d'}_d$ by setting the rows $\Sigma^{d'*}_{d,m}$ with $\abs{m} < d$ to zero, and for $d=0,1$, let $\Sigma^{d'*}_d=\Sigma^{d'}_d$.
Then if $F(\mu)$ is holomorphic and of rapid decay on $\mu \in \frak{a}^d_\delta$ for some $\delta > 0$, and additionally satisfies
\begin{align}
\label{eq:WallachsOrthoAssm}
	\piecewise{\displaystyle F(\mu^w)=F(\mu), \forall w \in \Weyl & \text{ when } d=0, \\ \displaystyle F(\mu^{w_2})=F(\mu) & \text{when } d=1}
\end{align}
($d \ge 2$ does not require any symmetries), we have
\begin{align}
\label{eq:WallachsOrtho}
	&\int_{U(\R)\backslash G} \int_{\frak{a}^d_0} F(\mu) \, v W^{d'}(g,\mu, \psi_I) \sinmu^{d*}(\mu) d\mu \, \trans{\wbar{v' W^{d'}(g,-\wbar{\mu'}, \psi_I)}} dg \\
	&= F(\mu') \paren{v \Sigma^{d'*}_d \trans{\wbar{v'}}}\nonumber
\end{align}
for $\mu'\in \frak{a}^d_0$ and $v,v'$ in the rowspace of $\Sigma^{d'}_d$. 
This implies in particular that
\begin{align}
\label{eq:WhittInv}
	\int_{U(\R)\backslash G} \Tr\paren{\int_{\frak{a}^d_0} F(\mu) \Sigma^{d'}_dW^{d'}(g,\mu, \psi_I) \sinmu^{d*}(\mu) d\mu \, \trans{\wbar{W^{d'}(g,-\wbar{\mu'}, \psi_I)}}} dg = F(\mu'),
\end{align}
since $\Tr\paren{\Sigma^{d*}_d}=1$.

We note that the expansion of \cref{thm:Wallach} converges by \cref{lem:WhittTransSuperPoly} and the bound of \cref{cor:GL3WhittBd}; in particular, Wallach's theorem on $L^2$ actually holds pointwise.
Also, the assumption \eqref{eq:WallachsOrthoAssm} may be dropped if we replace $F(\mu')$ by the appropriate average over the Weyl elements.


\subsection{Godement's spherical functions}
\label{sect:Gode}

We need a theorem of Godement \cite[Theorem 14]{Gode01}, which will require some translation:
\begin{thm}
A quasi-bounded function on $G$ which is invariant under $K$ is proportional to a spherical function of height one if and only if it is an eigenfunction of every differential operator in the left-$K$ invariant subspace of the universal enveloping algebra of $G$.
\end{thm}

First, a spherical function of height one is a function of the form
\begin{align}
\label{eq:SphFunDef}
	h^d_\mu(g) = \int_K \Tr\paren{\trans{\wbar{\WigDMat{d}(k)}} \Sigma^{d*}_d I^d_{\mu}(kg)} dk,
\end{align}
with $I^d_\mu$ as in \eqref{eq:IdDef}, where $\Sigma^{d*}_d$ is the matrix $\Sigma^d_d$ with every row but the first and last set to zero for $d \ge 2$ and $\Sigma^{d*}_d=\Sigma^d_d$ for $d=0,1$.
The term ``height one'' refers to the fact that the entries of $\Sigma^{d*}_d I^d_{\mu}$ are elements of a multiplicity-one $K$-isotypic component of a completely irreducible representation of $G$; note that the row space of $\Sigma^{d*}_d$ is one dimensional, and so in particular,
\begin{align}
\label{eq:SphFunAtI}
	h^d_\mu(I)=\Tr\paren{\Sigma^{d*}_d} = 1.
\end{align}
The phrase ``invariant under $K$'' refers to the fact that $h^d_\mu(k^{-1} g k) = h^d_\mu(g)$ for $k \in K$.

Second, the term ``quasi-bounded'' can be reinterpretted as moderate growth:
For $g\in G$, define $\norm{g}$ to be the Euclidean norm resulting from the standard inclusion $G \subset \R^9$, then we say $f(g)$ has moderate growth if there exists $r \ge 0$ such that
\[ \sup_{g \in G} \abs{f(g)}/\norm{g}^r < \infty. \]

Lastly, Godement considers the universal enveloping algebra in terms of the right-translation invariant differential operators; this can be switched to the left-translation invariant operators by inverting the function argument.
Then on vector- or matrix-valued functions, the left-$G$- and right-$K$-invariant operators are just the operators
\begin{align}
\label{eq:KinvY}
	Y^{a_1} \circ \cdots \circ Y^{a_l}, \qquad \text{ with } a_1+\ldots+a_l=0.
\end{align}
(The only $K$-invariant differential operator coming from the Lie algebra of $K$ is $\Delta_K$, and a vector-valued form is necessarily an eigenfunction of this.)
We can consider such an operator acting on $h^d_\mu$ by its action on the rows of $\Sigma^{d*}_d I^d_{\mu}$; by \cite[Proposition 19]{HWII}, when $\mu \in \frak{a}^d_0$, $h^d_\mu$ is an eigenfunction of every such operator.
(Since the rowspace of $\Sigma^{d*}_d$ is exactly $\mathcal{V}^d$ at the minimal $d$, and this is one-dimensional.)

And so we arrive at our translation:
\begin{thm}
\label{thm:GodeTrans}
Consider a matrix-valued function $F(g)$ which transforms as
\[ F(kgk')=\WigDMat{d}(k)F(g)\WigDMat{d}(k') \]
for $k,k'\in K$ and is an eigenfunction of every operator of the form \eqref{eq:KinvY} with eigenvalues matching $h^d_\mu$ for some $\mu \in \frak{a}^d_0$.
Set $f(g) = \Tr(F(g))$, and suppose $f(g)$ has moderate growth, then
\[ f(g) = f(I) h^d_\mu(g) \]
for all $g \in G$.
\end{thm}

A note about the convergence of the integral defining $h^d_\mu(g)$:
We may reduce to considering $h^0_{\Re(\mu)}(y)$ since the entries of the Wigner $\WigDName$-matrix are bounded by 1.
In \cref{sect:FourSph}, we give an integral representation \eqref{eq:AltSphFun} (which is just a change of variables) as an integral over $U(\R)$, and from this we may reduce to considering $h^0_{\Re(\mu)}(I)$ since for fixed $y$, we have $p_{\rho+\Re(\mu)}(\trans{u} y) \asymp p_{\rho+\Re(\mu)}(\trans{u})$ (see \cref{sect:ExplicitStars}).
On the other hand, $h^0_\mu(I)=1$ clearly converges absolutely.

We do not require the functional equations of the spherical functions, but if we define $h^d_{\mu,\chi}$ by replacing $\Sigma^{d*}_d$ in \eqref{eq:SphFunDef} with an appropriately defined $\Sigma^{d*}_\chi$ (this is easiest to define in terms of the intertwining operators $T^d(w,\mu^w)$), then $h^d_{\mu,\chi} = h^d_{\mu^w,\chi^w}$ for $w \in \Weyl$ and $\chi^w(v)=\chi(wvw^{-1})$, and this can be deduced from the functional equations of the Whittaker function \eqref{eq:WhittFEs} and \eqref{eq:TdSigmadComm}.
(This is because the action of the $Y^a$ operators on $T^d(w,\mu^w) I^d_\mu$ is identical to that on
\[ T^d(w,\mu^w) \Sigma^d_{\chi^w} W^d(\cdot,\mu^w,\psi_I) = \Sigma^d_\chi W^d(\cdot,\mu,\psi_I); \]
there is a slight caveat that the dimension of the rowspace of the latter is slightly smaller, but by \cite[Proposition 16 parts (3),(1) and (2)]{HWII} the vectors that are killed by $W^d(\cdot,\mu,\psi_I)$ do not lie in the rowspace of any $\Sigma^d_\chi$.)

Similarly, one can define the spherical functions away from the minimal weights $d'=d$, but we will only require the minimal-weight spherical functions, and in fact, only at $d=0,1$.

\subsection{Integrals of gamma functions}
We will make use of Barnes' first and second lemmas:
\begin{thm}[Barnes' first lemma, {\cite[sect. 1.7]{Bailey}}]
\label{thm:BarnesFirst}
For $a,b,c,d\in\C$,
\begin{align*}
	\int_{-i\infty}^{+i\infty} \Gamma(a+s) \Gamma(b+s) \Gamma(c-s) \Gamma(d-s) \frac{ds}{2\pi i} &= \frac{\Gamma(a+c) \Gamma(b+c) \Gamma(a+d) \Gamma(b+d)}{\Gamma(a+b+c+d)}.
\end{align*}
\end{thm}
\begin{thm}[Barnes' second lemma, {\cite[sect. 6.2]{Bailey}}]
\label{thm:BarnesSecond}
For $a,b,c,d,e,f\in\C$ with $a+b+c+d+e-f=0$,
\begin{align*}
	& \int_{-i\infty}^{+i\infty} \frac{\Gamma(a+s) \Gamma(b+s) \Gamma(c+s) \Gamma(d-s) \Gamma(e-s)}{\Gamma(f+s)} \frac{ds}{2\pi i} \\
	&= \frac{\Gamma(a+d) \Gamma(b+d) \Gamma(c+d) \Gamma(a+e) \Gamma(b+e) \Gamma(c+e)}{\Gamma(f-a) \Gamma(f-b) \Gamma(f-c)}.
\end{align*}
\end{thm}

\subsection{The $GL(2)$ Bessel functions}
From the classical Bessel functions $J_\nu(x)$ and $K_\nu(x)$, we define
\begin{align*}
	J^+_\nu(x) :=& \frac{\pi}{2} \frac{J_\nu(2x)+J_{-\nu}(2x)}{\cos\frac{\pi}{2}\nu}, & J^-_\nu(x) :=& \frac{\pi}{2} \frac{J_{-\nu}(2x)-J_\nu(2x)}{\sin\frac{\pi}{2}\nu}, & \wtilde{K}_\nu(x) :=& 2 \cos(\tfrac{\pi}{2}\nu) \, K_\nu(2x).
\end{align*}
For $\nu \in \Z$, we will not encounter $J^\varepsilon_\nu(x)$ unless $\varepsilon = (-1)^\nu$, so we \emph{define the opposite case to be zero} to avoid the $Y$-Bessel function.

We have the following bounds (with implied constants absolute)
\begin{equation}
\label{eq:GL2BesselBound}
\begin{aligned}
	J^\pm_\nu(x) \ll& \paren{1+\frac{x}{1+\abs{\nu}}}^{-1/2}, \qquad \text{ for } \Re(\nu) = 0, \text{ or } \nu \in \Z, \\
	\wtilde{K}_\nu(x) \ll& \paren{1+\frac{x}{1+\abs{\nu}}}^{-1/2}, \qquad \text{ for } \Re(\nu) = 0.
\end{aligned}
\end{equation}
(These bounds are both very weak, but simplicity is key, here.)
A stronger form of the $J$-Bessel bounds can be found in \cite[(2.9) and (2.10)]{BFKMM}, while the $K$-Bessel bound follows easily from the uniform asymptotic expansion \cite[7.13.2 (18)-(21)]{Erdelyi2}.

The Bessel functions satisfy recurrence relations \cite[10.6.1-2 and 10.29.1-2]{DLMF}
\begin{align*}
	J_{\pm1+\nu}(x) =& \frac{\nu}{x} J_\nu(x)\mp\frac{d}{dx} J_\nu(x), \\
	K_{\pm1+\nu}(x) =& \pm\frac{\nu}{x} K_\nu(x)-\frac{d}{dx} K_\nu(x),
\end{align*}
and we can see from \cite[(4.13)]{Subconv} (for $x,\abs{\nu}\ge 1$), the power series \cite[10.2.2 and 10.27.4]{DLMF} (for $x<1$), and \cite[10.9.6 and 10.32.8]{DLMF} (for $\abs{\nu}<1$) that
\begin{align*}
	\frac{d}{dx} e^{\frac{\pi}{2}\abs{\nu}} J_\nu(x), \frac{d}{dx} e^{-\frac{\pi}{2}\abs{\nu}} K_\nu(x) \ll& 1+\frac{1+\abs{\nu}}{x}, \qquad \text{ for } \Re(\nu)=0.
\end{align*}
These give us a bound on $\Re(\nu)=\pm1$, to which we apply Phragm\'en-Lindel\"of, and we have
\begin{align}
\label{eq:GL2BadBesselBound}
	J^\pm_\nu(x), \wtilde{K}_\nu(x) \ll& \paren{1+\frac{1+\abs{\nu}}{x}}^{\abs{\Re(\nu)}} \paren{1+\frac{x}{1+\abs{\nu}}}^{-\frac{1-\abs{\Re(\nu)}}{2}}, \qquad \text{ for } \abs{\Re(\nu)} < 1.
\end{align}

\subsection{The $GL(3)$ Bessel functions}
\label{sect:GL3Bessel}
This paper is primarily the story of the $GL(3)$ long-element Bessel functions, and so we require a great deal of information about them.
The term ``Bessel function'' in the higher rank groups can have different meanings for different authors, but here we are referring to the functions originally defined in \cite{SpectralKuz} and later generalized in \cite{WeylI} and \cite{WeylII}.
We note that some notational changes were made in the papers \cite{HWI} and \cite{HWII} (which came after \cite{SpectralKuz}) and these changes are discussed in \cite[section 4.5]{WeylI}.
There are also a large number of corrections to the constants of \cite{SpectralKuz} given in \cite[section 4.1]{ArithKuzI}.
The formulas below are all pulled from either \cite{ArithKuzI} (for $d=0,1$), \cite{WeylI} (for $d=1$), or \cite{WeylII} (for $d \ge 2$).

The $GL(3)$ Bessel functions are instrinsically defined by an integral transform of the Whittaker function:
For any sufficiently nice test function $F(\mu)$ and each $d \ge 0$ and $w \in \Weyl$, we define $K_w^d(g,\mu)$ by
\begin{align*}
	& \int_{\wbar{U}_w(\R)} \int_{\frak{a}^d_0} F(\mu) W^{d*}(gwxg',\mu)\,d\mu \,\wbar{\psi_I(x)} dx \\
	&= \int_{\frak{a}^d_0} F(\mu) K_w^d(g,\mu) W^{d*}(g',\mu)\,d\mu,
\end{align*}
where $\wbar{U}_w=(w^{-1} \trans{U} w)\cap U$ and the $\mu$ integral is simply to bypass some technical difficulties with convergence of the $x$-integral.
(We could, instead, drop the $\mu$ integral and take a sort of Riemann integral for the $x$ integral, but this is only a cosmetic improvement.)
Since the $Y^a$ operators are left-translation invariant, we may replace $W^{d*}$ with any $Y^{a_1} \cdots Y^{a_\ell} W^{d*}$ on both sides, and hence we may instead define $K_w^d(y,\mu)$ using the matrix-valued Whittaker function at any $d' \ge d$
\begin{align}
\label{eq:KdwDef}
	& \int_{\wbar{U}_w(\R)} \int_{\frak{a}^d_0} F(\mu) \Sigma^{d'^*}_d W^{d'}(gwxg', \mu, \psi_I)\,d\mu \,\wbar{\psi_I(x)} dx \\
	&= \int_{\frak{a}^d_0} F(\mu) K_w^d(g,\mu) \Sigma^{d'^*}_d W^{d'}(g', \mu, \psi_I)\,d\mu, \nonumber
\end{align}
and clearly this holds for $d' < d$ as both sides are zero.

Since the Bessel functions for $w \ne w_l$ will not make an appearance here, we now restrict to the case $w=w_l$.
Most of the following has an analog in the general case, but the long-element case is actually simpler to describe.

\subsubsection{Differential equations and asymptotics}
The Bessel functions are uniquely characterized as functions of $G$ satisfying the properties:
\begin{enumerate}
\item \(\displaystyle K^d_{w_l}(xg,\mu) = K^d_{w_l}(g(w_l x w_l),\mu) = \psi_I(x) K^d_{w_l}(g,\mu). \)
\item \(\displaystyle \Delta_i K^d_{w_l}(g,\mu) = \lambda_i(\mu) K^d_{w_l}(g,\mu) \), where $\Delta_1,\Delta_2$ are the $GL(3)$ Casimir operators and $\lambda_i(\mu)$ is defined by $\Delta_i p_{\rho+\mu} = \lambda_i(\mu) p_{\rho+\mu}$.
\item As $y\to 0$, $K^d_{w_l}(y,\mu)$ has the first-term asymptotics implied by \eqref{eq:JwlAsymp}, \eqref{eq:JsignwlDef} and \eqref{eq:K0wlViaJwl}-\eqref{eq:KdwlViaJwl} below.
\end{enumerate}
If we let $J_{w_l}(g,\mu)$ be the power series (Frobenius series) solution to 1 and 2 whose first-term asymptotic as $y\to 0$ is
\begin{align}
\label{eq:JwlAsymp}
	J_{w_l}(y,\mu) \sim p_{\rho+\mu}(y) \frac{(4\pi^2)^{2+\mu_1-\mu_3}}{\Gamma\paren{1+\mu_1-\mu_3} \Gamma\paren{1+\mu_1-\mu_2} \Gamma\paren{1+\mu_2-\mu_3}},
\end{align}
then
\begin{align}
\label{eq:JwlDef}
	J_{w_l}(y,\mu) = \abs{4\pi^2 y_1}^{1-\mu_3} \abs{4\pi^2 y_2}^{1+\mu_1} \sum_{n_1,n_2\ge 0} \frac{\Gamma\paren{n_1+n_2+\mu_1-\mu_3+1} \, (4\pi^2 y_1)^{n_1} (4\pi^2 y_2)^{n_2}}{\prod_{i=1}^3 \Gamma\paren{n_1+\mu_i-\mu_3+1}\Gamma\paren{n_2+\mu_1-\mu_i+1}}.
\end{align}

Set
\begin{align}
\label{eq:JsignwlDef}
	J^{\varepsilon_1,\varepsilon_2}_{w_l}(y,\mu) =& \varepsilon_2 J_{w_l}(y,\mu)+\varepsilon_1 J_{w_l}(y,\mu^{w_4})+\varepsilon_1 \varepsilon_2 J_{w_l}(y,\mu^{w_5}),
\end{align}
then if $\varepsilon=\sgn(y)$ (note the arguments of the $J_w(y,\mu)$ functions are still the \textbf{signed} $y$), we have
\begin{align}
\label{eq:K0wlViaJwl}
	K_{w_l}^0(y,\mu) =& -\frac{1}{16 \pi} \frac{J^{++}_{w_l}(y,\mu)-J^{++}_{w_l}(y,\mu^{w_2})}{\prod_{i<j} \sin \frac{\pi}{2}\paren{\mu_i-\mu_j}}, \\
\label{eq:K1wlViaJwl}
	K^1_{w_l}(y,\mu) =& -\frac{1}{16\pi} \frac{J^{\varepsilon_1,\varepsilon_2}_{w_l}(y,\mu)-J^{\varepsilon_1,\varepsilon_2}_{w_l}(y,\mu^{w_2})}{\cos\frac{\pi}{2}(\mu_1-\mu_3)\cos\frac{\pi}{2}(\mu_2-\mu_3)\sin\frac{\pi}{2}(\mu_1-\mu_2)},
\end{align}
and for $d \ge 2$,
\begin{align}
\label{eq:KdwlViaJwl}
	-4\pi \cos\pi\paren{\tfrac{d}{2}+3r} K_{w_l}^d(y,\mu^d(r)) =& \delta_{\varepsilon_1=-1} (\varepsilon_2)^d J_{w_l}(y,\mu(r)^{w_4}) + \delta_{\varepsilon_2=-1} (-\varepsilon_1)^d J_{w_l}(y,\mu(r)) \\
	&\qquad -\delta_{\varepsilon_1 \varepsilon_2=-1} (-\varepsilon_1)^d J_{w_l}(y,\mu(r)^{w_3}). \nonumber
\end{align}
In particular, when $\sgn(y)=(+,+)$ and $d \ge 2$, we have $K_{w_l}^d(y,\mu^d(r)) = 0$.

\subsubsection{Some relations among the Bessel functions}
We note that independent of sign, we have
\begin{align}
\label{eq:KdwlIota}
	K^d_{w_l}(y, \mu) = K^d_{w_l}(y^\iota, -\mu^{w_2}),
\end{align}
and in particular
\[ K^d_{w_l}(y, \mu^d(r)) = K^d_{w_l}(y^\iota, \mu^d(-r)) \qquad d \ge 2, \]
which follows from $J_{w_l}(y,\mu) = J_{w_l}(y^\iota,-\mu^{w_l})$ and $\mu^d(-r) = -\mu^d(r)^{w_2}$ via \eqref{eq:JsignwlDef} and \eqref{eq:KdwlViaJwl}.

Also, for $d \ge 2$, $K_{w_l}^d(y,\mu^d(r))$ can be obtained by analytic continuation from the $d=0,1$ functions:
That is, for $d \ge 2$,
\begin{align}
\label{eq:KdwlAnCont}
	K^d_{w_l}(y,\mu^d(r)) =& K^\delta(y,(\mu^d(r))^{w_3}), \qquad \set{0,1} \ni \delta \equiv d \pmod{2}.
\end{align}
This follows from the discussion on degeneracy of the $J_{w_l}$ function at $d \ge 2$ in \cite[section 8.1]{WeylII}.
Then one must also compare the matrices $\Sigma^{d'}_d$ which invariably accompany the Whittaker functions (and hence also the Bessel functions, being defined in terms of the Whittaker functions), and for $d \ge 2$, we have
\begin{align}
\label{eq:Sigmadw3FE}
	\trans{\wbar{W^{d'}(g,-\wbar{(\mu^d(r))^{w_3}},\psi_I)}} \Sigma^{d'}_\delta W^{d'}(g,(\mu^d(r))^{w_3},\psi_I) = \trans{\wbar{W^{d'}(g,-\wbar{\mu^d(r)},\psi_I)}} \Sigma^{d'}_d W^{d'}(g,\mu^d(r),\psi_I),
\end{align}
with $\delta$ as before, by \eqref{eq:WhittFEs}, \eqref{eq:TdOrtho} and \eqref{eq:TdSigmadComm}.
Here we have used the fact that the product of Whittaker functions on the right is invariant under $\Sigma^{d'}_{--} \mapsto \Sigma^{d'}_{-+}$ since only the rows $W^d_{m'}(g,\mu^d(r),\psi_I)$ with $m'\le-d$ contribute.
(See the beginning of \cref{sect:WhittFuns}.)

\subsubsection{Mellin-Barnes integrals}
\label{sect:BesselMBIntegrals}
In the development of Kuznetsov-type formulas, having Mellin-Barnes integrals for the kernel functions serves as an intermediate step between the purely algebraic function of the power series and the development of more functional integral representations, e.g. the double-Bessel integrals below.
On the other hand, Mellin-Barnes integrals are often sufficient for milder applications.
We will use these integral representations for both reasons.

Define the (normalized) Mellin transform of the Bessel functions by
\begin{equation}
\label{eq:KwlMellinTrans}
\begin{aligned}
	K^d_{w_l}(y,\mu) =:& \int_{-i\infty}^{+i\infty} \int_{-i\infty}^{+i\infty} \abs{4\pi^2 y_1}^{1-s_1} \abs{4\pi^2 y_2}^{1-s_2} \what{K}^d_{w_l}(s,\sgn(y),\mu) \frac{ds_1}{2\pi i} \frac{ds_2}{2\pi i}, \\
	\what{K}^d_{w_l}(s,v,\mu) =& \frac{1}{(2\pi)^8} \int_{Y^+} K^d_{w_l}(vy,\mu) (4\pi^2 y_1)^{s_1+1} (4\pi^2 y_2)^{s_2+1} dy.
\end{aligned}
\end{equation}
The unbounded portion of the contours in each $\int_{-i\infty}^{i\infty} \ldots ds_i$ must pass to the left of $\Re(s_i)=0$ (and the finite part must pass to the right of the poles of the integrands) to maintain absolute convergence for $\sgn(y) \ne (+,+)$.

Write $\varepsilon = \sgn(v) = \sgn(y)$, then for $d=0$, we have
\begin{align}
\label{eq:Kwl0Sgn}
	\paren{\prod_{i<j} \sin\tfrac{\pi}{2}(\mu_i-\mu_j)} \what{K}^0_{w_l}(s,v,\mu) = \frac{1}{16 \pi^4} G^0_0(2s,2\mu) \sum_{w\in \Weyl_3} S^{\varepsilon_1,\varepsilon_2}(s,\mu),
\end{align}
where $G^0_0(s,\mu)=\wtilde{G}(0,0,0,s,\mu)$ as in \cref{sect:MinWtWhitt} and
\begin{align*}
	S^{++}(s,\mu) :=& \tfrac{1}{3} \prod_{i<j} \sin\pi(\mu_i-\mu_j), \\
	S^{+-}(s,\mu) :=& \sin\pi(\mu_2-\mu_3) \sin\pi(s_1-\mu_1) \sin\pi(s_2+\mu_2) \sin\pi(s_2+\mu_3)/\sin\pi(s_1+s_2), \\
	S^{-+}(s,\mu) :=& \sin\pi(\mu_1-\mu_2) \sin\pi(s_1-\mu_2) \sin\pi(s_1-\mu_2) \sin\pi(s_2+\mu_3)/\sin\pi(s_1+s_2), \\
	S^{--}(s,\mu) :=& \sin\pi(\mu_1-\mu_3) \sin\pi(s_1-\mu_2) \sin\pi(s_2+\mu_2).
\end{align*}
For $d=1$, we have
\begin{align}
\label{eq:Kwl1Sgn}
	\what{K}^1_{w_l}(s,v,\mu) =& \frac{1}{16 \pi^4} G^0_0(2s,2\mu) \frac{\varepsilon_2 S^{\varepsilon_1,\varepsilon_2}(s,\mu)+\varepsilon_1 S^{\varepsilon_1,\varepsilon_2}(s,\mu^{w_4})+\varepsilon_1 \varepsilon_2 S^{\varepsilon_1,\varepsilon_2}(s,\mu^{w_5})}{\sin\tfrac{\pi}{2}(\mu_1-\mu_2) \cos\tfrac{\pi}{2}(\mu_1-\mu_3) \cos\tfrac{\pi}{2}(\mu_2-\mu_3)}.
\end{align}
For $d \ge 2$, it is somewhat more useful to write these as
\begin{align}
\label{eq:KwldMB}
	\what{K}_{w_l}^d(s,v,\mu^d(r)) =& \frac{1}{4\pi^2} (-\varepsilon_1\varepsilon_2)^d B^{\varepsilon_1,\varepsilon_2}_{w_l}\paren{s,\tfrac{2}{3}r} Q(d,s_1-r) Q(d,s_2+r),
\end{align}
where
\begin{align}
\label{eq:QBDef}
	Q(d,s) :=& \frac{\Gamma\paren{\frac{d-1}{2}+s}}{\Gamma\paren{\frac{d+1}{2}-s}}, &
	B^{\varepsilon_1,\varepsilon_2}_{w_l}(s,r) := \piecewise{
	0& \If \varepsilon=(+,+), \\
	B\paren{s_1+3r,1-s_1-s_2}& \If \varepsilon=(+,-), \\
	B\paren{s_2-3r,1-s_1-s_2}& \If \varepsilon=(-,+), \\
	B\paren{s_1+3r,s_2-3r}& \If \varepsilon=(-,-),}
\end{align}
and $B(a,b)$ is the usual beta function.
(Note: We've redefined $B^{\varepsilon_1,\varepsilon_2}_{w_l}(s,r)$ over \cite[(8)]{WeylII} in case $\varepsilon=(-,-)$ to separate the $d$-dependent factor $(-1)^d$.)

\subsubsection{Stirling's formula on the Mellin transform}
\label{sect:BesselStirling}

A noticable proportion of the argument will center on the distinction between the behavior of $\what{K}^d_{w_l}(s,v,\mu)$ as $\Im(s)$ becomes large for fixed $\mu$ and the behavior as $\norm{\mu}$ becomes large (along the spectrum $\bigcup_d \frak{a}^d_0 \cup \frak{a}^{0,c}_\theta \cup \frak{a}^{1,c}_\theta$) for $s$ essentially fixed.

A good discussion of the poles and residues of $\what{K}_{w_l}^d(s,v,\mu)$ for $\sgn(v)=(+,+)$ can be found in \cite[section 7]{ArithKuzI}, and we collect a few features here.
For $d=0,1$, the Bessel functions are just different linear combinations of the same power series, so the Mellin transforms have poles in the same places; that is, $\what{K}_{w_l}^d(s,v,\mu)$ has poles at $s=(\mu^w_3-\ell_1,-\mu^w_1-\ell_2)$ for $0 \le \ell_1,\ell_2 \in \Z$ and $w \in \Weyl$.
For $d\ge 2$, $\what{K}_{w_l}^d(s,v,\mu)$ has possible poles at
\[ s=(-\tfrac{d-1}{2}+r-\ell_1,2r-\ell_2), (-2r-\ell_1,-\tfrac{d-1}{2}-r-\ell_2), (-\tfrac{d-1}{2}+r-\ell_1,-\tfrac{d-1}{2}-r-\ell_2). \]

At positive distance from the poles and zeros, and assuming $\Re(s),\Re(\mu)$ in some fixed, compact set, for $d=0,1$, Stirling's formula gives
\begin{align}
	\what{K}_{w_l}^d(s,v,\mu) \ll& \abs{s_1+s_2}^{\frac{1}{2}-\Re(s_1+s_2)} \paren{\prod_{j=1}^3 \abs{s_1-\mu_j}^{\Re(s_1-\mu_j)-\frac{1}{2}} \abs{s_2+\mu_j}^{\Re(s_2+\mu_j)-\frac{1}{2}}} \\
	& \qquad \times \sum_{w\in\Weyl}\exp(-\tfrac{\pi}{2} h^{\varepsilon_1, \varepsilon_2}(\Im(s), \Im(\mu^w))),\nonumber
\end{align}
where $h^{\epsilon_1, \epsilon_2}(u, t)$ is given by
\begin{align*}
h^{\varepsilon_1, \varepsilon_2}(u, t) = & -\varepsilon_2 \abs{t_1-t_2} - \varepsilon_1\varepsilon_2\abs{t_1 - t_3} - \varepsilon_1\abs{t_2 - t_3} - \varepsilon_1\varepsilon_2\abs{u_1+u_2} + \varepsilon_1\varepsilon_2 \abs{u_1 - t_1} \\
& + \varepsilon_1 \abs{u_1 - t_2 }  + \abs{u_1 - t_3} + \abs{u_2 + t_1} + \varepsilon_2 \abs{u_2 + t_2} + \varepsilon_1\varepsilon_2 \abs{u_2 + t_3},
\end{align*}
provided $\Re(s)$ is bounded away from any $\Re(\mu_i,-\mu_j)$.
At positive distance from the poles and zeros, and assuming $\Re(s), \Re(r)$ in some fixed, compact set, for $d \ge 2$ (and $\sgn(v) \ne (+,+)$), Stirling's formula gives
\begin{align}
\label{eq:BesselStirlingdge2}
	\what{K}_{w_l}^d(s,v,\mu) \ll& \abs{d+i\Im(s_1-r)}^{-1+2\Re(s_1-r)} \abs{d+i\Im(s_2+r)}^{-1+2\Re(s_2+r)} \\
	& \times \abs{s_1+2r}^{-\frac{1}{2}+\Re(s_1+2r)} \abs{s_2-2r}^{-\frac{1}{2}+\Re(s_2-2r)} \abs{s_1+s_2}^{\frac{1}{2}-\Re(s_1+s_2)} \nonumber \\
	& \times \exp(-\tfrac{\pi}{2} (-\varepsilon_2\abs{\Im(s_1+2r)}-\varepsilon_1\abs{\Im(s_2-2r)}+\varepsilon_1\varepsilon_2\abs{\Im(s_1+s_2)})) \nonumber,
\end{align}
provided $\Re(s)$ is bounded away from the poles of the gamma functions.

The exponential parts coming from Stirling's formula at worst cancel and at best give exponential decay, i.e. $h^{\varepsilon_1, \varepsilon_2}(u, t) \ge 0$, with a similar statement for $d \ge 2$.
The Mellin transforms at signs $\sgn(v)=(+,+)$ and $\sgn(v)=(-,-)$ have opposing behavior in that the $(+,+)$ case has exponential decay in $\Im(s)$ for $\Im(s)$ large compared to $\norm{\mu}$, and the $(-,-)$ case has exponential decay in $\Im(\mu)$ for $\Im(\mu)$ large compared to $\Im(s)$.
Other than the $(+,+)$ case, all of the signs have an unbounded region in $\Im(s)$ where the exponential parts cancel, and in every case, there is an unbounded region of the spectrum where the exponential parts cancel for fixed $s$; these are the nearly self-dual forms of \cref{sect:WeylLaw}.
In addition, while the $(-,-)$ case only lacks exponential decay near $t=0$ for $\mu=\mu^d(it) \in \frak{a}^d_0$, $d \ge 2$, and the $(+,+)$ case is simply zero there, the mixed-sign cases lack exponential decay for all $\mu \in \frak{a}^d_0$, $d \ge 2$.

\subsubsection{Double-Bessel integrals}
\label{sect:DoubleBessel}

To argue the absolute convergence of several integrals during the construction of the Kuznetsov formula, we will need a little bit better decay rates for the $GL(3)$ Bessel functions than are obvious using the Mellin-Barnes integrals, and the easiest way to see this is through the double-Bessel integrals.
These are given in \cite[Lemma 5]{Subconv} for $d=0$ and \cite[(3.6)]{GPSSubconv} for $d \ge 2$ (beware the different normalizations in those papers), and give the $d=1$ case here for the first time; the proof in the case $d=1$ is identical to that of $d=0$.

The double-Bessel integrals for $GL(3)$ are
\begin{align*}
	\mathcal{J}^\pm_1(y,\mu) :=& \abs{\frac{y_1}{y_2}}^{\frac{\mu_2}{2}} \int_0^\infty J^\pm_{\mu_3-\mu_1}\paren{2\pi \abs{y_1}^{1/2}\sqrt{1+u^2}} J^\pm_{\mu_3-\mu_1}\paren{2\pi \abs{y_2}^{1/2}\sqrt{1+u^{-2}}} u^{3\mu_2} \frac{du}{u}, \\
	\mathcal{J}^\pm_2(y,\mu) :=& \abs{\frac{y_1}{y_2}}^{\frac{\mu_2}{2}} \int_1^\infty J^\pm_{\mu_3-\mu_1}\paren{2\pi \abs{y_1}^{1/2}\sqrt{u^2-1}} J^\pm_{\mu_3-\mu_1}\paren{2\pi \abs{y_2}^{1/2}\sqrt{1-u^{-2}}} u^{3\mu_2} \frac{du}{u}, \\
	\mathcal{J}_3(y,\mu) :=& \abs{\frac{y_1}{y_2}}^{\frac{\mu_2}{2}} \int_0^\infty \wtilde{K}_{\mu_3-\mu_1}\paren{2\pi \abs{y_1}^{1/2}\sqrt{1+u^2}} J^-_{\mu_3-\mu_1}\paren{2\pi \abs{y_2}^{1/2}\sqrt{1+u^{-2}}} u^{3\mu_2} \frac{du}{u}, \\
	\mathcal{J}_4(y,\mu) :=& \abs{\frac{y_1}{y_2}}^{\frac{\mu_2}{2}} \int_0^1 \wtilde{K}_{\mu_3-\mu_1}\paren{2\pi \abs{y_1}^{1/2}\sqrt{1-u^2}} \wtilde{K}_{\mu_3-\mu_1}\paren{2\pi \abs{y_2}^{1/2}\sqrt{u^{-2}-1}} u^{3\mu_2} \frac{du}{u}, \\
	\mathcal{J}_5(y,\mu) :=& \abs{\frac{y_1}{y_2}}^{\frac{\mu_2}{2}} \int_0^\infty \wtilde{K}_{\mu_3-\mu_1}\paren{2\pi \abs{y_1}^{1/2}\sqrt{1+u^2}} \wtilde{K}_{\mu_3-\mu_1}\paren{2\pi \abs{y_2}^{1/2}\sqrt{1+u^{-2}}} u^{3\mu_2} \frac{du}{u}.
\end{align*}
The bound \eqref{eq:GL2BesselBound} is sufficient to see these integrals converge for $\mu_3-\mu_1\in i\R$ and $\mathcal{J}^\pm_1$, $\mathcal{J}^\pm_2$ converge for $\mu_3-\mu_1\in\Z$.

These then relate back to the $GL(3)$ Bessel functions by:
For $\sgn(y)=(+,+)$, and $d=0,1$,
\begin{align}
\label{eq:DblBsslpp}
	\frac{K^d_{w_l}(y, \mu)}{\abs{y_1 y_2}} =& 16 \frac{\cos\frac{\pi}{2}(\mu_1-\mu_2)\cos\frac{\pi}{2}(\mu_2-\mu_3)}{\cos\frac{\pi}{2}(\mu_1-\mu_3)} \mathcal{J}_5(y,\mu).
\end{align}
For $\sgn(y)=(-,-)$, 
\begin{align}
\label{eq:DblBsslmm}
	\frac{K^d_{w_l}(y, \mu)}{8\abs{y_1 y_2}} =& \piecewise{
	\displaystyle \frac{1}{3} \sum_{w\in \Weyl_3} \paren{2 \mathcal{J}_1^-(y,\mu^w) +\mathcal{J}_1^+(y,\mu^w)} & d=0, \\[20pt]
	\displaystyle -\mathcal{J}_1^+(y,\mu)-\mathcal{J}_1^+(y,\mu^{w_4})+\mathcal{J}_1^+(y,\mu^{w_5}) & d=1, \\[10pt]
	\displaystyle \varepsilon \, \mathcal{J}_1^\varepsilon(y,\mu^d(r)^{w_3}) & d \ge 2, \varepsilon=(-1)^{d-1}.}
\end{align}
For $\sgn(y)=(+,-)$,
\begin{align}
\label{eq:DblBsslpm}
	\frac{K^d_{w_l}(y, \mu)}{8\abs{y_1 y_2}} =& \piecewise{
	\displaystyle \frac{1}{3} \sum_{w\in \Weyl_3} \paren{\mathcal{J}_2^-(y,\mu^w) + \mathcal{J}_3(y,\mu^w) + \mathcal{J}_4(y,\mu^w)} & d=0, \\[20pt]
	\displaystyle \mathcal{J}_2^-(y,\mu^{w_5})-\mathcal{J}_3(y,\mu^{w_5})+\mathcal{J}_4(y,\mu^{w_5}) & d=1, \\[10pt]
	\displaystyle \mathcal{J}_2^\varepsilon(y,\mu^d(r)^{w_3}) & d \ge 2, \varepsilon=(-1)^{d-1}.}
\end{align}
The case $\sgn(y)=(-,+)$ follows from \eqref{eq:KdwlIota}.

Note: Equations \eqref{eq:DblBsslmm} and \eqref{eq:DblBsslpm} correct the formulas \cite[(3.6)]{GPSSubconv}, which are all missing a factor $\frac{1}{x(1-x)}$. 

\subsection{A key Iwasawa decomposition and substitution}
\label{sect:ExplicitStars}

The Iwasawa decomposition of $w_l x$ and $w_l x y$ for $x \in U(\R)$ and $y \in Y^+$ will be used throughout the paper, so we take a moment to write it down explicitly now:
If $x^* y^* k^* = w_l x y$ with $k^*=\tildek{\alpha_1,\alpha_2,\alpha_3}$ (as in \eqref{eq:tildekDef}), then
\begin{equation}
\label{eq:wlxIwa}
\begin{aligned}
	x_1^* =& \frac{x_2 y_1^2+x_1 x_3}{\xi_1},& x_2^* =& \frac{x_1 y_2^2+x_2(x_1 x_2-x_3)}{\xi_2},& x_3^*=&\frac{x_3}{\xi_1}, \\
	y_1^* =& y_1 \frac{\sqrt{\xi_2}}{\xi_1}, & y_2^* =& y_2 \frac{\sqrt{\xi_1}}{\xi_2}, \\
	\alpha_1 =&\frac{-y_2 \sqrt{\xi_1}-i(x_1 y_2^2+x_2(x_1x_2-x_3))}{\sqrt{y_2^2+x_2^2}\sqrt{\xi_2}}, & \alpha_2 =& \frac{-x_3+iy_1 \sqrt{y_2^2+x_2^2}}{\sqrt{\xi_1}}, & \alpha_3 =& \frac{y_2-ix_2}{\sqrt{y_2^2+x_2^2}},
\end{aligned}
\end{equation}
where
\begin{align*}
	\xi_1 =& y_1^2 y_2^2+y_1^2 x_2^2+x_3^2 & \xi_2 =& y_1^2 y_2^2+y_2^2 x_1^2+(x_1 x_2-x_3)^2.
\end{align*}
In particular,
\[ p_{\rho+\mu}(y^*) = p_{\rho+\mu}(y) \xi_1^{-\frac{1+\mu_2-\mu_3}{2}} \xi_2^{-\frac{1+\mu_1-\mu_2}{2}}, \]
and we note that \eqref{eq:kabcDef} and \eqref{eq:WigDPrimary} imply that $\WigDMat{d}(k^*)$ is a linear combination of terms $\alpha_1^{k_1} \alpha_2^{k_2} \alpha_3^{k_3}$ with $k_i \in \Z$, $\abs{k_i} \le d$.

In most cases, the conjugation $xy \mapsto yx$ has easily understood consequences (i.e. in expressions such as $p_{\rho+\mu}(w_l xy) \mapsto p_{\rho+\mu}(w_l x) p_{\rho+\mu}(y^{w_l})$, $\psi_I(x) \mapsto \psi_y(x)$, $\psi_I(x^*) \mapsto \psi_{y^{w_l}}(x^*)$, $k^* \mapsto k^*$, etc.), and this effectively replaces $y \mapsto I$ in the above expressions.
Note also that the Iwasawa decomposition of $\trans{x}$ is identical to that of $w_l x w_l$ with $x_1$ and $x_2$ interchanged.

If we take $y=I$ and perform the sequence of subsitutions (see \cite[section 5]{Me01}; the substitution can also be explained by factoring the long element in the Weyl group)
\begin{align}
\label{eq:wlxSub}
	x_1 \mapsto \frac{x_1\sqrt{1+x_2^2+x_3^2}+x_2 x_3}{1+x_2^2} \quad \text{ followed by } \quad  x_3 \mapsto x_3\sqrt{1+x_2^2},
\end{align}
which has Jacobian
\[ dx \mapsto \sqrt{1+x_3^2} \,dx, \]
the equations \eqref{eq:wlxIwa} become
\begin{equation}
\label{eq:wlxIwaSubd}
\begin{aligned}
	x_1^* =& \frac{x_2\sqrt{1+x_3^2}+x_1 x_3}{(1+x_2^2)\sqrt{1+x_3^2}},& x_2^* =& \frac{x_1\sqrt{1+x_2^2}}{(1+x_1^2)\sqrt{1+x_3^2}},& x_3^*=&\frac{x_3}{\sqrt{1+x_2^2}(1+x_3^2)}, \\
	y_1^* =& \frac{\sqrt{1+x_1^2}}{(1+x_2^2)\sqrt{1+x_3^2}}, & y_2^* =& \frac{\sqrt{1+x_2^2}}{(1+x_1^2)\sqrt{1+x_3^2}}, \\
	\alpha_1 =&\frac{-1-ix_1}{\sqrt{1+x_1^2}}, & \alpha_2 =& \frac{-x_3+i}{\sqrt{1+x_3^2}}, & \alpha_3 =& \frac{1-ix_2}{\sqrt{1+x_2^2}}.
\end{aligned}
\end{equation}

\section{Some analytic preliminaries}
\label{sect:AnalyticPrelims}

\subsection{The trivial bound on the Whittaker functions}

First, a lemma on the absolute convergence of the Jacquet integral for the Whittaker function:
\begin{lem}
\label{lem:JacAbsConv}
	Suppose $t\in\R^3$, $t_1+t_2+t_3=0$ with $t_1-t_2,t_2-t_3 \ge \frac{1}{T} > 0$, then
	\[ \int_{U(\R)} p_{\rho+t}(w_l x) dx \ll 1+T^3. \]
\end{lem}
\begin{proof}
By the computations of \cref{sect:ExplicitStars} and \cite[3.251.2]{GradRyzh}, the integral is
\begin{align*}
	\int_{\R^3} (1+x_1^2)^{\frac{-1+t_2-t_1}{2}} (1+x_2^2)^{\frac{-1+t_3-t_2}{2}} (1+x_3^2)^{\frac{-1+t_3-t_1}{2}} dx =& \pi^{3/2} \frac{\Gamma\paren{\frac{t_1-t_2}{2}}\Gamma\paren{\frac{t_2-t_3}{2}}\Gamma\paren{\frac{t_1-t_3}{2}}}{\Gamma\paren{\frac{1+t_1-t_2}{2}}\Gamma\paren{\frac{1+t_2-t_3}{2}}\Gamma\paren{\frac{1+t_1-t_3}{2}}},
\end{align*}
and the gamma functions in the numerator each have a simple pole at $t=0$.
\end{proof}

For the most complex computations, we will restrict to the principal series cases $d=0,1$ where we may assume $\Re(\mu)=0$, and obtain the results for generalized principal series where $\mu=\mu^d(r)$, $\Re(r)=0$ by analytic continuation.
This requires a bound which is polynomial in $d'$, but unfortunately, it is not possible to make such a bound also polynomial in $\mu$:
\begin{lem}
\label{lem:AnContWhittBd}
	Let $d' \ge 0$, $y\in Y^+$, $k\in K$ and suppose $\Re(\mu_1-\mu_2),\Re(\mu_2-\mu_3) \ge \frac{1}{T} > 0$, then
	\begin{align}
	\label{eq:AnContOneWhittBd}
		W^{d'}_{m',m}(yk,\mu,\psi_I) \ll (1+T^3) y_1^{1-\Re(\mu_1)} y_2^{1+\Re(\mu_3)}.
	\end{align}
	If also $d \ge 0$, $y'\in Y^+$ and $\Re(\mu_1-\mu_2),\Re(\mu_2-\mu_3) \le A$, then
	\begin{equation}
	\label{eq:AnContTwoWhittBd}
	\begin{aligned}
		& \Tr\paren{\Sigma^{d'}_d W^{d'}(y',\mu,\psi_I) \trans{\wbar{W^{d'}(yk,-\wbar{\mu},\psi_I)}}} \\
		& \ll (1+T^3)^2 (1+d')^4 (1+d'+\norm{\mu})^{4A} (y_1 y_1')^{1-\Re(\mu_1)} (y_2 y_2')^{1+\Re(\mu_3)},
	\end{aligned}
	\end{equation}
\end{lem}
Here, the factor $y_1^{1-\Re(\mu_1)} y_2^{1+\Re(\mu_3)}$ cannot be improved except when one or more $\mu_i-\mu_j\in\Z$, where some of the terms in the power-series expansion/poles of the Mellin transform disappear; this cancellation is what gives us $\ll_\mu y_1 y_2$ for the minimal-weight Whittaker functions, even though they are far from the line of symmetry (i.e. $\Re(\mu)=0$).

\begin{proof}
The first bound follows from \eqref{eq:WhittGFEs} on taking $t=\Re(\mu)$ in \cref{lem:JacAbsConv} with
\[ \abs{I^d_{\mu,m',m}(w_l u k)} \le p_{\rho+t}(w_l u), \]
since $\abs{\WigD{d'}{m'}{m}(\cdot)} \le 1$.
For the second bound, we note that $\Gamma^{d'}_{\mathcal{W}}$ is diagonal and $\Sigma^{d'}_d$ is supported on the diagonal and anti-diagonal, hence the left-hand side may be written
\begin{align*}
	&  \pi^{2(\wbar{\mu_3-\mu_1})} \sum_\pm \sum_{\abs{m},\abs{j_1},\abs{j_2},\abs{j_3} \le d'} \Sigma^{d'}_{d,m,\pm m} W^{d'}_{\pm m, j_1}(y',\mu,\psi_I) \wbar{W^{d'}_{j_2, j_1}(yk,-\wbar{\mu}^{w_l},\psi_I)} \\
	& \qquad \times \wbar{\Gamma^{d'}_{\mathcal{W},j_2,j_2}(\wbar{\mu_2-\mu_3},+1)} \wbar{\WigD{d'}{j_3}{j_2}(\vpmpm{--}w_l) \Gamma^{d'}_{\mathcal{W},j_3,j_3}(\wbar{\mu_1-\mu_3},+1) \WigD{d'}{m}{j_3}(\vpmpm{--}w_l)} \\
	& \qquad \times \wbar{\Gamma^{d'}_{\mathcal{W},m,m}(\wbar{\mu_1-\mu_2},+1)},
\end{align*}
where the matrices at the end derive from applying the $-\wbar{\mu} \mapsto -\wbar{\mu}^{w_l}$ functional equation \eqref{eq:WhittFEs} (see \eqref{eq:Tdwl}) of the second Whittaker function.
Then the entries of $\Gamma^{d'}_\mathcal{W}(u,+1)$ are $\ll (1+d'+\abs{u})^{\Re(u)}$ by Stirling's formula applied to \eqref{eq:WhittGammas} away from the singularities.
The singularities of $\Gamma^{d'}_{\mathcal{W},m,m}(u,+1)$ on $\Re(u) > 0$ are removable and occur when $u \to n \in \N\cup\set{0}$ with $0 \ge 1-m+n \equiv 0 \pmod{2}$; at such points we have
\[ \Gamma^{d'}_{\mathcal{W},m,m}(n,+1) = (-1)^m \frac{\Gamma\paren{\frac{1+m+n}{2}}}{\Gamma\paren{\frac{1+m-n}{2}}}, \]
and the same bound applies.
\end{proof}

By Phragm\'en-Lindel\"of and the $\mu \mapsto \mu^{w_l}$ functional equation \eqref{eq:WhittFEs}, the bound \eqref{eq:AnContOneWhittBd} also applies at $\Re(\mu)=0$ with
\begin{align}
\label{eq:PhragLindT}
	T=\log(2+d')\log(y_1+y_1^{-1})\log(y_2+y_2^{-1})\log(2+\norm{\mu}).
\end{align}

Lastly, we want a bound that directly shows the absolute convergence of Wallach's Whittaker expansion; in fact, we need to see the rapid convergence of a rather complicated part \eqref{eq:FdgpDef} of that expansion piece-by-piece.
We need a bound for the left-hand side of \eqref{eq:AnContTwoWhittBd} for $\mu\in\frak{a}^d_{\frac{1}{2}}$ which is polynomial in $d$, $d'$, $\mu$ and $y$.
(In particular, we will avoid the use of the functional equations as in the above, since the entries in the matrix $T^d(w,\mu)$ can potentially become large far away from $\Re(\mu)=0$.)
Furthermore, we will want to see some extra decay near $y_i=0$ in an inverse Whittaker transform at the same level of generality.
Our solution is to bound the Mellin transform of the Whittaker function, with a slight modification for $d \ge 2$:
For $d \ge 2$ and $\mu=\mu^d(r)$, the trace of \eqref{eq:AnContTwoWhittBd} can be written
\begin{align*}
	\frac{1}{2}\sum_{N=0}^{\frac{d'-d}{2}} W^{d'}_{-d-2N}(y',\mu^d(r),\psi_I) \trans{\wbar{W^{d'}_{-d-2N}(yk,-\mu^d(\wbar{r}),\psi_I)}},
\end{align*}
as the other terms are zero.
The two Whittaker functions differ by the $w_2$ functional equation, which is a diagonal matrix:
\begin{align*}
	W^{d'}_{-d-2N}(yk,-\mu^d(\wbar{r}),\psi_I) =& T^d_{-d-2N,-d-2N}(w_2,-\mu^d(\wbar{r})) W^{d'}_{-d-2N}(yk,\mu^d(-\wbar{r}),\psi_I), \\
	T^d_{-d-2N,-d-2N}(w_2,\mu^d(r)^{w_2}) =& \pi^{1-d} \frac{\Gamma\paren{d+N}}{\Gamma\paren{N+1}},
\end{align*}
so for $d \ge 2$, we renormalize by the square-root of this quantity, and bound the Mellin transform of
\begin{align}
\label{eq:WtildeDef}
	\wtilde{W}^{d'}_{-d-2N}(y,\mu^d(r),\psi_I) :=& \pi^{-\frac{d-1}{2}} \sqrt{\frac{\Gamma\paren{d+N}}{\Gamma\paren{N+1}}} W^{d'}_{-d-2N}(y,\mu^d(r),\psi_I).
\end{align}

For convenience, we set
\begin{align*}
	\wtilde{\mathcal{W}}_{d,N}(y) :=& \pi^{-\frac{d-1}{2}} \sqrt{\frac{\Gamma\paren{d+N}}{\Gamma\paren{N+1}}} \paren{\tfrac{y}{4\pi}}^{-\frac{d-1}{2}}\mathcal{W}_{-d-2N}(\tfrac{y}{4\pi},d-1),
\end{align*}
then $\what{\wtilde{W}}^{d'}_{-d-2N}(s,r)$, the Mellin transform of $\wtilde{W}^{d'}_{-d-2N}(y,\mu^d(r),\psi_I)$, is given by the $-d-2N$ row of \eqref{eq:WhittMellinEval2}, with $(2\pi)^t \what{\mathcal{W}}^{d'}_{-d-2N}(t,\mu_1-\mu_2)$ replaced by $2^{-t} \what{\mathcal{W}}_{d,N}(t)$ where
\[ \what{\mathcal{W}}_{d,N}(s) := \int_0^\infty \wtilde{\mathcal{W}}_{d,N}(y) y^{s-1} dy. \]

We now collect some results on the $GL(2)$ Whittaker functions whose proof is somewhat technical and we postpone to the end.
The overarching goal of these lemmas is to show a generic, polynomial upper bound and an exponential decay bound for the Mellin transform of the $GL(2)$ Whittaker function in various regions of its parameters.
As a general rule, we have chosen simplicity over quality for these bounds.

Even though we can give an explicit expression for $\what{\mathcal{W}}_{d,N}(s)$ as a terminating hypergeometric series, it is quite difficult to directly obtain good bounds for this function, so our generic bound for the Mellin transform will follow from
\begin{lem}
\label{lem:GL2WhittBd}
The function $\wtilde{\mathcal{W}}_{d,N}(y)$ is bounded (up to an absolute constant) by
\begin{enumerate}
\item $\displaystyle \paren{(1+N/d)\log(3+((d+N)/y))}^{1/4}$ always,\\

\item $\displaystyle y^{1/2} (N+1)^{3/2} d^{-3/4}$ for $y<\frac{d-1}{e(N+1)}$, and\\

\item $\displaystyle \frac{(N+1)e^{-y/4}}{\sqrt{N! (d-1+N)!}}$ for $y > (d+3+2N)\Max{N,2\log^2(d+3+2N)}$.
\end{enumerate}
\end{lem}
The first part is the most difficult, and we rely on the asymptotics of Olver and Dunster, as they appear in \cite{DLMF}.
The second and final parts are proven directly; one can certainly improve these by applying the asymptotics, but this would require a deeper analysis in certain transition regions. 
(Again, the bounds are not optimal; in particular, the correct upper bound in part one is likely just 1.)

We apply this to the Mellin transform $\what{\mathcal{W}}_{d,N}(s)$ in the region of absolute convergence.
\begin{lem}
\label{lem:GL2WhittMellinBd}
For $\Re(s) \in [-\frac{1}{2}+\delta,\frac{1}{2}]$, $\delta > 0$, the function $\what{\mathcal{W}}_{d,N}(s)$ is bounded (up to a constant depending on $\delta$) by
\begin{enumerate}
\item $\displaystyle d^{1/2} (N+1)^{3/2} \log d$ always, and\\

\item $\displaystyle \exp\paren{-\tfrac{\pi}{8}\abs{\Im(s)}}$ for $\abs{\Im(s)} > \frac{4}{\pi} (5+d+2N)\log^2(5+d+2N)$.
\end{enumerate}
\end{lem}
The first part follows trivially from the previous lemma, so we leave its proof to the reader.
The second part will follow from Stirling's formula applied to the representation of $\what{\mathcal{W}}_{d,N}(s)$ as a terminating hypergeometric series.

The previous two lemmas concerned the $GL(2)$ Whittaker function at certain integer indices; the opposite case is where one of the indices is purely imaginary.
(If we apply \eqref{eq:PsiThetaInvMellin} to the definition \eqref{eq:classWhittDef}, we can see that the Mellin transform of the $GL(2)$ Whittaker function is a linear combination of $\mathcal{B}_{\varepsilon, m}(a,b)$ functions -- this is essentially what we have done to develop \eqref{eq:WhittMellinEval} and \eqref{eq:WhittMellinEval2}, so it is not entirely incorrect to say that $\mathcal{B}_{\varepsilon, m}(a,b)$ is the Mellin transform of the $GL(2)$ Whittaker function.)
\begin{lem}
\label{lem:GL2BBd}\ 

\begin{enumerate}
\item For $\Re(a)>0$ and $\Re(b) > -1$ the function $\mathcal{B}_{(-1)^\delta, m}(a,b)$ is holomorphic in $a$ and $b$ except for a simple pole at $b=0$ with residue $i^m$ when $m \equiv \delta \pmod{2}$.
We may write this as
\[ \res_{b=0} \mathcal{B}^d_\varepsilon(a,b) = \tfrac{1}{2}\paren{\Dtildek{d}{-i}+\varepsilon \Dtildek{d}{i}}. \]

\item If $\Re(a)>0$ and $\Re(b+1),\abs{\Re(b)} \ge \delta > 0$, we have
\[ \mathcal{B}_{\varepsilon, m}(a,b) \ll_\delta \paren{1+\frac{\abs{a}+\abs{m}}{\abs{b}}}^2 B\paren{\tfrac{\Re(a)}{2},\tfrac{\Re(b)}{2}}. \]

\item If $6>\Re(a),\Re(b+1),\abs{\Re(b)} \ge \delta > 0$, and
\[ \abs{\Im(a-b)} > \Max{\tfrac{8}{\pi}(\abs{m}+20)\log^2\paren{\tfrac{8}{\pi}(\abs{m}+20)},4\abs{\Im(a+b)}}, \]
we have
\[ \mathcal{B}_{\varepsilon, m}(a,b) \ll_\delta \exp-\tfrac{\pi}{16}\abs{\Im(a-b)}. \]
\end{enumerate}
\end{lem}
The first part follows trivially from \eqref{eq:BexplicitEval}, so we leave that to the reader.
The second part will follow from a recursion formula in the argument $b$, and the third is essentially a more careful treatment of \cite[(2.30)-(2.33)]{HWI}.

If we examine \eqref{eq:WhittMellinEval} on $\mu\in\frak{a}^0_A$ for some fixed $A$, and ignore the logarithms and constant factors, the first $\mathcal{B}^d_\varepsilon$ function gives exponential decay in $t$ when $\abs{\Im(t)}$ is larger than $T:=\Max{d',\abs{\Im(\mu_1-\mu_2)}, \abs{\Im(\mu_3)}}$.
Then the second $\mathcal{B}^d_\varepsilon$ function gives exponential decay in $s_1$ when $\abs{\Im(s_1)}$ is larger than $T$, and the third gives exponential decay in $s_2$ when $\abs{\Im(s_2)}$ is larger than $T$.
In the case $\mu=\mu^d(r)$ with $\abs{\Re(r)}\le A$, the same applies to \eqref{eq:WhittMellinEval2} with $T=\Max{d',\abs{\Im(r)}}$.
\begin{lem}
\label{lem:GL3WhittMellinBd}\ 

\begin{enumerate}
\item For $d'\ge d=0,1$, on the region $\mu\in\frak{a}^d_{\frac{1}{2}}$, $\abs{\Re(s_1)},\abs{\Re(s_2)} < \frac{1}{2}$, $\what{W}^{d'}(s,\mu)$ is holomorphic in $s$ and $\mu$ except for possible poles at $(s_1,s_2)=(\mu^w_3,-\mu^w_1)$, $w \in \Weyl$.
The poles are simple when the coordinates of $\mu$ are distinct, and up to order three (when $\mu=0$) otherwise.
If $\delta > 0$ is the minimum distance to any pole, we have the bound
\[ \what{W}^{d'}(s,\mu) \ll_\delta \piecewise{\exp-\frac{1}{100}\Max{\abs{\Im(s_1)},\abs{\Im(s_2)}} & \If \Max{\abs{\Im(s_1)},\abs{\Im(s_2)}} > 100 T \log^7 T, \\ {d'}^4 T^5 & \Otherwise,} \]
with $T=\Max{d',\abs{\Im(\mu_1-\mu_2)}, \abs{\Im(\mu_3)}}$.

\item For $d'\ge d\ge2$, on the region $\mu\in\frak{a}^d_{\frac{1}{2}}$, $\abs{\Re(s_1)},\abs{\Re(s_2)} < \frac{1}{2}$, $\what{\wtilde{W}}^{d'}_{-d-2N}(s,r)$ is holomorphic in $s$ and $r$ except for possible simple poles at $s_1=-2r$ and $s_2=2r$.
If $\delta > 0$ is the minimum distance to any pole, we have the bound
\[ \what{\wtilde{W}}^{d'}_{-d-2N}(s,r) \ll_\delta \piecewise{\exp-\frac{1}{100}\Max{\abs{\Im(s_1)},\abs{\Im(s_2)}} & \If \Max{\abs{\Im(s_1)},\abs{\Im(s_2)}} > 100 T \log^7 T, \\ {d'}^7 T^5 & \Otherwise,} \]
with $T=\Max{d',\abs{\Im(r)}}$.

\end{enumerate}
\end{lem}
The residues in $s_1$ or $s_2$ at the simple poles continue to have exponential decay in the other variable and satisfy the same bound; the double (or triple) poles in the rest of the paper may be avoided by analytic continuation (in $\mu$).
Clearly, the numerical values are not best possible, but these follow easily on inspection.

The proof is straight-forward, but the equations would be lengthy, so we omit certain details.
It is always possible to choose a vertical-line $t$ contour for meromorphy in a certain range of $\Re(s)$ and $\Re(\mu)$ as follows:
There are a finite number, say $n$, of vertical lines, at the linear combinations of $\Re(s)$ and $\Re(\mu)$ in \eqref{eq:WhittMellinEval} or \eqref{eq:WhittMellinEval2}, to be avoided, but if we divide up the allowed ranges of $\Re(s_i),\Re(\mu_i)\in(-\frac{1}{2},\frac{1}{2})$ into $4n-1$ overlapping, equal intervals of width $1/(2n)$, then on any combination of intervals, there is some place to put the $t$ contour and on the overlap of two intervals, we may shift the contour between these places.
There are two difficulties; first that we may be shifting past poles in $t$, thereby picking up extra terms from the residues, but again, the residues in $t$ have the same exponential decay and bound in $s$ and $\mu$, and second that one must check that there exists a choice of $t$ contour within the ranges of \cref{lem:GL2WhittMellinBd,lem:GL2BBd} (though it is possible to extend these, with some effort), and these are the details we have omitted.

\begin{cor}
\label{cor:GL3WhittBd}
With $T$ defined as in the lemma, and $d' \ge d \ge 0$, $\mu \in \frak{a}^d_{\frac{1}{2}}$,
\[ W^{d'}(y,\mu,\psi_I) \ll y_1^{1-\Max{\Re(\mu_i)}} y_2^{1+\Min{\Re(\mu_i)}} {d'}^4 T^7, \]
for $d=0,1$, and 
\[ \wtilde{W}^{d'}_{-d-2N}(y,\mu^d(r),\psi_I) \ll y_1^{1+2\Re(r)} y_2^{1-2\Re(r)} {d'}^7 T^7, \]
for $d \ge 2$.
\end{cor}

\subsection{Decay of the double-Bessel integrals}

As mentioned in \cref{sect:DoubleBessel}, we need to see some decay in the Bessel functions near $y_i=\infty$.
\begin{lem}
\label{lem:DoubleBesselBound}
	For $d \ge 0$ and $\mu \in \frak{a}^d_0$, we have
	\[ \frac{K^d_{w_l}(y, \mu)}{\abs{y_1 y_2}} \ll (\log(3+\norm{\mu})+\abs{\log\abs{y_1}}+\abs{\log\abs{y_2}})\prod_{i=1}^2 \paren{1+\frac{\abs{y_i}^{1/2}}{1+\norm{\mu}}}^{-1/2}. \]
\end{lem}
\begin{proof}
We apply the double-Bessel integral representations.
First, note that when $\sgn(y)=(+,+)$, $K^d(y,\mu)$ is symmetric in $\mu$, so we may assume that $\max_{i,j} \abs{\mu_i-\mu_j} = \abs{\mu_1-\mu_3}$ and the trigonometric functions in \eqref{eq:DblBsslpp} are harmless.
(One can achieve better bounds for the Whittaker function through other means, but we include it for uniformity.)
Now the analysis of $\mathcal{J}^\pm_1$, $\mathcal{J}_3$ and $\mathcal{J}_5$ are identical, and consists of simply applying the bound \eqref{eq:GL2BesselBound}, bearing in mind that when $d \ge 2$ and $\mu=\mu^d(r)$ we have $(\mu^{w_3})_2=-2r$, so the factor $u^{3\mu_2}$ is always of modulus 1.
One might worry about the case when, say, $\abs{y_1}$ is large and $\abs{y_2}$ is very small since $u^{\pm 1}$ needs to be large to overcome the 1 in \eqref{eq:GL2BesselBound} for convergence of the head/tail of the integral, but the factor $\frac{1}{u}$ means the integral over the interval
\[ \int_{\Min{\frac{\abs{y_i}^{1/2}}{1+\norm{\mu}}}}^{\Max{\frac{\abs{y_i}^{1/2}}{1+\norm{\mu}}}} \frac{du}{u} \]
will only contribute to the logarithmic factor in the lemma.

For $\mathcal{J}^\pm_2$ and $\mathcal{J}_4$ there is some extra work, and these two are essentially identical, after sending $u \mapsto u^{-1}$ in $\mathcal{J}^\pm_2$.
The difficulty arises when $u$ is close to 1, but this is simply dealt with:
If $1-u < \Min{\frac{1+\norm{\mu}^2}{\abs{y_1}}, \frac{1+\norm{\mu}^2}{\abs{y_2}}, \frac{1}{2}}$, then the integrand is bounded by 1 and the desired bound holds due to the decreasing measure of this set.
On the other hand, if say $\Min{(1+\norm{\mu}^2)/\abs{y_1},\frac{1}{2}} > 1-u > (1+\norm{\mu}^2)/\abs{y_2}$, then the integrand is bounded by $\frac{1+\norm{\mu}^{1/2}}{\paren{\abs{y_2}(1-u)}^{1/4}}$, and we see
\[ \int_0^{\Min{(1+\norm{\mu}^2)/\abs{y_1},\frac{1}{2}}} u^{-1/4} du \ll \paren{1+\frac{\abs{y_1}^{1/2}}{1+\norm{\mu}}}^{-3/2}. \]
\end{proof}

For the complementary spectrum, the bound becomes a bit worse:
\begin{lem}
\label{lem:BadDoubleBesselBound}
	For $d=0,1$ and $\mu=(X+it,-X+it,-2it)$, $0<X \le \frac{1}{2}-\delta$, $\delta > 0$,$t\in\R$, we have
	\[ \frac{K^d_{w_l}(y, \mu)}{\abs{y_1 y_2}} \ll_\delta (1+\abs{t}) \prod_{i=1}^2 (1+\abs{y_i}^{1/2})^{X-\frac{1}{2}}(1+\abs{y_i}^{-1/2})^{2X}. \]
\end{lem}
\begin{proof}
In the double-Bessel integrals $\mathcal{J}^\pm_i$, we may now have $\mu_3-\mu_1=2X,\pm(X+3it)$, depending on the permutation $w\in \Weyl_3$, so we must apply \eqref{eq:GL2BadBesselBound}.
Sacrificing strength for simplicity, we may collect our cases into the bound
\begin{align*}
	\max_{\nu\in\set{2X,X+3it}} \set{\abs{J^\pm_\nu(x)}, \abs{\wtilde{K}_\nu(x)}} \ll& (1+\abs{t})^{\frac{1}{2}}(1+x)^{X-\frac{1}{2}}(1+x^{-1})^{2X}.
\end{align*}

For $\mathcal{J}^\pm_1$, $\mathcal{J}_3$ and $\mathcal{J}_5$, the proof is essentially identical, but for $\mathcal{J}^\pm_2$ and $\mathcal{J}_4$, we need to revisit the region in $u$ near 1.
Suppose $\abs{y_1} \le \abs{y_2}$, then on the region $1-u < \Min{\frac{1}{\abs{y_2}},\frac{1}{2}}$, the integral is bounded by
\[ (1+\abs{t})\int_0^{\Min{\frac{1}{\abs{y_2}},\frac{1}{2}}} \paren{\frac{1}{u^2\abs{y_1 y_2}}}^X du \ll (1+\abs{t}) \abs{y_1 y_2}^{-X} (1+\abs{y_2})^{2X-1}, \]
and on the region $\Min{1/\abs{y_1},\frac{1}{2}} > 1-u > 1/\abs{y_2}$, the integral is bounded by
\[ (1+\abs{t}) \int_0^{\Min{\frac{1}{\abs{y_1}},\frac{1}{2}}} \paren{\frac{1}{u\abs{y_1}}}^X (u\abs{y_2})^{\frac{1}{2}X-\frac{1}{4}} du \ll (1+\abs{t}) \abs{y_1}^{-X} \abs{y_2}^{\frac{1}{2}X-\frac{1}{4}} (1+\abs{y_1})^{\frac{1}{2}X-\frac{3}{4}}. \]
The result follows.
\end{proof}

\subsection{Inequalities for the derivatives of Bessel functions}
\begin{lem}
\label{lem:BesselDervAsymp}
	For $d\ge 0$ and $\mu \in \frak{a}^d_X$, $X \ge 0$ and any $\epsilon > 0$, we have
	\[ (y_1 y_2)^{-1} \paren{y_1 \partial_{y_1}}^{j_1} \paren{y_2 \partial_{y_2}}^{j_2} K^d_{w_l}(y, \mu) \ll_{\mu,j_1,j_2} \paren{\abs{y_1}^{-X-\epsilon} + \abs{y_1}^{\alpha_1+\epsilon}} \paren{\abs{y_2}^{-X-\epsilon} + \abs{y_2}^{\alpha_2+\epsilon}}, \]
	where
	\[ \alpha_1=\tfrac{2}{3}j_1+\tfrac{1}{3} j_2+\tfrac{1}{6}\Max{0,3X-j_1-j_2}, \qquad \alpha_2=\tfrac{1}{3}j_1+\tfrac{2}{3} j_2+\tfrac{1}{6}\Max{0,3X-j_1-j_2}. \]
\end{lem}
\begin{proof}
We apply the Mellin-Barnes integrals of \cref{sect:BesselMBIntegrals} as in \eqref{eq:Kwl0Sgn}-\eqref{eq:KwldMB}, and shift the $s$ contours back to $\Re(s) = (-\alpha_1-\epsilon,-\alpha_2-\epsilon)$; we may assume that there are no poles on these lines, by reducing $\epsilon > 0$ as necessary.
For the shifted integral, after applying Stirling's formula (treating $\mu$ as constant), we are bounding
\[ \abs{y_1}^{\alpha_1+\epsilon} \abs{y_2}^{\alpha_2+\epsilon} \int_{\Re(s) = (-\alpha_1,-\alpha_2)} \abs{s_1}^{j_1-3\alpha_1-\frac{3}{2}-3\epsilon} \abs{s_2}^{j_2-3\alpha_2-\frac{3}{2}-3\epsilon} \abs{s_1+s_2}^{\alpha_1+\alpha_2+\frac{1}{2}+2\epsilon} \abs{ds}, \]
and the integral converges by applying $\abs{s_1+s_2} \ll \abs{s_1}\abs{s_2}$.

For the residue at $s=(\mu_3^w,-\mu_1^w)$, $w\in \Weyl$, we have the bound $\abs{y_1 y_2}^{-X}$, unless we encountered a double or triple pole (if $\mu_i=\mu_j$ or $\mu=0$, respectively), introducing some logarithmic terms, which we fold into the $\abs{y_1 y_2}^{-\epsilon}$ factor.
The other residues have bounds intermediate to the main terms.
Of course, for $d \ge 2$, the first two-variable residues in both $s_1$ and $s_2$ occur with one of $s_1$ or $s_2$ is much farther to the left, but the bound of the lemma (which is weaker in the case $y \to 0$) is sufficient.

For the mixed terms, say for the residue at $s_1=\mu_i$ with the $s_2$ contour at $\Re(s_2) = -\alpha_2-\epsilon$, we are bounding
\[ \abs{y_1}^{-X} \abs{y_2}^{\alpha_2+\epsilon} \int_{\Re(s_2) = -\alpha_2} \abs{s_2}^{j_2+X-2\alpha_2-1-2\epsilon} \abs{ds_2}, \]
and the integral converges.
Again, double or triple poles will introduce logarithms, but $\epsilon > 0$ still assures convergence in the face of logarithms in $s_2$ (from the digamma function; see \cite[(7.4)]{ArithKuzI} for an example in the positive-sign case), and any logarithms in $y$ may be rolled into the $\epsilon$ powers.
\end{proof}

\subsection{Decay of the inverse Whittaker/Bessel transforms}

Inverse transforms involving Whittaker and Bessel functions tend to have slightly more decay near $y_i=0$ than the functions themselves; this is typical of functions given by Mellin-Barnes integrals.
Essentially, we will need to show the convergence of the Bessel transform of the inverse Bessel transform of a function, and the lemma of the previous section combined with the following lemma will demonstrate the necessary convergence.
\begin{lem}
\label{lem:WhittBound}
For $d' \ge d \ge 0$ and some $\frac{1}{2} > \delta' > \delta > 0$, suppose $F(\mu)$ is holomorphic and $\ll (d'+\norm{\mu})^{-100}$ on $\mu \in \frak{a}^d_{\delta'}$, then
\[ \int_{\frak{a}^d_0} F(\mu) K^d_{w_l}(y,\mu) \, d\mu \ll_F \abs{y_1 y_2}^{1+\delta}, \]
for $d=0,1$,
\[ \int_{\frak{a}^d_0} F(\mu) W^{d'}(y,\mu) \, d\mu \ll_F \abs{y_1 y_2} \prod_{i=1}^2\abs{y_i+1/y_i}^{-\delta}, \]
and for $d \ge 2$,
\[ \int_{\frak{a}^d_0} F(\mu) \wtilde{W}^{d'}_{-d-2N}(y,\mu) \, d\mu \ll_F \abs{y_1 y_2} \prod_{i=1}^2\abs{y_i+1/y_i}^{-\delta}. \]
\end{lem}
\begin{proof}
We again apply the Mellin-Barnes integrals of \cref{sect:BesselMBIntegrals} for $K^d_{w_l}(y,\mu)$.
When both $y_1,y_2<1$, we shift the $s$ contours back to $\Re(s)=(-\delta',-\delta')$, and this picks up possible poles at $s_1=\mu_j$ and $s_2=-\mu_k$ for $j \ne k$ (since either $\Gamma(s_1-\mu_j)$ and $\Gamma(s_2+\mu_j)$ are not both in the numerator, or $\Gamma(s_1+s_2)$ is in the denominator).
Though we avoid it here, the results of such a contour shifting operation are written out explicitly for the inverse Whittaker transform at weight 1 in \cite[section 8.1]{WeylI}.

For the residue at, say $s=(\mu_3,-\mu_1)$ and $d=0,1$, we shift the $\mu$ contours to $\Re(\mu)=(\delta,0,-\delta)$.
For the mixed residues, say the residue at $s_1=\mu_3$ with the $s_2$ contour shifted to $\Re(s_2)=-\delta'$ with $d=0,1$, we may safely shift the $\mu$ contours to $\Re(\mu)=(\delta,0,-\delta)$.
When $d \ge 2$, the only poles on $\Re(s)=0$ are at $s_1=\mu_3=-2r$ and $s_2=-\mu_3=2r$, which cannot happen simultaneously, so we only have the shifted contours and the mixed residues; for the mixed residue at, say, $s_1=-2r$ with the $s_2$ contour shifted to $\Re(s_2)=-\delta'$, we may safely shift the $r$ contour to $\Re(r)=\delta/2$.

This produces the correct power (or possibly better) for $y_1$ and $y_2$, and we simply need to know that the $s$ integrals, the mixed residues and the residues all converge and are bounded by $\norm{\mu}^{100}$.
This is, in fact, the case, and certainly one may replace $\norm{\mu}^{100}$ with 1; this type of optimization is the intent of \cite[Lemma 2]{Me01} and \cite[section 9]{WeylII}, but if one believes the convergence (which requires analyzing Stirling's formula), then the polynomial boundedness is obvious (from Stirling's formula), so we leave it there.
What is important, however, is that nothing exponential in $d$ arises, and the argument of \cite[section 9]{WeylII} demonstrates this fact.

In case one or both of the $y_i>1$, we simply stop the shifting at, say $\Re(s_i)=-\delta/100$, which assures convergence, but doesn't really affect the rest of the argument.

For the Whittaker functions, we use \cref{lem:GL3WhittMellinBd} and note that there is no need to shift $\Re(s_i)$ negative when $y_i > 1$, and we can simply leave that contour at $\Re(s_i)=\delta'$ to produce the required decay.
Again, one can certainly do better than $\norm{\mu}^{100}$, but this is not necessary at present.

\end{proof}

\subsection{A stationary phase lemma}

Much like \cite[section 2.6.2-4]{SpectralKuz}, we have a difficult interchange of integrals to perform in \cref{sect:InterchangeOfIntegrals}.
The interchange of integrals in this paper (interchanging $x$ and $y$ outside the $\mu$ integral) is essentially orthogonal to that of \cite{SpectralKuz} (interchanging the $x$ and $\mu$ integrals with $y$ essentially fixed), and they are of comparable difficulty.
This time, we apply the lessons learned in \cite{Subconv} and reduce to a one-dimensional stationary phase argument.
One might argue that this is a bit messier than the proof of \cite[section 2.6.2-4]{SpectralKuz}, but it is far more explicit (try to write out the full integral representation described there, it's quite difficult), and there is some hope (for the author, at least) that the interchange of integrals required to build an arithmetic Kuznetsov-type formula on real reductive groups would also reduce to this same lemma:

\begin{lem}
\label{lem:InterchangeStationaryPhase}
	Let $\eta>0$, $0\ne A\in\R$, $B > 0$, $C>\frac{1}{2}$, $T > 6^{1/\eta}$, and suppose $w(x)$ is smooth and compactly supported on $1+\abs{x} \in [\frac{1}{2}C,2C]$ with $w^{(j)}(x) \ll_j C^{-j}$ then for any $\eta > 0$,
	\begin{align}
	\label{eq:InterchangeLemInt}
		\int_{\R} w(x) \e{Ax+B\frac{x}{1+x^2}} dx \ll_\eta C T^{-\eta},
	\end{align}
	whenever
	\begin{align}
	\label{eq:InterchangeLemAssm1}
		\abs{A} > T^\eta \Max{\frac{1}{C}, \frac{B}{C^2}},
	\end{align}
	or
	\begin{align}
	\label{eq:InterchangeLemAssm2}
		T^{2\eta} C < B < T^{-\eta} C^3 \qquad \text{ and } \qquad C^{1+\delta} \ll B,
	\end{align}
	for some $0< \delta < \frac{1}{10}$.
\end{lem}
\begin{proof}
We only need very weak results here, so we use a simplified version of (the proof of) \cite[Lemma 8.1]{BKY01}.
Let $\phi(x)=Ax+B\frac{x}{1+x^2}$, then
\[ \phi'(x) = A+B\frac{1-x^2}{(1+x^2)^2}, \]
and the higher derivatives are
\[ \phi^{(j)}(x) = B\frac{(-1)^j j! x^{j+1}+\ldots}{(1+x^2)^{j+1}} \ll \frac{B}{C^{j+1}}, \qquad j \ge 2. \]
Whenever $\phi'(x)$ is bounded away from zero on the support of $w(x)$, we may apply integration by parts so that the integral in \eqref{eq:InterchangeLemInt} becomes
\begin{align}
\label{eq:InterchangeLemIntByParts}
	-\frac{1}{2\pi i} \int_{\R} \e{\phi(x)} \paren{\frac{w'(x)}{\phi'(x)}-w(x)\frac{\phi''(x)}{(\phi'(x))^2}} dx \ll& \sup_{1+\abs{x}\in [\frac{1}{2}C,2C]} \paren{\frac{1}{\abs{\phi'(x)}}+\frac{B}{C^2 \abs{\phi'(x)}^2}}.
\end{align}

\noindent\emph{Case I} Suppose \eqref{eq:InterchangeLemAssm1} holds; since $\abs{\frac{1-x^2}{(1+x^2)^2}} < 6 C^{-2}$, this implies $\phi'(x) \asymp \abs{A}$.
Note that \eqref{eq:InterchangeLemAssm1} implies also that $\abs{A} \gg T^\eta \frac{\sqrt{B}}{C^{3/2}}$, and so \eqref{eq:InterchangeLemIntByParts} implies the bound \eqref{eq:InterchangeLemInt}.

\noindent\emph{Case II} Suppose \eqref{eq:InterchangeLemAssm2} holds.
We have $C > T^{3\eta/2} > 14$ (so that $\abs{x} > 6$) and
\[ \frac{1-x^2}{(1+x^2)^2} \in \braces{-\frac{4}{C^2},-\frac{1}{5 C^2}}. \]
If $A > \frac{5B}{C^2}$ or $A < \frac{B}{6C^2}$ (including when $A < 0$), then $\abs{\phi'(x)} \gg \frac{B}{C^2} > \frac{T^\eta}{C}$ and \eqref{eq:InterchangeLemIntByParts} again implies the bound \eqref{eq:InterchangeLemInt}.

Otherwise, we use
\[ B\frac{x}{1+x^2}-\frac{B}{x} = \frac{B}{x} \paren{(1+x^{-2})^{-1}-1} \ll \frac{B}{C^3} < T^{-\eta}, \]
so the integral of \eqref{eq:InterchangeLemInt} is
\[ \int_{\R} w(x) \e{Ax+\frac{B}{x}} dx +O\paren{CT^{-\eta}}, \]
and this has a unique stationary point $\phi'(x)=0$ at $1+\abs{x}=1+\sqrt{B/A} \in [\frac{1}{3}C,3C]$.
Further, since $\abs{x} > 6$, we have $\abs{\phi''(x)} \asymp \frac{B}{C^3}$.

Taking $X=1$, $V=C$, $Y=\frac{B}{C}$, $Q=C$, $V_1 = 3C$ in \cite[Proposition 8.2]{BKY01}, gives the integral is bounded by
\[ \frac{QX}{\sqrt{Y}} = \frac{C^{3/2}}{B^{1/2}} \ll C T^{-\eta}, \]
provided
\[ \Max{1,C,\frac{B}{C}}^{3\delta'} \ll \frac{B}{C}, \]
for some $\delta' > 0$, and this is true by assumption with $\delta'=\delta/3$.
Note: By taking only the first term (the sum starts at $n=0$, so we use $A=\frac{\delta'}{4}<\frac{\delta'}{3}$), the error bound in \cite[Proposition 8.2]{BKY01} is
\[ \Max{1,C,\frac{B}{C}}^{-\delta/12} < 1 < C T^{-\eta}. \]

\end{proof}

\section{The arithmetic Kuznetsov formula}
\label{sect:ArithKuzProof}

Let $u(x)$ be smooth and compactly supported on $U(\R)$ with
\begin{align}
\label{eq:uUnitIntCond}
	\int_{U(\R)} u(x) dx = 1.
\end{align}
For a smooth, compactly supported $f:Y\to\C$ and any $g'\in G$, we extend to $G$ via the Bruhat decomposition by
\begin{align}
\label{eq:FgpuDef}
	F_{g',u}(xyw x' g') := \piecewise{\psi_I(x) f(y) \psi_I(x') u(x') & \If w=w_l,\\0 & \Otherwise,} \qquad x,x'\in U(\R), y\in Y,
\end{align}
and define a Poincar\'e series by
\begin{align}
\label{eq:Pdef}
	P_m(g) =& \sum_{\gamma\in U(\Z)\backslash\Gamma} F_{I,u}(m \gamma g).
\end{align}
(It is sufficient for our purposes to take $g'=I$ in \eqref{eq:FgpuDef} and \eqref{eq:TgpDef}, but the $g'$-independence of \cref{lem:WeirdBesselExpand} below is interesting and not at all obvious.)

Then we have the Bruhat expansion of the Fourier coefficient
\begin{align}
\label{eq:MainBruhat}
	\mathcal{P} :=& \int_{U(\Z)\backslash U(\R)} P_m(x n^{-1}) \wbar{\psi_n(x)} dx \\
	=& \sum_{c\in\N^2} \sum_{v\in V} S_{w_l}(\psi_m,\psi_n,cv) \int_{U(\R)} F_{I,u}(m cvw_l x n^{-1}) \wbar{\psi_n(x)} dx \nonumber \\
	=& \sum_{c\in\N^2} \sum_{v\in V} S_{w_l}(\psi_m,\psi_n,cv) f(m cv n^\iota) p_{-2\rho}(n) \int_{U(\R)} u(x) dx. \nonumber
\end{align}

We also have the spectral expansion of the Fourier coefficient
\begin{align}
\label{eq:PreMainSpectral}
	\mathcal{P} =& \sum_{d=0}^\infty \int_{\mathcal{B}^{d*}} \mathcal{P}^d_\Xi \, d\Xi,
\end{align}
where
\begin{align*}
	\mathcal{P}^d_\Xi :=& \sum_{d'=0}^\infty (2d'+1) \Tr\paren{\paren{\int_{U(\Z)\backslash U(\R)} \Xi^{d'}(un^{-1}) \wbar{\psi_n(u)} du} \paren{\int_{\Gamma\backslash G} P_m(g) \trans{\wbar{\wtilde{\Xi}^{d'}(g')}} dg}}.
\end{align*}
From \eqref{eq:FWcoefDef}, we may write this as
\begin{align}
\label{eq:PdXiEval}
	\mathcal{P}^d_\Xi =& \rho_\Xi(n) \wbar{\rho_\Xi(m)} \frac{p_{\rho}(m)}{p_{\rho}(n)} F_{I,u}(\mu_\Xi),
\end{align}
where
\begin{align}
\label{eq:FdgpDef}
	F^d_{g',u}(\mu) = \sum_{d'\ge 0} (2d'+1) \Tr\paren{\Sigma^{d'}_d W^{d'}(g',\mu, \psi_I) \int_{U(\R)\backslash G} F_{g',u}(g) \trans{\wbar{W^{d'}(g,-\wbar{\mu}, \psi_I)}} dg}.
\end{align}
Note that $F^d_{g',u}(\mu)$ converges (rapidly), has rapid decay in $\mu$ (hence also in $d$), and is holomorphic on $\mu\in\frak{a}^d_{\frac{1}{2}}$ by \cref{lem:WhittTransSuperPoly} and the bound of \cref{cor:GL3WhittBd}.
Further, $F^d_{g',u}(\mu)$ is invariant under $\mu \mapsto \mu^{w_2}$ by the comment immediately preceeding \cref{sect:WhitFEs}, and so we may freely replace $\mu$ with $-\wbar{\mu} \in \set{\mu,\mu^{w_2}}$ in \eqref{eq:FdgpDef} whenever $\mu\in\frak{a}^d_0$ or $\mu=\mu_\Xi$ for some Maass form $\Xi$.
Lastly, $F^d_{g',u}(\mu)$ satisfies \eqref{eq:WallachsOrthoAssm}, since when $d=0$, $\Sigma^{d'}_d=\Sigma^{d'}_{++}$ commutes with the functional equations of the Whittaker functions by \eqref{eq:TdSigmadComm}.

The heart of the construction of the arithmetic Kuznetsov formula is
\begin{lem}
\label{lem:WeirdBesselExpand}
	Independent of the choice of $g'\in G$ and $u$ satisfying \eqref{eq:uUnitIntCond}, we have
	\begin{align*}
		F^d_{g',u}(-\wbar{\mu}) &= \frac{1}{4} \int_Y f(y) \wbar{K^d(y,\mu)} dy.
	\end{align*}
\end{lem}
This will be proved along the way in the construction of the Bessel expansion in \cref{sect:BesselExpand}.

Combining \eqref{eq:MainBruhat}, \eqref{eq:PreMainSpectral}, \eqref{eq:PdXiEval} and \cref{lem:WeirdBesselExpand} gives
\begin{align}
\label{eq:PreArithKuz}
	& \frac{1}{4} \sum_{d=0}^\infty \int_{\mathcal{B}^{d*}} \rho_\Xi(n) \wbar{\rho_\Xi(m)} \int_Y f(y) \wbar{K^d(y,\mu_\Xi)} dy \, d\Xi \\
	&= \sum_{c\in\N^2} \sum_{v\in V} S_{w_l}(\psi_m,\psi_n,cv) \frac{f(m cv n^\iota)}{p_{\rho}(mn)}. \nonumber
\end{align}
For $y=m cv n^\iota$, we have $\abs{y_1 y_2} = \frac{p_\rho(mn)}{c_1 c_2}$, so if we replace $f(y) \mapsto \abs{y_1 y_2} f(y)$, we may also express the arithmetic Kuznetsov formula as
\begin{align}
\label{eq:PreArithKuz2}
	\mathcal{K}_L(f) =& \sum_{d=0}^\infty \int_{\mathcal{B}^{d*}} \rho_\Xi(n) \wbar{\rho_\Xi(m)} F(d,\mu) d\Xi,
\end{align}
with $F(d,\mu)$ and $\mathcal{K}_L(f)$ as in \cref{thm:ArithKuzHecke}, and the theorem follows by applying the conversions to Hecke eigenvalues given in \cref{sect:HeckeConvert}.

\subsection{An note on the method of Baruch and Mao}
\label{sect:BaruchMao}
If one could connect the Bessel-like distributions of \cite{BaruchMao01} to the Bessel distributions of \cite{Baruch01} and the Bessel functions of that paper to the Bessel functions of \cref{sect:GL3Bessel}, then \cref{lem:WeirdBesselExpand} would follow from \cite[(23.6)]{BaruchMao01} and \cite[Theorem 2.3]{Baruch01}.
This seems as difficult as the current method for $GL(3)$, but it is certainly a worthwhile approach as those results hold at the level of reductive groups.
We hope that the method of \cref{sect:BesselExpand} will help explore the deeper question of the hyper-Kloosterman sums; recall the discussion of Problem 4 in \cref{sect:WhatsNext}.

\subsection{An aside on the method of Cogdell and Piatetski-Shapiro}
The definition \eqref{eq:KdwDef} of the Bessel functions can be interpretted formally as
\[ \int_{U(\R)} \Sigma^{d'^*}_d W^{d'}(y w_l x t,\mu,\psi_I) \wbar{\psi_I(x)} dx = K^d_{w_l}(y, \mu) \Sigma^{d'^*}_d W^{d'}(t, \mu,\psi_I) \]
(this can be made precise by considering a type of Riemann integral in $x$), and it follows that
\[ \int_{U(\R)} \Sigma^{d'^*}_d W^{d'}(y w_l x,\mu,\psi_I) \wbar{\psi_t(x)} dx = K^d_{w_l}(y t^\iota, \mu) \Sigma^{d'^*}_d W^{d'}(t, \mu,\psi_I) \frac{1}{(t_1 t_2)^2}. \]
Then by Fourier inversion we have the closest statement to \cite[(4.3)]{CPS}:
\[ \int_{\R} \Sigma^{d'^*}_d W^{d'}(y w_l x,\mu,\psi_I) dx_3 = \int_{\R^2} K^d_{w_l}(y t^\iota, \mu) \Sigma^{d'^*}_d W^{d'}(t, \mu,\psi_I) \frac{dt_1 \, dt_2}{(t_1 t_2)^2}. \]
On the other hand, by Plancherel, if $u(x) = u(x_1,x_2)$, then
\begin{align*}
	&\int_{U(\R)} \Sigma^{d'^*}_d W^{d'}(y w_l x,\mu,\psi_I) u(x) dx \\
	&= \int_{\R^2} \hat{u}(-t_1,-t_2) K^d_{w_l}(y t^\iota, \mu) \Sigma^{d'^*}_d W^{d'}(t, \mu,\psi_I) \frac{dt_1 \, dt_2}{(t_1 t_2)^2},
\end{align*}
and this looks like we want to flatten out the function of $x_3$.
So we take $u(x) = u_{12}(x_1,x_2) u_3\paren{\frac{x_3}{X}}$, and for any fixed $d$, we have
\begin{align*}
	&\lim_{X\to\infty} \int_{U(\R)} \Sigma^{d'^*}_d W^{d'}(y w_l x,\mu,\psi_I) u(x) dx \\
	&= u_3(0) \int_{\R^2} \what{u_{12}}(-t_1,-t_2) K^d_{w_l}(y t^\iota, \mu) \Sigma^{d'^*}_d W^{d'}(t, \mu,\psi_I) \frac{dt_1 \, dt_2}{(t_1 t_2)^2},
\end{align*}
but on the Bruhat side, we have
\[ \int_{U(\R)} u(x) \wbar{\psi_I(x)} dx = X \int_{\R^3} u_{12}(x_1,x_2) u_3(x_3) dx. \]
Since the spectral expansion over $d$ and $\mathcal{B}^{d*}$ converges rapidly independent of $X$ (by the usually Fourier analysis argument using the Casimir operators on $G$), we conclude that the effective length of the $d'$ sum must be tending to infinity to recover the $X$.
As the $d'$ sum is rather difficult to handle, even in the asymptotic $y \to 0$, we detour from the method of Cogdell and Piatetski-Shapiro and instead prove \cref{lem:WeirdBesselExpand} by applying the methods of Wallach.
We could replace the Iwasawa decomposition with the Bruhat decomposition in our model of the irreducible representations of $G$, effectively replacing the $d'$ sum (over representations of $K$) with a sum (and integral) over elements of the infinite-dimensional representation of $U(\R)$, but this will encounter precisely the same difficulty.
(That is to say, the naive approach to the Kirillov model fails to be useful; i.e. the fundamental difficulty in applying the method of Cogdell and Piatetski-Shapiro is that the characters of $U(\R)$ no longer span the $L^2$-space.)

\section{The Bessel expansion}
\label{sect:BesselExpand}

For each fixed $g'$, we consider the operator
\begin{align}
\label{eq:TgpDef}
	T_{g'}(f)(g) = \int_{U(\R)} f(g w_l x g') \wbar{\psi_I(x)} dx,
\end{align}
whenever $f:G\to\C$ is nice enough for convergence.
For any smooth, compactly supported $f:Y\to\C$, we choose some smooth, compactly supported $u:U(\R)\to\C$ with the property \eqref{eq:uUnitIntCond}, then the function $F_{g',u}$ in \eqref{eq:FgpuDef} is constructed so that
\[ f(y) = T_{g'}(F_{g',u})(y). \]
Alternately, we may apply $T_{g'}$ to Wallach's expansion of $F_{g',u}$, using the fact that the set of elements not of the form $xyw_l x'$ have measure zero in $G$.
Then we pull the $T_{g'}$ operator inside the $d$ and $d'$ sums (justified by \cref{lem:WhittTransSuperPoly,lem:WhittBound}, and \cref{cor:GL3WhittBd}) and use the definition \eqref{eq:KdwDef} of the Bessel functions so that
\begin{align*}
	f(y) =& \sum_{d\ge 0} \sum_{d'\ge 0} (2d'+1) \int_{\frak{a}^d_0} \Tr\paren{\Sigma^{d'}_d W^{d'}(g',\mu, \psi_I) \int_{U(\R)\backslash G} F_{g',u}(g) \trans{\wbar{W^{d'}(g,-\wbar{\mu}, \psi_I)}} dg} \\
	& \qquad \times K^d_{w_l}(y,\mu) \sinmu^{d*}(\mu) d\mu.
\end{align*}
Interchanging the $d'$ sum with the $\mu$ integral (justified by the comment following \eqref{eq:FdgpDef}) gives
\begin{align}
\label{eq:PreBesselExpand}
	f(y) =& \sum_{d\ge 0} \int_{\frak{a}^d_0} F^d_{g',u}(\mu) K^d_{w_l}(y,\mu) \sinmu^{d*}(\mu) d\mu,
\end{align}
with $F^d_{g',u}(\mu)$ as in \eqref{eq:FdgpDef}.

We wish to show that $F^d_{g',u}(\mu)$ is essentially the Bessel transform of $f$ at weight $d$.
Our initial approach is to compute the Bessel transform of \eqref{eq:PreBesselExpand} at each weight $d$, but we encounter technical difficulties for $d \ge 2$, and these will be handled by analytic continuation, instead.
If the $K^d_{w_l}$ exhibit the expected orthogonality (which we demonstrate in certain cases below), we should have
\begin{align}
\label{eq:CddDef}
	\int_Y f(y) \wbar{K^d_{w_l}(y,\mu)} dy = F^d_{g',u}(\mu) C^{d,d}(\mu),
\end{align}
where $C^{d_1,d_2}(\mu)$ is defined by
\begin{align}
\label{eq:CddpDef}
	\int_Y \int_{\frak{a}^{d_1}_0} F(\mu) K^{d_1}_{w_l}(y,\mu) \sinmu^{d_1}_2(\mu) d\mu \, \wbar{K^{d_2}_{w_l}(y,\mu')} dy = F(\mu') C^{d_1,d_2}(\mu'),
\end{align}
for $\mu' \in \frak{a}^{d_2}_0$ and Schwartz-class $F(\mu)$.
These integrals converge by contour shifting near $y_i=0$ and the slight decay at $y_i=\infty$ coming from the double-Bessel integrals; see \cref{lem:DoubleBesselBound,lem:WhittBound}.

In \cref{sect:Ortho}, we will show that \eqref{eq:CddpDef} holds with $C^{d_1,d_2}(\mu')=0$ for $d_1 \ne d_2$, provided $\Max{d_1,d_2} \ge 2$.
We will show in \cref{sect:Cdd} that \eqref{eq:CddpDef} holds with
\begin{align}
\label{eq:CddEval}
	C^{d,d}(\mu) = 4,
\end{align}
for $d=0,1$ and $\mu\in\frak{a}^d_0$.
Lastly, in \cref{sect:d01Stade,sect:d01Ortho}, we show the remaining orthogonality
\begin{align}
\label{eq:C1001}
	C^{1,0}(\mu)=C^{0,1}(\mu)=0.
\end{align}
This is sufficient to demonstrate \eqref{eq:CddDef} at $d=0,1$ and hence also \cref{lem:WeirdBesselExpand} for $d=0,1$ and $\mu\in\frak{a}^d_0$ (so that $\mu=-\wbar{\mu}$).
(Certainly \eqref{eq:CddpDef} and \eqref{eq:CddEval} are true in the remaining case $d_1=d_2=d\ge 2$, but several interchange-of-integral problems become unwieldy.)

The extension of \cref{lem:WeirdBesselExpand} from $d=0,1$ and $\mu\in\frak{a}^d_0$ to all $d$ and $\mu\in\frak{a}^d_{\frac{1}{2}}$ proceeds by analytic continuation:
The Bessel transform of \cref{lem:WeirdBesselExpand} is a holomorphic function of $\wbar{\mu}$ by compact support of $f$, and \eqref{eq:FdgpDef} is entire and of super-polynomial decay in $\mu$ by \cref{lem:WhittTransSuperPoly} and the bound of \cref{lem:AnContWhittBd}.
Then we continue to $\mu = \mu^d(r)^{w_3}$ and apply \eqref{eq:KdwlAnCont} and \eqref{eq:Sigmadw3FE}.
This concludes the proof of \cref{lem:WeirdBesselExpand} and the Bessel expansion, \cref{thm:BesselExpand}, as well.

\begin{rem*}
It is tempting to take $f$ to be an approximation to the identity in \cref{lem:WeirdBesselExpand} and conclude that
\begin{align*}
	K^d_{w_l}(y,\mu) ``="& \frac{1}{2\pi^2}\sum_{d'\ge 0} (2d'+1) \Tr\paren{\Sigma^{d'}_d \trans{\wbar{W^{d'}(g',-\wbar{\mu}, \psi_I)}} \int_{U(\R)} \psi_I(x) u(x) W^{d'}(y w_l x g',\mu, \psi_I) dx}
\end{align*}
(independent of the choice of $g'$ and $u$), but we have no way to justify interchanging the $Y$ integral in $U(\R)\backslash G \cong Yw_l U(\R)$ with the $d'$ sum.
(Here is where we might benefit from trading the representation theory of $K$ for the representation theory of $U(\R)$.)
\end{rem*}

\subsection{The easy orthogonality}
\label{sect:Ortho}
In \eqref{eq:CddpDef}, we may identify $Y$ with $U(\R)\backslash G / \trans{U(\R)}$ (up to a measure zero subspace), by using the extension of $K^d_{w_l}(y,\mu)$ to
\[ K^d_{w_l}(xyvw_lx'w_l,\mu) = \psi_I(x) K^d_{w_l}(y,\mu) \psi_I(x'), \]
provided we extend $f(y)$ in an identical manner, and these $K^d_{w_l}$ are eigenfunctions of the Casimir operators with eigenvalues matching $p_{\rho+\mu}$.
For $d_1 \ge 2$, we know that $K^{d_1}_{w_l}$ is an eigenfunction of the symmetric operator $\Lambda_{\frac{d_1-1}{2}}$ (as in \eqref{eq:LambdaX}) and the eigenvalues of this operator in $\frak{a}^{d_2}_0$ (see \eqref{eq:LambdaXEigen1} and \eqref{eq:LambdaXEigen2}) are
\[ \piecewise{\displaystyle \prod_{i<j} \paren{(d_1-1)^2-(\mu_i-\mu_j)^2} & \If d_2=0,1, \\[15pt] \displaystyle \prod_\pm (d_1\pm d_2)\paren{(d_1-1\pm \tfrac{d_2-1}{2})^2-9r^2} & \If d_2 \ge 2, \mu=\mu^{d_2}(r),} \]
which are bounded away from zero when $d_1 \ne d_2$.
Hence $C^{d_1,d_2}(\mu)=0$ whenever $d_1 \ne d_2$ and $d_1 \ge 2$ (and the same for $d_2 \ge 2$).

The above approach does not apply for the orthogonality between $d=0$ and $d=1$ since these satisfy the exact same differential equations; the Bessel functions for these weights are distinguished by their asymptotics.
(The different asymptotics imply different functional equations for the Bessel functions, but the author was unable to apply this directly to the orthogonality.)

\subsection{Computing $C^{d,d}(\mu)$}
\label{sect:Cdd}

For $d=0,1$ and $F(\mu)$ Schwartz-class, holomorphic on $\frak{a}^d_{\frac{1}{2}}$ and satisfying \eqref{eq:WallachsOrthoAssm} (recall $F^d_{g',u}(\mu)$ is such a function by the comments preceeding \cref{lem:WeirdBesselExpand}), let
\begin{align*}
	\what{F}^d(y) :=& \int_{\frak{a}^d_0} F(\mu) K^d_{w_l}(y,\mu) \sinmu^{d*}(\mu) \, d\mu,
\end{align*}
and compute
\begin{align*}
	\innerprod{\what{F}^d,\what{F'}^d}_Y =& \int_Y \int_{\frak{a}^d_0} \int_{\frak{a}^d_0} F(\mu) \wbar{F'(\mu')} K^d_{w_l}(y,\mu) \wbar{K^d_{w_l}(y,\mu')} \sinmu^{d*}(\mu) \, d\mu \, \wbar{\sinmu^{d*}(\mu') \, d\mu'} \, dy \\
	=& \int_Y \int_{\frak{a}^d_0} \int_{\frak{a}^d_0} \int_{Y^+} \frac{F(\mu) \wbar{F'(\mu')}}{\wtilde{\Psi}^d(\mu,\mu',1)} \Tr\biggl(\Sigma^d_d  K^d_{w_l}(y,\mu) W^d(t,\mu,\psi_I) \\
	& \qquad \times \trans{\wbar{K^d_{w_l}(y,\mu') W^d(t,-\wbar{\mu'},\psi_I)}} \biggr) t_1^2 t_2 \,dt \, \sinmu^{d*}(\mu) \, d\mu \, \wbar{\sinmu^{d*}(\mu') \, d\mu'} \, dy,
\end{align*}
where $\wtilde{\Psi}^d(\mu,\mu',t)$ is given by \eqref{eq:tildePsid}.
Note that $\wtilde{\Psi}^d(\mu,\mu',1)$ is never zero on $\frak{a}^d_0$.

For the moment, the combined $y$-$\mu'$-$\mu$-$t$ integral converges absolutely, so we may freely interchange integrals, and we do this now to pull the $t$ integral outermost, where it will remain until the final step.
Then applying the definition of the Bessel functions,
\begin{align*}
	\innerprod{\what{F}^d,\what{F'}^d}_Y =& \int_{Y^+} \int_Y \int_{U(\R)} \int_{\frak{a}^d_0} \int_{U(\R)} \int_{\frak{a}^d_0} \frac{F(\mu) \wbar{F'(\mu')}}{\wtilde{\Psi}^d(\mu,\mu',1)} \\
	& \qquad \times \Tr\paren{\Sigma^d_d W^d(yw_l xt,\mu,\psi_I) \trans{\wbar{W^d(y w_l x' t,-\wbar{\mu'},\psi_I)}}} \\
	& \qquad \times \wbar{\psi_I(x)} \psi_I(x') \sinmu^{d*}(\mu) \, d\mu \, dx \, \wbar{\sinmu^{d*}(\mu') \, d\mu'} \, dx' \, dy \, t_1^2 t_2 \, dt.
\end{align*}

Now substitute $x \mapsto t x t^{-1}$, $x' \mapsto t x' t^{-1}$, $y t^{w_l} \mapsto y$ and then $x' \mapsto x x'$ and interchange integrals.
This time, we need to be rather more careful:
The $y t^{w_l} \mapsto y$ can actually be done as soon as the $t$ integral is outermost, and we can trivially pull the $x$ integral outside the $\mu'$ integral to begin with.
The double $x$-$x'$ integral converges absolutely, so we may do those substitutions and end with the integrals in the outermost-to-innermost order $t$,$y$,$x'$,$x$,$\mu'$,$\mu$.
Since the triple $y$-$x'$-$x$ integral just fails to converge absolutely, we need to carefully justify interchanging the $x'$ and $y$ integrals.

We fix $\alpha:\R\to\R$ with $\alpha(0)=1$ whose Fourier transform is smooth and compactly supported, and define
\[ \int_{U(\R)}^* \cdots dx' = \lim_{R\to\infty} \int_{U(\R)} \alpha(x'_1/R) \alpha(x'_2/R) \cdots dx', \]
and we may replace $\int_{U(\R)}$ with $\int_{U(\R)}^*$ in the absolutely-convergent $x'$ integral by dominated convergence.
We must show the limit in $R$ can be pulled out of the $y$ integral, and again by dominated convergence, after reversing the subsitution $x' \mapsto x x'$, it is sufficient to show
\begin{lem}
\label{lem:InterchangeOfIntegrals}
For $d=0,1$, suppose
\[ \Lambda^d(\mu') F''(\mu') := \frac{\wbar{F'(-\wbar{\mu'})}}{\wtilde{\Psi}^d(\mu,\mu',1)} \sinmu^{d*}(\mu') \]
is holomorphic with rapid decay on $\frak{a}^d_{\frac{1}{2}}$.
For fixed $t\in(\R^+)^2$, and $R>0$ large compared to $t$, there exists some $\eta > 0$ so that
\[ \int_{U(\R)} \alpha\paren{\tfrac{x'_1-x_1}{R}} \alpha\paren{\tfrac{x'_2-x_2}{R}} \int_{\frak{a}^d_0} F''(\mu') W^{d*}_m(y w_l x',\mu',\psi_I) d\mu' \wbar{\psi_t(x')} dx' \ll \frac{(y_1 y_2)^{1+\eta}}{(1+y_1+y_2)^{4\eta}}, \]
independent of $R$ and $x$, for each $m$, and the implied constant depends polynomially on $\mu$.
\end{lem}
The proof of this fact is deferred to \cref{sect:InterchangeOfIntegrals}.

Then we may interchange the $x'$ and $y$ integrals to arrive at
\begin{align*}
	\innerprod{\what{F}^d,\what{F'}^d}_Y =& 8\pi^2 \int_{Y^+} \int_{U(\R)}^* \int_{\frak{a}^d_0} \int_{U(\R)\backslash G} \int_{\frak{a}^d_0} \frac{F(\mu) \wbar{F'(\mu')}}{\wtilde{\Psi}^d(\mu,\mu',1)} \Tr\paren{\Sigma^d_d W^d(g,\mu,\psi_I) \trans{\wbar{W^d(g x',-\wbar{\mu'},\psi_I)}}} \\
	& \qquad \times \sinmu^{d*}(\mu) \, d\mu \, dg \, \wbar{\sinmu^{d*}(\mu') \, d\mu'} \, \psi_t(x') dx' \, t_1^4 t_2^3 \, dt,
\end{align*}
using \eqref{eq:dgBruhat}.

View the $g$ integral as a function of $x'\in G$, say
\[ \wbar{H(x')} = \int_{U(\R)\backslash G} \int_{\frak{a}^d_0} \frac{F(\mu) \wbar{F'(\mu')}}{\wtilde{\Psi}^d(\mu,\mu',1)} \trans{\wbar{W^d(g x',-\wbar{\mu'},\psi_I)}} \Sigma^d_d W^d(g,\mu,\psi_I) \sinmu^{d*}(\mu) \, d\mu \, dg. \]
Then for $k,k' \in K$, we have $H(kx'k') = \WigDMat{d}(k)H(x')\WigDMat{d}(k')$, and $H$ is an eigenfunction of every operator of the form \eqref{eq:KinvY} with eigenvalues matching $h^d_{-\wbar{\mu'}}$.
Then we use \cref{thm:GodeTrans} to see $\Tr(H(x')) = \Tr(H(I)) h^d_{\mu'}(x')$, so that
\begin{align*}
	\innerprod{\what{F}^d,\what{F'}^d}_Y =& 8\pi^2 \int_{Y^+} \int_{U(\R)}^* \int_{\frak{a}^d_0} \int_{U(\R)\backslash G} \int_{\frak{a}^d_0} \frac{F(\mu) \wbar{F'(\mu')}}{\wtilde{\Psi}^d(\mu,\mu',1)} \\
	& \qquad \times \Tr\paren{\Sigma^d_d W^d(g,\mu,\psi_I) \trans{\wbar{W^d(g,-\wbar{\mu'},\psi_I)}}} \\
	& \qquad \times \sinmu^{d*}(\mu) \, d\mu \, dg \, \wbar{h^d_{-\wbar{\mu'}}(x') \sinmu^{d*}(\mu') \, d\mu'} \, \psi_t(x') dx' \, t_1^4 t_2^3 \, dt\biggr).
\end{align*}

Now applying Whittaker inversion \eqref{eq:WhittInv} (here is where we need \eqref{eq:WallachsOrthoAssm}), we have
\begin{align*}
	\innerprod{\what{F}^d,\what{F'}^d}_Y =& 8\pi^2 \int_{Y^+} \int_{U(\R)}^* \int_{\frak{a}^d_0} \frac{F(\mu') \wbar{F'(\mu')}}{\wtilde{\Psi}^d(\mu',\mu',1)} \wbar{h^d_{-\wbar{\mu'}}(x') \sinmu^{d*}(\mu') \, d\mu'} \, \psi_t(x') dx' \, t_1^4 t_2^3 \, dt.
\end{align*}

For a fixed value of $R$, the $x'$ integral now converges absolutely, and in fact
\begin{prop}
\label{lem:FourSphFunAlpha}
For $t\in\R^2$ with $t_1 t_2 \ne 0$, $d=0,1$ and $\mu \in \frak{a}^d_0$,
\[ \int_{U(\R)}^* h^d_\mu(x) \wbar{\psi_t(x)} dx = \frac{1}{2\pi^2} \Tr\paren{\trans{\wbar{W^d(I,-\wbar{\mu},\psi_t)}}\Sigma^d_d W^d(I,\mu,\psi_t)}, \]
and the $R$-limit converges normally in $\mu$.
\end{prop}
The proof of this fact is deferred to \cref{sect:FourSph}.

Using \cref{lem:FourSphFunAlpha} and \eqref{eq:WhittGFEs}, the $x'$ integral may be evaluated as
\begin{align*}
	\innerprod{\what{F}^d,\what{F'}^d}_Y =& 4 \int_{Y^+} \int_{\frak{a}^d_0} \frac{F(\mu') \wbar{F'(\mu')}}{\wtilde{\Psi}^d(\mu',\mu',1)} \Tr\paren{\trans{W^d(t,\mu',\psi_I)} \Sigma^d_d \wbar{W^d(t,-\wbar{\mu'},\psi_I)}} \wbar{\sinmu^{d*}(\mu') \, d\mu'}\, t_1^2 t_2 \, dt.
\end{align*}

Applying \eqref{eq:tildePsid}, this becomes
\begin{align*}
	\innerprod{\what{F}^d,\what{F'}^d}_Y =& 4 \int_{\frak{a}^d_0} F(\mu) \wbar{F'(\mu')} \wbar{\sinmu^{d*}(\mu') \, d\mu'}.
\end{align*}

Taking $F'(\mu)$ to be an approximation to the identity, we have \eqref{eq:CddEval}.
(Since $F$ and $F'$ both satisfy \eqref{eq:WallachsOrthoAssm}, this does not interfere with the construction of an approximation to the identity.)

\subsection{Stade's formula for $d=0 \times d=1$}
\label{sect:d01Stade}
In order to apply the method of the previous section to the orthogonality of the $d=0$ and $d=1$ Bessel functions, we need some non-zero integral of the Whittaker functions in question.
We want to compute
\[ \Psi^{0,1}(\mu,\mu',t) := \frac{1}{2} \sum_\pm \int_{Y^+} W^{0*}(y,\mu) W^{1*}_{\pm1}(y,\mu') (y_1^2 y_2)^t dy, \]
and we do this by following along \cite[Section 5]{WeylI}.
This linear combination is chosen so that parity considerations give $\Psi^{0,1} = \Psi^{0,1}_{0,0}$, where by Mellin-Parseval, (using $\wtilde{G}^d$ as in \cref{sect:MinWtWhitt})
\begin{align*}
	16\pi^{3t} \Psi^{0,1}_{\ell_1,\ell_2} =& \int_{\Re(s)=\frak{s}} \wtilde{G}^0((0,0),s,\mu) \wtilde{G}^1((\ell_1,\ell_2),(2t_1-s_1,t_2-s_2),\mu') \frac{ds}{(2\pi i)^2}.
\end{align*}
The usual process (essentially due to Stade \cite[Lemma 2.1]{Stade02}, see \cite[Section 5]{WeylI}) of applying Barnes' first lemma (\cref{thm:BarnesFirst}) in reverse to produce a four-dimensional Mellin-Barnes integral
\begin{align*}
	16\pi^{3t} \Psi^{0,1}_{0,0} &= \tfrac{1}{4} \int_{\Re(u)=\frak{u}} \int_{\Re(s)=\frak{s}} \Gamma\paren{\tfrac{s_1+\mu_3+u_1}{2}}\Gamma\paren{\tfrac{s_2+u_1}{2}}\Gamma\paren{\tfrac{\mu_1-u_1}{2}}\Gamma\paren{\tfrac{\mu_2-u_1}{2}} \\
	& \qquad \times \Gamma\paren{\tfrac{1+t-s_1+\mu'_3+u_2}{2}}\Gamma\paren{\tfrac{-s_2+u_2}{2}}\Gamma\paren{\tfrac{t+\mu'_1-u_2}{2}}\Gamma\paren{\tfrac{t+\mu'_2-u_2}{2}} \\
	& \qquad \times \Gamma\paren{\tfrac{s_1-\mu_3}{2}} \Gamma\paren{\tfrac{s_2+\mu_3}{2}} \Gamma\paren{\tfrac{2t-s_1-\mu'_3}{2}} \Gamma\paren{\tfrac{1+t-s_2+\mu'_3}{2}} \frac{ds}{(2\pi i)^2} \frac{du}{(2\pi i)^2},
\end{align*}
followed by applying Barnes' first lemma to the $s$-integrals gives
\begin{align*}
	16\pi^{3t} \Psi^{0,1}_{0,0} &= \Gamma\paren{\tfrac{2t-\mu_3-\mu'_3}{2}} \Gamma\paren{\tfrac{1+t+\mu_3+\mu'_3}{2}} \\
	& \times \int_{\Re(u)=\frak{u}} \frac{\Gamma\paren{\tfrac{\mu_1-u_1}{2}}\Gamma\paren{\tfrac{\mu_2-u_1}{2}} \Gamma\paren{\tfrac{2t+\mu_3-\mu'_3+u_1}{2}} \Gamma\paren{\tfrac{u_1+u_2}{2}} \Gamma\paren{\tfrac{1+t+\mu'_3+u_1}{2}}}{\Gamma\paren{\tfrac{1+3t+u_1+u_2}{2}}} \\
	& \qquad \times \Gamma\paren{\tfrac{1+t-\mu_3+\mu'_3+u_2}{2}} \Gamma\paren{\tfrac{\mu_3+u_2}{2}} \Gamma\paren{\tfrac{t+\mu'_1-u_2}{2}}\Gamma\paren{\tfrac{t+\mu'_2-u_2}{2}} \frac{du}{(2\pi i)^2}.
\end{align*}

Applying Barnes' second lemma (\cref{thm:BarnesSecond}) to first the $u_1$- and then the $u_2$-integral gives
\begin{align}
\label{eq:Psi01Eval}
	\Psi^{0,1}(\mu,\mu',t) =& \frac{1}{4\pi^{3t} \Gamma\paren{\tfrac{1+3t}{2}}} \prod_{i,j} \Gamma\paren{\tfrac{\delta_{j=3}+t+\mu_i+\mu'_j}{2}}.
\end{align}
We note that this is non-zero for all $\mu$, $\mu'$ at $t=1$.

One might hope that this is related to the Rankin-Selberg convolution of a $d=0$ and a $d=1$ Maass cusp form, but the $V$ sum in the Fourier-Whittaker expansion \eqref{eq:CuspFourierExpand} prevents this from working.
(See the analysis of the next section.)

\subsection{The orthogonality of $d=0,1$}
\label{sect:d01Ortho}
If we follow along the method of \cref{sect:Cdd} using $\Psi^{0,1}$ in place of $\wtilde{\Psi}^d$ and dropping the assumption \eqref{eq:WallachsOrthoAssm}, we have
\begin{align*}
	\innerprod{\what{F}^0,\what{F'}^1}_Y =& \frac{1}{8\pi^2} \int_{Y^+} \int_{U(\R)} \int_{\frak{a}^1_0} \int_{U(\R)\backslash G} \int_{\frak{a}^0_0} \frac{F(\mu) \wbar{F'(\mu')}}{\wtilde{\Psi}^{0,1}(\mu,\mu',1)} \Matrix{\tfrac{1}{2}&0&\tfrac{1}{2}} \trans{\wbar{W^{1*}(g x',-\wbar{\mu'},\psi_I)}} \\
	& \qquad \times W^{0*}(g,\mu,\psi_I) \sinmu^{0*}(\mu) \, d\mu \, dg \, \wbar{\sinmu^{1*}(\mu') \, d\mu'} \, \psi_t(x') dx' \, t_1^4 t_2^3 \, dt.
\end{align*}

Then the $g$-integral, as a function of $x'$ may be written as
\[ \Matrix{\tfrac{1}{2}&0&\tfrac{1}{2}} \trans{\wbar{H(x')}}, \]
for a row-vector valued function $H(x')$ which transforms as
\[ H(kyk') = H(y) \WigDMat{1}(k'). \]
Thus it is sufficient to determine $H(y) = \Matrix{H_{-1}(y)&H_0(y)&H_1(y)}$, but $Y^+$ and $V$ commute, so we have
\[ H(y) = H(\vpmpm{\varepsilon_1,\varepsilon_2} y) = H(y) \WigDMat{1}(\vpmpm{\varepsilon_1,\varepsilon_2}) = \varepsilon_2 \Matrix{\varepsilon_1 H_{-\varepsilon_2}(y)&H_0(y)&\varepsilon_1 H_{\varepsilon_2}(y)}. \]
We conclude that $H(y)=0$, and hence
\[ \innerprod{\what{F}^0,\what{F'}^1}_Y = 0. \]
Then taking either $F$ or $F'$ to be an approximation to the identity implies \eqref{eq:C1001}.
(This is why we must drop the assumption \eqref{eq:WallachsOrthoAssm}.)

\subsection{The interchange of integrals}
\label{sect:InterchangeOfIntegrals}
Here we prove \cref{lem:InterchangeOfIntegrals}.
Writing
\[ T=\abs{y_1^{-1}+y_1}\abs{y_2^{-1}+y_2}, \]
and $x^* y^* k^*=w_l x$, we are trying to prove
\begin{align*}
	& \int_{U(\R)} \alpha\paren{\tfrac{x'_1-x_1}{R}} \alpha\paren{\tfrac{x'_2-x_2}{R}} \int_{\frak{a}^d_0} F''(\mu') W^{d*}(y y^*,\mu',\psi_I) d\mu' \\
	& \qquad \times \WigDMat{d}(k^*) \e{-t_1 x'_1-t_2 x'_2+y_1 x^*_1+y_2 x^*_2} dx' \ll \abs{y_1 y_2} T^{-\eta},
\end{align*}
for some $\eta > 0$.
In other words, we need to save a factor of $T^\eta$ over the trivial bound.
By sending $x' \mapsto \vpmpm{--}(x')^\iota\vpmpm{--}$ as necessary, we may assume $y_2 \ge y_1$.

We write $f(yy^*)$ for the $\mu'$ integral, and note that by the assumptions on $F''(\mu)$ and \cref{lem:WhittBound}, we have the bound
\[ f(y') \ll \abs{y_1' y_2'}\prod_{i=1}^2 \abs{y_i'+1/y_i'}^{-1/3}, \]
as always, with polynomial dependence of the implied constant on $\mu$.
In particular, we may assume $y_2 \ge 1$ and $T > 6^{1/\eta}$ since otherwise replacing $f(y y^*)$ with $\abs{y_1^* y_2^*}^{1+\eta}$ is sufficient for absolute convergence of the $x'$ integral, independent of $R$ and $x$.

We may simplify slightly by Fourier-expanding $\alpha$ and showing the bound
\[ \mathcal{I} := \int_{U(\R)} f(y y^*) \WigDMat{d}(k^*) \e{-t'_1 x'_1-t'_2 x'_2+y_1 x^*_1+y_2 x^*_2} dx' \ll \abs{y_1 y_2} T^{-\eta}, \]
with $t_i' = t_i+\xi_i/R$, where $\xi_i$ is the variable introduced by the Fourier inversion.
Since $t$ is fixed and $\hat{\alpha}$ is compactly supported, we may assume $R$ is large enough that $t_i' \asymp 1$.

We use the computations of \cref{sect:ExplicitStars} after the substitution \eqref{eq:wlxSub}.
We may apply smooth partitions of unity (we may Mellin expand the original $f$ as necessary), and we drop the primes for convenience, so that
\begin{align*}
	\mathcal{I} \ll& \log^3 T \sup_{X_1,X_2,X_3, k_1,k_2,k_3} \wtilde{\mathcal{I}}, \\
	\wtilde{\mathcal{I}} :=& X_3 \int_{U(\R)} w(x) \e{\phi(x)} dx, \\
	w(x) :=& \tilde{f}\paren{\frac{(1+x_2^2) X_3^2}{X_2^2 y_1^{4/3} y_2^{2/3}},\frac{(1+x_1^2) X_3^2}{X_1^2 y_1^{2/3} y_2^{4/3}},\frac{x_3^2}{X_3^2}} \paren{\frac{1+ix_1}{\sqrt{1+x_1^2}}}^{k_1} \paren{\frac{1+ix_2}{\sqrt{1+x_2^2}}}^{k_2} \paren{\frac{1+ix_3}{\sqrt{1+x_3^2}}}^{k_3}, \\
	\phi(x) :=& -t_1 x_1\frac{\sqrt{1+x_3^2}}{\sqrt{1+x_2^2}}-t_1 \frac{x_2 x_3}{\sqrt{1+x_2^2}}-t_2 x_2+y_1 \frac{x_2}{1+x_2^2} \\
	& \qquad +y_1\frac{x_1 x_3}{(1+x_2^2)\sqrt{1+x_3^2}}+y_2 x_1\frac{\sqrt{1+x_2^2}}{(1+x_1^2)\sqrt{1+x_3^2}},
\end{align*}
where $\tilde{f}$ is smooth and compactly supported on $[\frac{1}{2},2]^3$, $k_1,k_2,k_3\in\Z$, $\abs{k_i} \le d$, and
\begin{align}
\label{eq:Xbds}
	T^{-2\eta} < X_1,X_2 < T^{2\eta}, \qquad T^{-2\eta} < X_3 < \sqrt{2} \Min{X_1 y_1^{1/3} y_2^{2/3},X_2 y_1^{2/3} y_2^{1/3}}.
\end{align}
The lower bound on $X_3$ follows because we can trivially bound the integral in the range $\abs{x_3} < 2 T^{-2\eta}$ and the upper bound follows from $1+x_1^2, 1+x_2^2 \ge 1$; the bounds on $X_1,X_2$ follow by the polynomial decay of $f$.
Now it is sufficient to show
\begin{align*}
	\wtilde{\mathcal{I}} \ll& \abs{y_1 y_2} T^{-\eta}
\end{align*}
for some $\eta > 0$ and $t$ in some fixed, compact set, and we take $\eta$ in the statement of the lemma to be $\eta/2$ to account for the $\log^3 T$.

For $k \in \Z$, we have
\[ \frac{d}{dx} \paren{\frac{1+ix}{\sqrt{1+x^2}}}^k = \frac{k i}{1+x^2}\paren{\frac{1+ix}{\sqrt{1+x^2}}}^k, \]
so
\begin{align*}
	\partial_{x_1} w(x) =& \paren{\frac{2 x_1 X_3^2}{X_1^2 y_1^{2/3} y_2^{4/3}} \tilde{f}_2\paren{\ldots}+\frac{k_1 i}{1+x_1^2}\tilde{f}\paren{\ldots}} \paren{\ldots}^{k_1} \paren{\ldots}^{k_2} \paren{\ldots}^{k_3} \ll \frac{X_3}{X_1 y_1^{1/3} y_2^{2/3}},
\end{align*}
using
\[ \frac{1}{1+x_1^2} \le \frac{1}{\sqrt{1+x_1^2}}, \qquad \tilde{f}_2(u_1,u_2,u_3) := \partial_{u_2} \tilde{f}(u_1,u_2,u_3). \]
Similarly,
\[ \partial_{x_1}^n w(x) \ll \paren{\frac{X_3}{X_1 y_1^{1/3} y_2^{2/3}}}^n. \]

We will apply \cref{lem:InterchangeStationaryPhase} with
\[ A = -t_1 \frac{\sqrt{1+x_3^2}}{\sqrt{1+x_2^2}}+y_1\frac{x_3}{(1+x_2^2)\sqrt{1+x_3^2}}, \qquad B=y_2\frac{\sqrt{1+x_2^2}}{\sqrt{1+x_3^2}}, \qquad C=\frac{X_1 y_1^{1/3} y_2^{2/3}}{\sqrt{1+x_3^2}}, \]
so that
\[ B \asymp \frac{X_2 y_1^{2/3} y_2^{4/3}}{X_3^2}, \qquad C \asymp \frac{X_1 y_1^{1/3} y_2^{2/3}}{X_3}, \]
and
\[ A = -\frac{t_1}{\sqrt{1+x_2^2}\sqrt{1+x_3^2}}\paren{x_3^2-\paren{\frac{y_1}{t_1\sqrt{1+x_2^2}}} x_3+1} \asymp \frac{\abs{x_3-r_+}\abs{x_3-r_-}}{X_2 y_1^{2/3} y_2^{1/3}}, \]
with
\[ r_\pm=\frac{y_1}{2t_1\sqrt{1+x_2^2}}\pm\sqrt{\frac{y_1^2}{4t_1^2(1+x_2^2)}-1}, \]
which might be complex.
Note that \eqref{eq:Xbds} implies $C \gg 1$ and
\[ \frac{B}{C^2} \asymp \frac{X_2}{X_1^2}. \]

By \cref{lem:InterchangeStationaryPhase}, the $x_1$ integral saves a factor $T^\eta$ unless
\begin{align}
\label{eq:Pigeons}
	\abs{x_3-r_+}\abs{x_3-r_-} \ll X_3^2 T^\eta \Max{\frac{y_1^{1/3} X_2}{y_2^{1/3} X_1 X_3}, \frac{X_2^2 y_1^{2/3} y_2^{1/3}}{X_1^2 X_3^2}}.
\end{align}
On the other hand, if $x_3$ is constrained to an interval of length less than $X_3 T^{-\eta}$ then that portion of the $x_3$ integral saves a factor $T^\eta$.
By the pigeonhole principle applied to \eqref{eq:Pigeons}, we conclude that $\wtilde{\mathcal{I}}$ is sufficiently small unless
\[ \Max{\frac{y_1^{1/3} X_2}{y_2^{1/3} X_1 X_3}, \frac{X_2^2 y_1^{2/3} y_2^{1/3}}{X_1^2 X_3^2}} \gg T^{-3\eta}, \]
or in other words, unless $X_3 \ll y_1^{1/3} y_2^{1/6} T^{11\eta}$.

If $X_3 \ll y_1^{1/3} y_2^{1/6} T^{11\eta}$, then
\[ \abs{y_1}^{-1} \ll \abs{y_2}^{1+\frac{\frac{1}{2}+39\eta}{1+39\eta}}, \qquad T \ll \abs{y_2}^{2+\frac{\frac{1}{2}+39\eta}{1+39\eta}}, \]
and for $\eta < \frac{1}{111}$,
\[ \frac{B}{C} \gg \Max{T^{2\eta}, C^{1/4}}, \qquad \frac{B}{C^2} \ll T^{-\eta} C, \]
so a second application of \cref{lem:InterchangeStationaryPhase} implies that $\wtilde{\mathcal{I}}$ is suffciently small in this case, as well.

\subsection{The Fourier transform of a spherical function}
\label{sect:FourSph}
Here we prove \cref{lem:FourSphFunAlpha}; the proof follows that of \cite[Lemma 3]{Me01}.
By sending $x \mapsto \vpmpm{--}x^\iota\vpmpm{--}$ and using the $K$-invariance (under conjugation) of $h^d_\mu$, it's enough to prove
\[ \int_{U(\R)}^* h^d_\mu(\trans{x}) \wbar{\psi_t(x)} dx = \frac{1}{2\pi^2} \Tr\paren{\trans{\wbar{W^d(I,-\wbar{\mu},\psi_{t^\iota})}}\Sigma^d_d W^d(I,\mu,\psi_{t^\iota})}. \]

Using the analysis of \cite[Section 4.3]{Me01}, which naturally extends to the non-spherical case, we may write the spherical function as
\begin{align}
\label{eq:AltSphFun}
	h^d_\mu(g) = \frac{1}{2\pi^2} \Tr\paren{\int_{U(\R)} \trans{\wbar{I^d_{-\wbar{\mu}}(\trans{u})}} \Sigma^{d*}_d I^d_{\mu}(\trans{u} g) du}.
\end{align}
The proof of this fact is simply the conversion \eqref{eq:dbark} (and the following comment) and the $V$-invariance of the original integrand of \eqref{eq:SphFunDef}:
\begin{align*}
	h^d_\mu(g) = \int_{V\backslash K} \Tr\paren{\trans{\wbar{\WigDMat{d}(\wbar{k})}} \Sigma^{d*}_d I^d_{\mu}(\wbar{k}g)} d\wbar{k},
\end{align*}
plus some manipulations of the Iwasawa decomposition of $\trans{u}=x^* y^* k^*$:
\[ I^d_\mu(k^* g) = I^d_\mu((y^*)^{-1} (x^*)^{-1} \trans{u} g) = p_{-\rho-\mu}(y^*) I^d_\mu(\trans{u} g), \qquad p_{\rho-\mu}(y^*)\trans{\wbar{\WigDMat{d}(k^*)}} = \trans{\wbar{I^d_{-\wbar{\mu}}(\trans{u})}}.  \]

We need a reasonably precise statement of the convergence and asymptotics of the $u_3$ integral in such an integral representation:
For $x_1,x_2\in\R$, $0 \le n_1, n_2,n_3 \in \Z$, $s_1,s_2 \in \C$, with $n_1+n_2+2s_1+2s_2 < -1$, define $X_3(x_1,x_2, n_1, n_2, n_3, s_1, s_2)$ by the integral
\begin{align*}
	&\int_{-\infty}^\infty \paren{\frac{x_3}{\sqrt{\xi_1}}}^{n_1} \paren{\frac{x_1 x_2-x_3}{\sqrt{\xi_2}}}^{n_2} \paren{\frac{\sqrt{\xi_1}}{\sqrt{1+x_2^2}\sqrt{\xi_2}}}^{n_3} \xi_1^{s_1} \xi_2^{s_2} dx_3,
\end{align*}
where
\[ \xi_1=1+x_2^2+x_3^2, \qquad \xi_2=1+x_1^2+(x_1 x_2-x_3)^2. \]

\claim For $\epsilon>0$, set $t_i = \Max{-\frac{1}{2}+\epsilon,\Re(s_i)}$, $i=1,2$.
Suppose $\Re(s_1),\Re(s_2) < \delta < 0$, $\Re(s_1+s_2) <-\frac{1}{2}-\delta$ and $\epsilon < \Min{\delta,\frac{1}{4}}$ so that $-2 < 2t_1+2t_2 < -1$, then
\[ X_3 \ll_{\epsilon,\delta} (1+x_2^2)^{\Re(s_1)-t_1} (1+x_1^2)^{\Re(s_2)-t_2} \abs{x_1 x_2}^{1+2t_1+2t_2}. \]
If instead $\Re(s_1),\Re(s_2) < -\frac{1}{2}-\delta$ and $\abs{x_1}\abs{x_2} > 1$, we have
\begin{align*}
	X_3 \ll_\delta& \abs{x_1}^{2\Re(s_2)}(1+x_2^2)^{\Re(s_1+s_2)+\frac{1}{2}}+\abs{x_2}^{2\Re(s_1)}(1+x_1^2)^{\Re(s_1+s_2)+\frac{1}{2}}\\
	& \qquad +\abs{x_1}^{2\Re(s_2)+1}\abs{x_2}^{2\Re(s_1)+1}(1+x_1^2)^{\Re(s_1)}(1+x_2^2)^{\Re(s_2)}.
\end{align*}
\begin{proof}
Since the quantities $\alpha_i$ in \eqref{eq:wlxIwa} all lie on the unit circle and $\abs{x_1 x_2-x_3} \le \sqrt{\xi_2}$, the $n_i$ powers are all $\le 1$ in absolute value and we have
\begin{align}
\label{eq:X3bdFirst}
	X_3 \ll& \int_{-\infty}^\infty (1+x_2^2+x_3^2)^{\Re(s_1)} (1+x_1^2+(x_1 x_2-x_3)^2)^{\Re(s_2)} dx_3,
\end{align}
and similarly
\[ X_3 \ll (1+x_2^2)^{\Re(s_1)-t_1} (1+x_1^2)^{\Re(s_2)-t_2} \int_{-\infty}^\infty \abs{x_3}^{2t_1} \abs{x_1 x_2-x_3}^{2t_2} dx_3. \]
Then we substitute $x_3 \mapsto x_1 x_2 x_3$, and the resulting (convergent) integral may be bounded entirely in terms of $\delta$ and $\epsilon$.

For the second bound, we return to \eqref{eq:X3bdFirst} and split the $x_3$ integral into regions $\abs{x_3} < \frac{1}{2}\abs{x_1 x_2}$, $\abs{x_1 x_2 - x_3} < \frac{1}{2}\abs{x_1 x_2}$ and the remaining region, which has $\abs{x_1 x_2 - x_3} \asymp \abs{x_3} \gg \abs{x_1 x_2}$.
The integral over the first region is
\begin{align*}
	\ll& \abs{x_1}^{2\Re(s_2)} (1+x_2^2)^{\Re(s_2)} \int_0^{\frac{1}{2}\abs{x_1 x_2}} (1+x_2^2+x_3^2)^{\Re(s_1)} dx_3,
\end{align*}
into which we substitute $x_3 \mapsto x_3\sqrt{1+x_2^2}$ and use $\Re(s_1) < -\frac{1}{2}$; the second region is handled similarly.
The integral over the remaining region is
\begin{align*}
	\ll& \int_{\abs{x_1 x_2}}^\infty (x_2^2+x_3^2)^{\Re(s_1)} (x_1^2+x_3^2)^{\Re(s_2)} dx_3 \\
	=& \abs{x_1}^{2\Re(s_2)+1}\abs{x_2}^{2\Re(s_1)+1} \int_1^\infty (1+x_1^2 x_3^2)^{\Re(s_1)} (1+x_2^2 x_3^2)^{\Re(s_2)} dx_3,
\end{align*}
and at most one of $\abs{x_1},\abs{x_2} < 1$.
\end{proof}

The claim implies in particular that the six-fold integral
\begin{align*}
	& \int_{U(\R)} \alpha\paren{\tfrac{x_1}{R}} \alpha\paren{\tfrac{x_2}{R}} h^d_\mu(\trans{x}) \wbar{\psi_t(x)} dx \\
	&= \frac{1}{2\pi^2} \Tr\paren{\int_{U(\R)} \alpha\paren{\tfrac{x_1}{R}} \alpha\paren{\tfrac{x_2}{R}} \int_{U(\R)} \trans{\wbar{I^d_{-\wbar{\mu}}(\trans{u})}} \Sigma^{d*}_d I^d_{\mu}(\trans{u} \trans{x}) du \wbar{\psi_t(x)} \, dx}
\end{align*}
also converges absolutely (and normally in $\mu$) when
\begin{align}
\label{eq:nubd}
	\abs{\Re(\mu_1-\mu_2)}, \abs{\Re(\mu_2-\mu_3)} < \frac{1}{4}-2\delta, \qquad \delta > 0.
\end{align}
Indeed, substituting $x \mapsto x u^{-1}$, the $x_3$ and $u_3$ integrals are bounded by
\[ X_3(x_2,x_1,0,0,0,-\tfrac{3}{8}-\delta,-\tfrac{3}{8}-\delta) \quad \text{ and } \quad X_3(u_2,u_1,0,0,0,-\tfrac{3}{8}-\delta,-\tfrac{3}{8}-\delta), \]
respectively.
The convergence of the remaining integrals follows from
\[ \int_{-\infty}^\infty \abs{\alpha\paren{\tfrac{x_i-u_i}{R}}} \abs{x_i}^{-\frac{1}{2}-\delta} dx_i \ll_R (1+\abs{u_i})^{-\frac{1}{2}-\delta}, \]
by the rapid decay of $\alpha$.

This is where our construction would fail for $d \ge 2$, as some additional weight in $x_3$ would be necessary farther away from $\Re(\mu)=0$, and that would necessarily lie outside the span of the characters of $U(\R)$.
As with the Jacquet integral for the Whittaker function \eqref{eq:JacWhittDef} (whose oscillation is due to the character), the difficulty arises because the trivial bound on the defining integral for the spherical function \eqref{eq:AltSphFun} (whose oscillation is due to the Wigner $\WigDName$-matrix) does not give the correct asymptotics on $\mu \in \frak{a}^d_0$ for $d \ge 2$ (compare the power of $y$ in \cref{lem:AnContWhittBd} and \cref{cor:GL3WhittBd} for the $d \ge 2$ Whittaker function).

So we now assume, by analytic continuation that $\mu$ satisfies \eqref{eq:nubd} and
\begin{align}
\label{eq:muCond}
	\Re(\mu_1) < \Re(\mu_2) < \Re(\mu_3).
\end{align}

We apply the expansion \eqref{eq:AltSphFun}, substitute $x \mapsto x u^{-1}$, and use the same trick as in \cref{sect:InterchangeOfIntegrals}, to replace $\int^*_{U(\R)}$ with $ \lim_{R\to\infty} \int_{[-R,R]^2\times\R}$:
\begin{align*}
	\int_{U(\R)} h^d_\mu(\trans{x}) \wbar{\psi_t(x)} dx =& \frac{1}{2\pi^2} \Tr\Biggl(\lim_{R\to\infty} \lim_{R' \to \infty} \int_{\R^2} \what{\alpha}(\xi_1) \what{\alpha}(\xi_2)  \\
	& \qquad \times \int_{U(\R)} \trans{\wbar{I^d_{-\wbar{\mu}}(\trans{u})}} \e{(t_1-\xi_1/R) u_1+(t_2-\xi_2/R) u_2} du \\
	& \qquad \times \Sigma^{d*}_d \int_{[-R',R']^2\times\R} I^d_{\mu}(\trans{x}) \e{-(t_1-\xi_1/R) x_1-(t_2-\xi_2/R) x_2} dx\Biggr).
\end{align*}
For ease of comparison, we substitute $\trans{u} \mapsto w_l u w_l$ and $\trans{x} \mapsto w_l x w_l$ and apply the cycle invariance of the trace to the resulting $\WigDMat{d}(w_l)$, which effectively swaps $u_1$ and $u_2$ and replaces $\trans{u} \to w_l u$, and the same for $x$.

Now we consider a fixed $R$ and assume that $R$ is large enough that $t_{3-i}' := t_i-\xi_i/R \asymp t_i$, $i=1,2$.
Then the $u$ integral is exactly $\trans{\wbar{W^d(I,-\wbar{\mu},\psi_{t'})}}$ (which becomes $\trans{\wbar{W^d(I,-\wbar{\mu},\psi_{t^\iota})}}$ as $R \to \infty$), and it is enough to prove that
\begin{align*}
	\lim_{R' \to \infty} \int_{[-R',R']^2\times\R} I^d_{\mu}(\trans{x}) \wbar{\psi_{t'}(x)} dx =& W^d(I,\mu,\psi_{t'}),
\end{align*}
where the limit converges normally in $\mu$ satisfying \eqref{eq:nubd} and \eqref{eq:muCond} for $t'$ in some compact set.
Moreover, on $\Re(\mu_1)>\Re(\mu_2)>\Re(\mu_3)$, the integral converges to the Whittaker function, so by analytic continuation, it is enough to simply prove the normal convergence of the limit on the larger region \eqref{eq:nubd}, dropping the condition \eqref{eq:muCond}.

As before, we may replace the Wigner $\WigDName$-matrix with integral powers of $\alpha_1$, $\alpha_2$, and $\alpha_3$ as given by \eqref{eq:wlxIwa} at $y=I$.
Then using $\abs{s_i}=1$ we may apply the binomial theorem to reduce again to positive integer powers of
\begin{align*}
	&\frac{\sqrt{\xi_1}}{\sqrt{1+x_2^2}\sqrt{\xi_2}}, && \frac{1}{\sqrt{1+x_2^2}}, && \frac{x_2}{\sqrt{1+x_2^2}}, && \frac{x_1}{\sqrt{\xi_2}}, && \frac{(x_2 x_2-x_3)}{\sqrt{1+x_2^2}\sqrt{\xi_2}}, && \frac{x_3}{\sqrt{\xi_1}}, && \frac{\sqrt{1+x_2^2}}{\sqrt{\xi_1}}.
\end{align*}
So we have to prove the normal convergence of
\begin{align*}
	&\lim_{R' \to \infty} \int_{[-R',R']^2} X_3(x_1,x_2,n_1,n_2,n_3,s_1,s_2) x_1^{n_4} x_2^{n_5} (1+x_2^2)^{-\frac{n_2+n_5+n_6-n_7}{2}} \e{-t'_1 x_1-t'_2 x_2} dx_1\, dx_2,
\end{align*}
where
\[ 2s_1=-1+\mu_3-\mu_2-n_7, \qquad 2s_2=-1+\mu_2-\mu_1-n_4, \qquad 0 \le n_i \in \Z. \]

To compute its derivatives, we write $X_3$ as
\begin{align*}
	& (1+x_2^2)^{-\frac{n_3}{2}} \int_{-\infty}^\infty x_3^{n_1} \paren{x_1 x_2-x_3}^{n_2} \xi_1^{s_1-\frac{n_1-n_3}{2}} \xi_2^{s_2-\frac{n_2+n_3}{2}} dx_3,
\end{align*}
then
\begin{align*}
	\frac{d}{dx_2} X_3 =& -n_3\frac{x_2}{1+x_2^2} X_3(x_1,x_2,n_1,n_2,n_3,s_1,s_2) \\
	&\qquad+n_2 x_1 X_3(x_1,x_2,n_1,n_2-1,n_3,s_1,s_2-\tfrac{1}{2}) \\
	&\qquad+(2s_1-n_1+n_3) x_2 X_3(x_1,x_2,n_1,n_2,n_3,s_1-1,s_2) \\
	&\qquad+(2s_2-n_2-n_3) x_1 X_3(x_1,x_2,n_1,n_2+1,n_3,s_1,s_2-\tfrac{1}{2}).
\end{align*}
By sending $x_3 \mapsto x_1 x_2-x_3$, we also have
\begin{align*}
	\frac{d}{dx_1} X_3 =& n_1 x_2 X_3(x_1,x_2,n_1-1,n_2,n_3,s_1-\tfrac{1}{2},s_2) \\
	&\qquad+(2s_1-n_1+n_3) x_2 X_3(x_1,x_2,n_1+1,n_2,n_3,s_1-\tfrac{1}{2},s_2) \\
	&\qquad+(2s_2-n_2-n_3) x_1 X_3(x_1,x_2,n_1,n_2,n_3,s_1,s_2-1).
\end{align*}

If we perform an integration by parts in both $x_1$ and $x_2$ (integrating the exponential terms and differentiating the rest), using the above formulas, the worst case for convergence (including the endpoints) is the term
\begin{align*}
	&\int_{[-R',R']^2} X_3(x_1,x_2,n_1,n_2,n_3,s_1-\tfrac{1}{2},s_2-\tfrac{1}{2}) \\
	& \qquad \times x_1^{n_4+1} x_2^{n_5+1} (1+x_2^2)^{-\frac{n_2+n_5+n_6-n_7}{2}} \e{-t'_1 x_1-t'_2 x_2} dx_1\, dx_2,
\end{align*}
when $\Re(s_1)=-\frac{3}{8}-\delta-\frac{n_7}{2}$ and $\Re(s_2)=-\frac{3}{8}-\delta-\frac{n_4}{2}$ where $\delta>0$ is small, and in this case, the claim shows the integral is bounded by
\begin{align*}
	&\int_0^{R'} \int_0^{1/x_2} (x_1 x_2)^{4\epsilon} (1+x_1^2)^{-\frac{3}{8}-\delta-\epsilon} (1+x_2^2)^{-\frac{3}{8}-\delta-\epsilon} dx_1\, dx_2 \\
	&+\int_1^{R'} \int_{1/x_2}^1 \paren{x_1^{-\frac{7}{4}-2\delta} x_2^{-\frac{5}{2}-4\delta} +x_2^{-\frac{7}{4}-2\delta}+x_1^{-\frac{3}{4}-2\delta} x_2^{-\frac{5}{2}-4\delta}} dx_1\, dx_2 \\
	&+\int_1^{R'} \int_1^{x_2} \paren{x_1^{-\frac{7}{4}-2\delta} x_2^{-\frac{5}{2}-4\delta} +x_1^{-\frac{5}{2}-4\delta} x_2^{-\frac{7}{4}-2\delta}+x_1^{-\frac{5}{2}-4\delta} x_2^{-\frac{5}{2}-4\delta}} dx_1\, dx_2,
\end{align*}
for some small $\epsilon>0$ and this has the required convergence.

\section{Smooth sums of Kloosterman sums}
\label{sect:SmoothSums}

We now prove \cref{thm:SmoothSums}.
If we define $\hat{f}(s,v)$ to be the Mellin transform of $f$ at each sign,
\[ \hat{f}(s,v) = \int_{Y^+} f(yv) y_1^{s_1} y_2^{s_2} \frac{dy_1 \, dy_2}{y_1 y_2}, \]
then the transform that occurs in applying \eqref{eq:PreArithKuz2} to the sum in \cref{thm:SmoothSums} is
\begin{align}
\label{eq:tildefMellin}
	\wtilde{f}^d(\mu) :=& \frac{1}{4} \int_Y \abs{y_1 y_2} f\paren{\frac{X_1 y_1}{m_1 n_2}, \frac{X_2 y_2}{m_2 n_1}} \wbar{K^d(y,\mu)} dy \\
	=& 4\pi^4 \sum_{v\in V} \int_{\Re(s)=(\epsilon,\epsilon)} \hat{f}(-s,v) \what{K}^d_{w_l}(s,v,\wbar{\mu}) \paren{\frac{X_1}{4\pi^2 m_1 n_2}}^{s_1} \paren{\frac{X_2}{4\pi^2 m_2 n_1}}^{s_2} \frac{ds}{(2\pi i)^2} \nonumber
\end{align}
Note that the super-polynomial decay of $\hat{f}(s,v)$ implies that on $\norm{s} \gg \norm{\mu}^\epsilon$, the integrand (and the residue at any pole of $\what{K}^d_{w_l}$) is $\ll \norm{\mu}^{-100}$.

The proof is given by analyzing the Mellin transform of the Bessel function at each sign.
As noted in \cref{sect:BesselStirling}, if we apply Stirling's formula to the Mellin transforms at different signs for $\mu \in \frak{a}^0_{\frac{1}{2}}=\frak{a}^1_{\frac{1}{2}}$, they vary only by the regions of exponential decay, while the polynomial parts all agree.
Thus the analysis of the weight zero and weight one spectra is identical to that of \cite[section 8]{ArithKuzI}, and we just need to analyze the weight $d \ge 2$ part.
Of course, the $\sgn(y)=(-,-)$ case does have exponential decay in every direction for the weight zero and weight one spectra, but it will turn out that the weight $d \ge 2$ part will recover the terms $X_1^{-1-\epsilon}+X_2^{-1-\epsilon}$, so there is nothing to be gained from a special treatment.

For the weight $d \ge 3$ case, we simply apply \eqref{eq:BesselStirlingdge2} to \eqref{eq:tildefMellin} and insert this into the Weyl law using the comment following \eqref{eq:JdBound}:
We have
\[ \int_{\mathcal{B}^{d*}} \rho_\Xi(n)\wbar{\rho_\Xi(m)} \wtilde{f}^d(\mu_\Xi) d\Xi \ll d^{1+\epsilon} (X_1 X_2)^\epsilon, \]
and summing over $d \le (X_1 X_2)^\epsilon \Max{X_1^{-1/2},X_2^{-1/2}}$ gives the result.

Because of the failure of absolute convergence at $d=2$ (and before demonstrating \cref{cor:SpecKuz2CondConv}), we currently only have an upper bound of $T^4$ (instead of $T^3$) for the sum over cusp forms of spectral parameters $\abs{r} \le T$.
However, for $d=2$ and $v=\vpmpm{++}$, the factor $\what{K}^d_{w_l}(s,v,\wbar{\mu^d(r)})$ is zero, and for $v=\vpmpm{--}$ it has exponential decay in $r$.
For the remaining cases $v=\vpmpm{+-}$ and $v=\vpmpm{-+}$, we may shift contours to $\Re(s)=(\epsilon,-\frac{1}{3}-10\epsilon)$ and $\Re(s)=(-\frac{1}{3}-10\epsilon,\epsilon)$, respectively, without encountering any poles, and this is sufficient for our purposes by \eqref{eq:BesselStirlingdge2}.

\section{The reflection formula}
\label{sect:Reflect}

We now analyze in detail the reflection formula \eqref{eq:MainKlZetammCase}.
Because the results in this section are somewhat speculative, we will be extremely brief in the presentation, and most of the details of the proofs are left to the reader.
A common element in the manipulations of hypergeometric functions and Mellin-Barnes integrals is that polynomials of the summation/integration variable correspond to derivatives of the defined function, e.g.
\[ \int_{\Re(s)=0} s F(s) z^{-s} ds = -z\frac{d}{dz} \int_{\Re(s)=0} F(s) z^{-s} ds, \]
with a similar statement for hypergeometric series.

For the smooth sums, we consider the formula \eqref{eq:MainKlZetammCase} where $f_{--}(y) = f_X(y)$ has the form \eqref{eq:fXDef}.
We see by the double-Bessel integral \eqref{eq:DblBsslmm} and the bound \cite[(3.1)]{GPSSubconv} \cite[(2.10)]{BFKMM}
\begin{align}
\label{eq:JdBound}
	J_d(y) \ll \Min{\paren{\frac{2y}{d}}^d,d^{-1/3}, (y^2-d^2)^{-1/4}}, \qquad 0 \le d \in \Z
\end{align}
that $\wtilde{f}^d(\mu)$ has super-exponential decay in $d$ past $D_2 = (X_1 X_2)^\epsilon \Max{X_1^{-1/2},X_2^{-1/2}}$, and by the Weyl law (\cref{sect:WeylLaw}), the spectral sum over $d < D_1$ grows about like $D_1^{2+\epsilon}$ so we take $D_1 = (X_1 X_2)^{\theta/2}$.
Then all of the terms of \eqref{eq:MainKlZetammCase} are bounded by $(X_1 X_2)^{\theta+\epsilon}$, except the portion of the long-element terms with weight functions $H^*_{w_l,3}(y)$ and $H^*_{w_l,4}(y)$ given below and only in the case where one of $X_1,X_2<1$; these are precisely the source of the terms $X_1^{-1-\epsilon}$ and $X_2^{-1-\epsilon}$ in \cref{thm:SmoothSums}.

For the unweighted zeta function, we consider the formula \eqref{eq:MainKlZetammCase} with $D_1=3$ and $D_2=\infty$, where $f_{--}(y) = f_M(y) = \prod_{i=1}^2 y_i^s \exp(1-y_i^{1/M})$ and take the limit as $M \to \infty$.
Note that the limit $M \to \infty$ can be conducted in the Mellin domain by realizing the Mellin transform of $\exp(1-x^{1/M})$, which is $e M \Gamma(Ms)$, forms an approximation to the identity along a contour on $\Re(s)=0$ with a semi-circular bump of radius $\frac{1}{M}$ to the right around $s=0$.
We give only the first terms in a sequence of contour shifting and describe how one obtains the more general result, because the higher-order terms corresponding to residues of the beta function
\[ \res_{a=-n} B(a,b) = \frac{(-1)^n}{n!} \frac{\Gamma(b)}{\Gamma(b-n)} = \frac{(-1)^n}{n!} \prod_{i=1}^n (b-i), \]
with $n > 0$ result in derivatives of the various special functions, and this is somewhat complicated to write out.
Again, the main obstruction to the analytic continuation comes from the terms $H^*_{w_l,3}(y)$ and $H^*_{w_l,4}(y)$ as the other terms tend to converge in left half-planes in $s$.
The term $H^*_{w_4,2}(y_1)$ and its equivalent for $w=w_5$ introduce additional complications at $\Re(s_i)=0$.

We work mostly in the Mellin domain, and to start, we note that for $d \ge 2$ and $\Re(r)=0$, we have
\begin{align*}
	\wtilde{f}^d(\mu^d(r)) =& (-1)^d \pi^2 \int_{\Re(s)=(\epsilon,\epsilon)} \hat{f}(r-s_1,-r-s_2) B\paren{s_1-3r,s_2+3r} Q(d,s_1) Q(d,s_2) \frac{ds}{(2\pi i)^2}
\end{align*}
where $Q(d,s)$ is given in \eqref{eq:QBDef}, and
\[ \hat{f}(s_1,s_2) = \int_{Y^+} f_{--}(y) y_1^{s_1} y_2^{s_2} \frac{dy_1 \, dy_2}{y_1 y_2}. \]

\subsection{The trivial term}
The crown jewel would be a formula with the $d$ sums evaluated explicitly, and we can accomplish this for the identity Weyl element term.
The Bessel function and Kloosterman sum for $w=I$ are just
\begin{align*}
	 K_I^d(y;r) =& 1, & S_I(\psi_m,\psi_n,c) =& \piecewise{1 & \If mc=n, c\in V, \\ 0 & \Otherwise.}
\end{align*}

For $\Re(s_1+s_2)<-1$, $\Re(s_i) >-1$, define
\[ F_I(s,r) = \sum_{d=3}^\infty (-1)^d Q(d,s_1) Q(d,s_2) \paren{(\tfrac{d-1}{2})^3-9r^2(\tfrac{d-1}{2})}, \]
then we may evaluate $F_I(s,r)$ as follows:
\begin{align*}
	F_I(s,r) =& \paren{\paren{z\tfrac{d}{dz}}^3-9r^2 z\tfrac{d}{dz}}\left.\paren{-z\pFqStar32{1,1+s_1,1+s_2}{2-s_1,2-s_2}{z}+z^{3/2} \pFqStar32{1,\frac{3}{2}+s_1,\frac{3}{2}+s_2}{\frac{5}{2}-s_1,\frac{5}{2}-s_2}{z}}\right|_{z=1},
\end{align*}
where
\[ \pFqStar32{a,b,c}{d,e}{z} = \frac{\Gamma(a)\Gamma(b)\Gamma(c)}{\Gamma(d)\Gamma(e)}\pFq32{a,b,c}{d,e}{z}. \]
(This expression is valid in the region $-1<\Re(s_i) < -\frac{1}{2}$.)
Then using
\[ \frac{d}{dz} \pFqStar32{a,b,c}{d,e}{z} = \pFqStar32{a+1,b+1,c+1}{d+1,e+1}{z}, \]
and the well-known generalizations of Dixon's formula [Chu (31),(33),Theorem 5] for the so-called ``nearly-poised'' $\pFqName32$ (plus a great deal of algebra), we have
\begin{align*}
	F_I(s,r) =& \frac{1}{8(s_1+s_2)}\paren{\frac{4s_1 s_2+3}{1+s_1+s_2}+36r^2-3}\frac{\Gamma\paren{\frac{3}{2}+s_1}\Gamma\paren{\frac{3}{2}+s_2}}{\Gamma\paren{\frac{3}{2}-s_1}\Gamma\paren{\frac{3}{2}-s_2}} \\
	&-\frac{1}{2(s_1+s_2)}\paren{\frac{s_1 s_2}{1+s_1+s_2}+9r^2}\frac{\Gamma\paren{1+s_1}\Gamma\paren{1+s_2}}{\Gamma\paren{1-s_1}\Gamma\paren{1-s_2}}.
\end{align*}
One can check that the residue at $s_1+s_2=0,-1$ is zero.

Thus the sum
\begin{align}
\label{eq:HIstarEval}
	H_I^*(I) =& \frac{1}{4\pi} \int_{\Re(r)=0} \int_{\Re(s)=0} \hat{f}(r-s_1,-r-s_2) B\paren{s_1-3r,s_2+3r} F_I(s,r) \frac{ds}{(2\pi i)^2} \frac{dr}{2\pi i} \\
	&\qquad - \sum_{d\in [3,D_1)\cup(D_2,\infty)} H_I^d(\wtilde{f}^d; I) \nonumber
\end{align}
is given by a meromorphic function with readily identifiable poles.
We must initially shift the $s$ and $r$ contours to the region of absolute convergence of the $F_I(s,r)$ series, but we see that it has no poles, so there is no trouble shifting things back afterward.
To accommodate $D_2 < \infty$ for the smooth sums in \eqref{eq:HIstarEval}, one may instead apply Stirling's approximation in the form
\begin{align*}
	Q(d,s) =& \sum_{j=0}^N P_j(s) (\tfrac{d-1}{2})^{2s-2j-1}+O(d^{2\Re(s)+6\epsilon-2N-3}),
\end{align*}
for $s \ll d^{\epsilon}$ and $\Re(s)$ in some fixed, compact set and the $P_j(s)$ are some polynomials in $s$; for example,
\begin{align*}
	P_0(s)=&1, & P_1(s)=&-\tfrac{1}{3}s(s-\tfrac{1}{2})(s-1), & \text{etc.}
\end{align*}
Then the $d$ sum can be computed in terms of Hurwitz's zeta function and the factor $(-1)^d$ guarantees the cancellation of the possible poles, so in this case as well, we may shift back to $\Re(s)=0$.

Unfortunately, the $d$-sums in the $w\ne I$ terms also need a little help to converge in a useful region, but we will certainly not be able to evaluate the resulting expressions in $\pFqName43$ (two of which are at $-1$) and $\pFqName54$.

\subsection{The long element term}
After contour shifting, we may write
\begin{align}
\label{eq:HwlstarEval}
	H_{w_l}^*(y) =& H_{w_l,0}^*(y) + \delta_{\varepsilon_2=-1} \paren{H_{w_l,1}^*(y)-H_{w_l,3}^*(y)} + \delta_{\varepsilon_1=-1} \paren{H_{w_l,2}^*(y)+H_{w_l,4}^*(y)},
\end{align}
where $H_{w_l,0}^*(y)$ has the contours of the second Bessel function shifted left
\begin{align*}
	H_{w_l,0}^*(y) =& \pi \sum_{d=D_1}^{D_2} (\varepsilon_1\varepsilon_2)^d \int_{\Re(s)=(\epsilon,\epsilon)} \int_{\Re(r)=0} \int_{\Re(s')=(-1+\epsilon,-1+\epsilon)} \abs{4\pi^2 y_1}^{-r-s'_1} \abs{4\pi^2 y_2}^{r-s'_2} \\
	& \qquad \times \hat{f}(-s) B\paren{s_1-2r,s_2+2r} Q(d,s_1+r) Q(d,s_2-r) \\
	& \qquad \times B^{\varepsilon_1,\varepsilon_2}_{w_l}\paren{s',r} Q(d,s'_1) Q(d,s'_2) \paren{\tfrac{d-1}{2}}\paren{\tfrac{d-1}{2}-3r}\paren{\tfrac{d-1}{2}+3r} \frac{ds'}{(2\pi i)^2} \frac{dr}{2\pi i} \frac{ds}{(2\pi i)^2},
\end{align*}
$H_{w_l,1}^*(y)$ and $H_{w_l,2}^*(y)$ have residues in $s'$ with their $r$ contours shifted appropriately
\begin{align*}
	H_{w_l,1}^*(y) =& \pi \sum_{d=D_1}^{D_2} (\varepsilon_1\varepsilon_2)^d J_{d-1}\paren{4\pi\sqrt{\abs{y_2}}} \int_{\Re(s)=(\epsilon,\epsilon)} \int_{\Re(r)=\frac{1}{4}+\epsilon} \abs{4\pi^2 y_1}^{2r} \abs{4\pi^2 y_2}^r \\
	& \qquad \times \hat{f}(-s) B\paren{s_1-2r,s_2+2r} Q(d,s_1+r) Q(d,s_2-r) \\
	& \qquad \times Q(d,-3r) \paren{\tfrac{d-1}{2}}\paren{\tfrac{d-1}{2}-3r}\paren{\tfrac{d-1}{2}+3r} \frac{dr}{2\pi i} \frac{ds}{(2\pi i)^2} \\
	H_{w_l,2}^*(y) =& \pi \sum_{d=D_1}^{D_2} (\varepsilon_1\varepsilon_2)^d J_{d-1}\paren{4\pi\sqrt{\abs{y_1}}} \int_{\Re(s)=(\epsilon,\epsilon)} \int_{\Re(r)=-\frac{1}{4}-\epsilon} \abs{4\pi^2 y_2}^{-2r} \abs{4\pi^2 y_1}^{-r} \\
	& \qquad \times \hat{f}(-s) B\paren{s_1-2r,s_2+2r} Q(d,s_1+r) Q(d,s_2-r) \\
	& \qquad \times Q(d,3r) \paren{\tfrac{d-1}{2}}\paren{\tfrac{d-1}{2}-3r}\paren{\tfrac{d-1}{2}+3r} \frac{dr}{2\pi i} \frac{ds}{(2\pi i)^2},
\end{align*}
and $H_{w_l,3}^*(y)$ and $H_{w_l,4}^*(y)$ have the residues in $r$
\begin{align*}
	H_{w_l,3}^*(y) =& \frac{\pi}{2} \sum_{d=D_1}^{D_2} (\varepsilon_1\varepsilon_2)^d \paren{\tfrac{d-1}{2}} J_{d-1}\paren{4\pi\sqrt{\abs{y_2}}} \int_{\Re(s)=(\epsilon,\epsilon)} \hat{f}(-s) \\
	& \qquad \times \abs{64\pi^6 y_1^2 y_2}^{\frac{s_1}{2}} Q(d,s_2-\tfrac{1}{2}s_1) \frac{ds}{(2\pi i)^2} \\
	H_{w_l,4}^*(y) =& \frac{\pi}{2} \sum_{d=D_1}^{D_2} (\varepsilon_1\varepsilon_2)^d \paren{\tfrac{d-1}{2}} J_{d-1}\paren{4\pi\sqrt{\abs{y_1}}} \int_{\Re(s)=(\epsilon,\epsilon)} \hat{f}(-s) \\
	& \qquad \times \abs{64\pi^6 y_1 y_2^2}^{\frac{s_2}{2}} Q(d,s_1-\tfrac{1}{2}s_2) \frac{ds}{(2\pi i)^2}.
\end{align*}
We have used
\begin{align}
\label{eq:BesselJMB}
	J_{d-1}(2\sqrt{x}) =& \int_{\Re(s)=-\epsilon} Q(d,s) x^{-s} \frac{ds}{2\pi i},
\end{align}
and
\[ Q(d,s)Q(d,-s) \paren{\tfrac{d-1}{2}+s} \paren{\tfrac{d-1}{2}-s} = 1, \]
and $B^{\varepsilon_1,\varepsilon_2}_{w_l}\paren{s,r}$ is given in \eqref{eq:QBDef}.
The presence of the $J$-Bessel function recovers the super-exponential convergence of the $d$ sum, and this holds even for the Kloosterman zeta function (which drops the $s$ integrals).

We note that Mellin inversion gives a slightly prettier formula
\begin{align*}
	H_{w_l,3}^*(y) =& \frac{\pi}{2} \sum_{d=D_1}^{D_2} (\varepsilon_1\varepsilon_2)^d \paren{\tfrac{d-1}{2}} J_{d-1}\paren{4\pi\sqrt{\abs{y_2}}} \int_0^\infty J_{d-1}(2\sqrt{t}) f\paren{8\pi^3 \sqrt{\abs{y_1^2 y_2/t}},t} \frac{dt}{t} \\
	H_{w_l,4}^*(y) =& \frac{\pi}{2} \sum_{d=D_1}^{D_2} (\varepsilon_1\varepsilon_2)^d \paren{\tfrac{d-1}{2}} J_{d-1}\paren{4\pi\sqrt{\abs{y_1}}} \int_0^\infty J_{d-1}(2\sqrt{t}) f\paren{t,8\pi^3 \sqrt{\abs{y_1 y_2^2/t}}} \frac{dt}{t}.
\end{align*}

As mentioned above, shifting the $s'$ contours farther results in polynomials in $s'_1$ or $s'_2$, provided we remove a few initial terms from the $d$ sum (equivalently, we may increase $D_1$), and this yields derivatives of the function $J_{d-1}\paren{4\pi\sqrt{\abs{y_i}}}$, which may be expressed in terms of the shifts $J_{d-1\pm n}$.

\subsection{The $w_4$ and $w_5$ terms}
The $w_4$ and $w_5$ Bessel functions and Kloosterman sums are symmetric
\begin{align*}
	 K_{w_4}^d((y_1,1);\mu(r)) =& \frac{(\varepsilon_1 i)^d }{4\pi^2} \int_{-i\infty}^{+i\infty} \abs{8\pi^3 y_1}^{1-r-s} Q(d,s) \Gamma\paren{s+3r} \exp\paren{\varepsilon_1 \tfrac{\pi i}{2}(s+3r)} \frac{ds}{2\pi i} \\
	K_{w_5}^d(y;r) =& K_{w_4}^d(\vpmpm{-+} y^\iota;-r), \\
	\abs{S_{w_4}(\psi_m,\psi_n,c)} \le& \delta_{m_2 c_1 = n_1 c_2^2} d(c_1) (\abs{m_2}, \abs{n_2}, c_2) c_1, \\
	S_{w_5}(\psi_m,\psi_n,c) =& S_{w_4}(\psi_{m_2,m_1},\psi_{n_2,-n_1},(c_2,c_1)),
\end{align*}
and we just give the decomposition for $w=w_4$.

The contour shifting yields
\begin{gather}
\label{eq:Hw4starEval}
	H_{w_4}^*(y_1,1) = H_{w_4,0}^*(y_1) + H_{w_4,1}^*(y_1) + H_{w_4,2}^*(y_1), \\
\begin{aligned}
	H_{w_4,0}^*(y_1) =& \frac{1}{4} \sum_{d=D_1}^{D_2} (-\varepsilon_1 i)^d (d-1) \int_{\Re(r)=0} \int_{\Re(s)=(\epsilon,\epsilon)} \int_{\Re(s')=-1+\epsilon}\abs{8\pi^3 y_1}^{-r-s'} \\
	& \qquad \times \hat{f}(-s) Q(d,s_1+r) Q(d,s_2-r) Q(d,s') B\paren{s_1-2r,s_2+2r} \\
	& \qquad \times \Gamma\paren{s'+3r} \exp\paren{\varepsilon_1 \tfrac{\pi i}{2}(s'+3r)} \frac{ds}{(2\pi i)^3} \paren{\paren{\tfrac{d-1}{2}}^2-9r^2} \frac{dr}{2\pi i}, \\
	H_{w_4,1}^*(y_1) =& \frac{1}{4} \sum_{d=D_1}^{D_2} (-\varepsilon_1 i)^d (d-1) \int_{\Re(r)=\frac{1}{3}-\epsilon} \int_{\Re(s)=(\epsilon,\epsilon)} \abs{8\pi^3 y_1}^{2r} \\
	& \qquad \times \hat{f}(-s) Q(d,s_1+r) Q(d,s_2-r) Q(d,-3r) B\paren{s_1-2r,s_2+2r} \\
	& \qquad \times \frac{ds}{(2\pi i)^2} \paren{\paren{\tfrac{d-1}{2}}^2-9r^2} \frac{dr}{2\pi i}, \\
	H_{w_4,2}^*(y_1) =& \frac{1}{8} \sum_{d=D_1}^{D_2} (-\varepsilon_1 i)^d (d-1) \int_{\Re(s)=(\epsilon,-\tfrac{1}{2}-\epsilon)} \abs{8\pi^3 y_1}^{s_1} \hat{f}(-s) Q(d,s_2-\tfrac{1}{2}s_1) \frac{ds'}{2\pi i} \frac{ds}{(2\pi i)^2}.
\end{aligned} \nonumber
\end{gather}

Again, Mellin inversion gives
\begin{align*}
	H_{w_4,2}^*(y_1) =& \frac{1}{8} \sum_{d=D_1}^{D_2} (-\varepsilon_1 i)^d (d-1) \int_0^\infty J_{d-1}\paren{2\sqrt{t}} f(8\pi^3\abs{y_1}/t^{1/2},t) \frac{dt}{t},
\end{align*}
but for the Kloosterman zeta function (which drops the $s$ integrals), we do not have the fast convergence of the $d$ sum, so we handle this like the trivial term:
For $\Re(s)<-\tfrac{1}{2}$, we have
\begin{align*}
	F_{w_4}(\varepsilon_1, s) :=& \sum_{d=3}^\infty (-\varepsilon_1 i)^d (d-1) Q(d,s) = \frac{\Gamma\paren{\frac{3}{2}+s}}{\Gamma\paren{\frac{3}{2}-s}}+i\varepsilon_1 \frac{\Gamma\paren{1+s}}{\Gamma\paren{1-s}},
\end{align*}
so we may also write
\begin{align*}
	H_{w_4,2}^*(y_1) =& \frac{1}{8} \int_{\Re(s)=(\epsilon,-\tfrac{1}{2}-\epsilon)} \abs{8\pi^3 y_1}^{s_1} \hat{f}(-s) \Biggl(F_{w_4}(\varepsilon_1, s_2-\tfrac{1}{2}s_1) \\
	& \qquad -\sum_{d\in [3,D_1)\cup(D_2,\infty)} (-\varepsilon_1 i)^d (d-1) Q(d,s_2-\tfrac{1}{2}s_1)\Biggr) \frac{ds}{(2\pi i)^2}.
\end{align*}
The evaluation of the $d$ sum goes by well-known generalizations of Kummer's theorem to nearly-poised series (equation (6) and the following display in \cite{LavGrondRath}); the factor $(d-1)$ results in a derivative which is treated by index shifts.

The $s_2$ contour here is shifted to accommodate $D_2 < \infty$; for the Kloosterman zeta function, we use $D_2=\infty$ so this may be taken at $\Re(s_2)=\epsilon$.
For the smooth sums, one may apply Stirling's approximation and express the $d$ sum in terms of Hurwitz' zeta function as in the case $w=I$; the factor $(-\varepsilon_1 i)^d$ guarantees the cancellation of the possible poles, so again, we may shift up to $\Re(s_2)=\epsilon$.

Shifting the $s'$ and $r$ contours farther yields polynomials in $s_1$, but importantly not in $d$, provided we remove some initial terms of the $d$ sum (equivalently, we may increase $D_1$), so our evaluation of that sum is safe.

\section{Bounds for the $GL(2)$ Whittaker functions}
\label{sect:WhittBds}

We prove Lemmas \ref{lem:GL2WhittBd}-\ref{lem:GL2BBd}.
A common theme in the proof of the exponential decay bounds is an argument of the form:
$\frac{y}{\log y}$ is increasing for $y > e$ and so for $a \ge 5$ and $y > a \log^2 a$,
\[ \frac{y}{\log y} > \frac{a \log^2 a}{\log a+2\log \log a} > a. \]

We also make two references to (generalized) hypergeometric series
\[ \pFq{p}{q}{a_1,\ldots a_p}{b_1,\ldots,b_q}{z} := \sum_{k=0}^\infty \frac{(a_1)_k\cdots(a_p)_k}{k!\,(b_1)_k\cdots (b_q)_k} z^k, \]
where
\[ (a)_k := \prod_{j=0}^{k-1} (a+j) = \frac{\Gamma(a+k)}{\Gamma(a)} = (-1)^k \frac{\Gamma(1-a)}{\Gamma(1-a-k)} \]
is the rising Pochhammer symbol.
When $p=q+1$, there is often some difficulty with branch cuts along $z \in (1,\infty)$, but we will always have $a_1\in-\N$ so that the series terminates (and hence defines a polynomial in $z$).
We will be using these strictly for their arithmetic with respect to shifting indices.

\subsection{The direct function, at integral indices}
We now prove \cref{lem:GL2WhittBd}.

\subsubsection{The decay near zero and infinity.}
We prove the second and third parts first.
It follows from \eqref{eq:classWhittDef} and \cite[13.14.9]{DLMF},
\begin{align*}
	W_{\frac{d}{2}+N,\frac{d-1}{2}}(y) =& e^{-y/2} \sum_{k=0}^N \frac{N! \, (d-1+N)! (-1)^{N-k}}{k! \, (N-k)! \, (d-1+k)!} y^{\frac{d}{2}+k},
\end{align*}
that
\begin{align}
\label{eq:FdNySum}
	\wtilde{\mathcal{W}}_{d,N}(y) =& 2\pi e^{-y/2} \sum_{k=0}^N \frac{(-1)^{N-k} \sqrt{N! \, (d-1+N)!}}{k! \, (N-k)! \, (d-1+k)!} y^{\frac{d-1}{2}+k}.
\end{align}

The absolute ratio of successive terms in \eqref{eq:FdNySum} is
\[ \frac{(N-k)y}{(k+1)(d+k)}, \]
and clearly this is a decreasing function of $k$.
If $y<\frac{d}{N+1}$, then the first term is the sum dominates, and we have
\[ \abs{\wtilde{\mathcal{W}}_{d,N}(y)} \le 2\pi (N+1) \frac{e^{-y/2} y^{\frac{d-1}{2}}}{\sqrt{(d-1)!}} \sqrt{\binom{d-1+N}{d-1}} \le 2\pi (N+1) \frac{e^{-y/2} ((N+1)y)^{\frac{d-1}{2}}}{\sqrt{(d-1)!}}, \]
using $\binom{d-1+N}{d-1} \le (N+1)^{d-1}$.
\Cref{lem:GL2WhittBd}.2 follows from $\frac{d-1}{e(N+1)}<\frac{d}{N+1}$ and Stirling's formula.
(This would fail if we tried to pull out a power of $y$ that was greater than $\frac{d-1}{2}$, so the largest power that works for all $d\ge2$ is $\frac{2-1}{2}=\frac{1}{2}$.)

On the other hand, if $y > N(d-1+N)$ then the last term dominates and we have
\[ \abs{\wtilde{\mathcal{W}}_{d,N}(y)} \le 2\pi (N+1) \frac{e^{-y/2} y^{\frac{d-1}{2}+N}}{\sqrt{N! (d-1+N)!}}. \]
If also $y > 2(d+3+2N) \log^2(d+3+2N)$, then
\[ e^{-y/2} y^{\frac{d-1}{2}+N} \le e^{-y/4} \exp\paren{\frac{\log y}{2}\paren{-\frac{y/2}{\log y}+(d-1+2N)}} \le e^{-y/4}, \]
and \cref{lem:GL2WhittBd}.3 follows.

\subsubsection{The generic bound}
\Cref{lem:GL2WhittBd}.1 is more difficult and we use the bounds of Olver and Dunster as they appear in \cite[sections 13.20(iv),13.21(ii),13.21(iii)]{DLMF}.
There is, perhaps, something of the absurd in applying asymptotics to derive upper bounds for the same functions, but the author was unable to locate the necessary bounds in the literature.
The result is fairly simple for $N=0$ (use \eqref{eq:FdNySum}, Stirling's formula, and basic calculus), so assume $N \ge 1$.

Hopefully the reader will forgive the lack of introduction for the various special functions used here, but they are standard functions that will only appear in this section.
The notation $U(\nu,x)$ is for the parabolic cylinder function, $J_\nu(x)$ is the standard $J$-Bessel function, and $\Ai(x)$ is the Airy function; see \cite[section 9,10,12]{DLMF}.
The $\operatorname{env}$ notation is used to avoid the zeros of functions in their oscillating regions, and should be viewed as the modulus of the oscillation; for instance $\operatorname{envcos}(x)=\operatorname{envsin}(x)=1$.
Outside of the oscillating region, $\operatorname{env}$ is the usual absolute value, up to an absolute constant; see, e.g. \cite[section 2.8iii]{DLMF}.

We recall the necessary equations from \cite{DLMF}.
Let
\begin{align*}
	H_{\kappa,\mu}(x) =& \frac{2\pi}{\sqrt{\Gamma\paren{\kappa+\mu+\frac{1}{2}}\Gamma\paren{\kappa-\mu+\frac{1}{2}}}} x^{-\frac{1}{2}} W_{\kappa,\mu}(x), \\
	\wtilde{H}_{\kappa,\mu}(x) =& \frac{2\pi}{\Gamma(2\mu+1)} \sqrt{\frac{\Gamma\paren{\kappa+\mu+\frac{1}{2}}}{\Gamma\paren{\kappa-\mu+\frac{1}{2}}}} x^{-\frac{1}{2}} M_{\kappa,\mu}(x),
\end{align*}
where $M_{\kappa,\mu}(x)$ is the classical $M$-Whittaker function.

For $x > 0$ and $\kappa\ge\mu>0$, define
\begin{align*}
	\hat{x}=&\frac{x}{\kappa}, & \hat{\mu}=&\frac{\mu}{\kappa}, & x_\pm=&2(1\pm\sqrt{1-\hat{\mu}^2}),
\end{align*}
\begin{align*}
	\alpha=&\sqrt{1-\hat{\mu}}, & X=&\sqrt{\abs{\hat{x}^2-4\hat{x}+4\hat{\mu}^2}},
\end{align*}
\begin{align*}
	\Psi_1=& \frac{\alpha^2(\zeta_1^2-1)}{\hat{x}^2-4\hat{x}+4\hat{\mu}^2}, &
	\Psi_2 =& \frac{\hat{\mu}^2 (1-\zeta_2)}{\hat{x}^2-4\hat{x}+4\hat{\mu}^2}, &
	\Psi_3 =& \frac{\zeta_3}{\hat{x}^2-4\hat{x}+4\hat{\mu}^2},
\end{align*}
\begin{align*}
	f_A(\hat{\mu},\hat{x}) =& \frac{1}{2}X+\hat{\mu}\log\paren{\frac{\hat{x}\sqrt{1-\hat{\mu}^2}}{\abs{2\hat{\mu}^2-\hat{x}+\hat{\mu} X}}}+\log\paren{\frac{2\sqrt{1-\hat{\mu}^2}}{\abs{2-\hat{x}-X}}}, \\
	f_B(\hat{\mu},\hat{x}) =& \frac{1}{2}X-\hat{\mu}\arctan\paren{\frac{\hat{x}-2\hat{\mu}^2}{\hat{\mu} X}}+\arctan\paren{\frac{\hat{x}-2}{X}}, \\
	f(\hat{\mu},\hat{x}) =& \piecewise{
		f_A(\hat{\mu},\hat{x})-\frac{\pi}{2}\alpha^2 & \If 0 < \hat{x} \le x_-, \\
		f_B(\hat{\mu},\hat{x}) & \If x_- \le \hat{x} \le x_+, \\
		f_A(\hat{\mu},\hat{x})+\frac{\pi}{2}\alpha^2 & \If \hat{x} \ge x_+.}
\end{align*}

When $\delta_1 \le \hat{\mu} < 1$, for some fixed $0<\delta_1 <1$, we may combine \cite[13.20.16 and 13.20.18]{DLMF} into:
\begin{align}
\label{eq:DLMF13.20iv}
	H_{\kappa,\mu}(x) =& \frac{2^{5/4} \pi^{3/4}}{\sqrt{\Gamma\paren{\kappa-\mu+\frac{1}{2}}}} \Psi_1^{1/4} \bigl(U(\mu-\kappa,2\alpha\zeta_1\sqrt{\kappa}) \\
	& \qquad +\envU(\mu-\kappa,2\alpha\abs{\zeta_1}\sqrt{\kappa}) O\paren{\mu^{-2/3}}\bigr), \nonumber
\end{align}
with $\zeta_1=\zeta_1(\hat{\mu},\hat{x})$ defined by
\begin{align*}
	\alpha^{-2} f(\hat{\mu},\hat{x}) = g_1(\zeta_1) :=& \zeta_1\sqrt{\abs{\zeta_1^2-1}}+\piecewise{
		\displaystyle \arccosh\paren{-\zeta_1} -\frac{\pi}{2} & \If \zeta_1\le-1, \\
		\displaystyle \arcsin\paren{\zeta_1} & \If -1\le\zeta_1\le1, \\
		\displaystyle \frac{\pi}{2}-\arccosh\paren{\zeta_1} & \If \zeta_1\ge1.}
\end{align*}

When $\hat{\mu} \le 1-\delta_2$ and $\hat{x}\in(0,(1-\delta_2)x_+)$, for some fixed $0<\delta_2 <1$, we may apply \cite[13.21.13]{DLMF}:
\begin{align}
\label{eq:DLMF13.21ii}
	\wtilde{H}_{\kappa,\mu}(x) =& 2^{3/2}\pi \Psi_2^{1/4} \bigl(J_{2\mu}(2\mu\sqrt{\zeta_2}) +\envJ_{2\mu}(2\mu\sqrt{\zeta_2}) O\paren{\kappa^{-1}}\bigr),
\end{align}
with $\zeta_2=\zeta_2(\hat{\mu},\hat{x})$ defined by
\begin{align*}
	\frac{1}{2\hat{\mu}}\paren{f(\hat{\mu},\hat{x})+\frac{\pi}{2}\alpha^2} = g_2(\zeta_2) :=& \sqrt{\abs{\zeta_2-1}} - \piecewise{
		\displaystyle \frac{1}{2}\log\paren{\frac{1+\sqrt{1-\zeta_2}}{1-\sqrt{1-\zeta_2}}} & \If 0 < \zeta_2 \le 1, \\
		\displaystyle \arctan\sqrt{\zeta_2-1} & \If \zeta_2\ge1.}
\end{align*}
Note: This corrects a discrepancy between [DLMF (13.21.12)] and the original [Dunster (3.7)] for $\hat{x} > 2$, which is attainable when $\delta_2 < \frac{1}{2}$ for small $\hat{\mu}$.

When $\hat{\mu} \le 1-\delta_3$ and $\hat{x}\in((1+\delta_3)x_-,\infty)$, for some fixed $0<\delta_3 <1$, we may apply \cite[13.21.23]{DLMF}:
\begin{align}
\label{eq:DLMF13.21iii}
	H_{\kappa,\mu}(x) =& \sqrt{2} \kappa^{-1/3} \Psi_3^{1/4} \paren{\Ai\paren{\kappa^{2/3}\zeta_3}+\envAi\paren{\kappa^{2/3}\zeta_3} O\paren{\kappa^{-1}}},
\end{align}
with $\zeta_3=\zeta_3(\hat{\mu},\hat{x})$ defined by
\begin{align*}
	f(\hat{\mu},\hat{x})-\frac{\pi}{2}\alpha^2 = g_3(\zeta_3) :=& \piecewise{
		\displaystyle -\frac{2}{3}(-\zeta_3)^{3/2} & \If \zeta_3 \le 0, \\
		\displaystyle \frac{2}{3}\zeta_3^{3/2} & \If \zeta_3 \ge 0.}
\end{align*}

Notice that if $\hat{\mu} \le 1-\delta$ then
\[ (1-\delta)x_+ \ge (1-\delta)2(1+\sqrt{(2-\delta)\delta}) > 2(1-\delta), \]
\[ (1+\delta)x_- \le (1+\delta)2(1-\sqrt{(2-\delta)\delta}) < 2(1-\delta), \]
so the $\hat{x}$ intervals in \eqref{eq:DLMF13.21ii} and \eqref{eq:DLMF13.21iii} always overlap.
We choose $\delta_1=\delta_2=\delta_3=\frac{1}{2}$ and separate \eqref{eq:DLMF13.21ii} from \eqref{eq:DLMF13.21iii} at $\hat{x}=1$.

The functions $f,g_i$ and $\zeta_i$ are all increasing and differentiable in $\hat{x}$ (except for the $g_i$ and $f$ functions at $\hat{x}=x_\pm$, as these are all multiplied by $-i$ on the range $x_-<\hat{x}<x_+$, but this cancels in the definition of the $\zeta_i$), and we have
\begin{align*}
	f(\hat{\mu},0)=& -\infty, & f(\hat{\mu},x_-)=&-\tfrac{\pi}{2}\alpha^2, & f(\hat{\mu},x_+)=&\tfrac{\pi}{2}\alpha^2 & f(\hat{\mu},\infty)=&\infty, \\
	\zeta_1(\hat{\mu},0) =& -\infty, & \zeta_1(\hat{\mu},x_-) =& -1, & \zeta_1(\hat{\mu},x_+) =& 1, & \zeta_1(\hat{\mu},\infty) =& \infty, \\
	\zeta_2(\hat{\mu},0) =&0, & \zeta_2(\hat{\mu},x_-) =& 1 & \hat{\mu}^2\zeta_2(\hat{\mu},1) \in& (\tfrac{1}{2},1), \\
	& & \zeta_3(\hat{\mu},1) \in& (-2,-1) & \zeta_3(\hat{\mu},x_+) =& 0 & \zeta_3(\hat{\mu},\infty) =& \infty.
\end{align*}
The zeros in the denominators of the $\Psi_i$ are cancelled by zeros in the numerators (in the region of definition in $\hat{x}$), and the result is $\Psi_i \asymp 1$, independent of $\hat{\mu}$ for $\hat{x} \asymp 1$.
As $\hat{x} \to \infty$,
\[ \Psi_1 \sim \frac{1}{2\hat{x}}, \qquad \Psi_3 \sim \paren{\frac{3}{4}}^{2/3} \hat{x}^{-4/3}, \]
and as $\hat{x} \to 0$,
\[ \Psi_1\sim\frac{1}{4\hat{\mu}} \log(1/\hat{x}), \qquad \Psi_2\sim\frac{1}{4}. \]

Returning to the proof of \cref{lem:GL2WhittBd}.1, notice that if $\kappa=\frac{d}{2}+N$ and $\mu=\frac{d-1}{2}$, then
\[ \wtilde{\mathcal{W}}_{d,N}(x) = H_{\kappa,\mu}(x) = (-1)^N \wtilde{H}_{\kappa,\mu}(x), \]
using \cite[13.14.33]{DLMF}; we now assume $\kappa$ and $\mu$ are of this form.

\noindent\emph{Case I}. Suppose $\frac{1}{2} \le \hat{\mu} < 1$.

The connection \cite[12.7.2]{DLMF} with Hermite functions and Indritz's bound \cite[18.14.9]{DLMF} for the Hermite functions implies $\abs{U(-N-\frac{1}{2},x)} \le \sqrt{N!}$, so on this region, we have
\[ \wtilde{\mathcal{W}}_{d,N}(x) \ll \paren{(\kappa/\mu)\log(3+(\kappa/x))}^{1/4}. \]

\noindent\emph{Case II}. Suppose $\hat{\mu} \le \frac{1}{2}$ and $0 < \hat{x} \le 1$.

We have \cite[10.14.1]{DLMF}
\[ \abs{J_\nu(x)} \le 1, \qquad \nu \ge 0, x\in\R\]
so on this region
\[ \wtilde{\mathcal{W}}_{d,N}(x) \ll 1. \]

\noindent\emph{Case III}. Suppose $\hat{\mu} \le \frac{1}{2}$ and $\hat{x} \ge 1$.

From \cite[9.7.5 and 9.7.9]{DLMF}, we have
\[ \abs{\Ai(x)} \ll (1+\abs{x})^{-1/4} \exp\paren{-\tfrac{2}{3} \Max{x,0}^{3/2}}, \]
so on this region
\[ \wtilde{\mathcal{W}}_{d,N}(x) \ll \kappa (\kappa+x)^{-4/3}. \]

\subsection{The Mellin transform, at integral indices}
We now prove \cref{lem:GL2WhittMellinBd}.2.

Directly from \eqref{eq:FdNySum}, we have
\begin{align*}
	\what{\mathcal{W}}_{d,N}(s) =& 2\pi\sum_{k=0}^N \frac{(-1)^{N-k} \sqrt{N! \, (d-1+N)!}}{k! \, (N-k)! \, (d-1+k)!} 2^{\frac{d-1}{2}+k+s} \Gamma\paren{\tfrac{d-1}{2}+k+s},
\end{align*}
and this may be written in terms of a terminating hypergeometric series as
\[ \what{\mathcal{W}}_{d,N}(s) = (-1)^N 2^{\frac{d+1}{2}+s} \pi \frac{\Gamma\paren{\frac{d-1}{2}+s}}{(d-1)!} \sqrt{\frac{(d-1+N)!}{N!}} \pFq21{-N,\frac{d-1}{2}+s}{d}{2}. \]

The recursion \cite[9.137.2]{GradRyzh}
\[ \pFq21{a-1,b}{c}{z} = \frac{a(z-1)}{a-c}\pFq21{a+1,b}{c}{z}-\frac{c-2a+(a-b)z}{a-c}\pFq21{a,b}{c}{z} \]
for the hypergeometric function implies the recursion
\begin{align*}
	\what{\mathcal{W}}_{d,N+1}(s) =& \sqrt{\frac{N(d+N-1)}{(N+1)(d+N)}} \what{\mathcal{W}}_{d,N-1}(s) + \frac{2s-1}{\sqrt{(N+1)(d+N)}} \what{\mathcal{W}}_{d,N}(s).
\end{align*}

From the base cases
\begin{align*}
	\wtilde{\mathcal{W}}_{d,0}(s) =& 2^{\frac{d+1}{2}+s} \pi \frac{\Gamma\paren{\tfrac{d-1}{2}+s}}{\sqrt{(d-1)!}} & \wtilde{\mathcal{W}}_{d,1}(s) =& \frac{2s-1}{\sqrt{d}} \wtilde{\mathcal{W}}_{d,0}(s),
\end{align*}
and induction, we may deduce
\[ \abs{\what{\mathcal{W}}_{d,N}(s)} \le \abs{\what{\mathcal{W}}_{d,0}(s)}\prod_{i=1}^N \paren{1+\frac{\abs{2s-1}}{\sqrt{i(d-1+i)}}} \le (N+2)^2 (1+\abs{s-\tfrac{1}{2}})^N 2^{d/2} \frac{\abs{\Gamma\paren{\frac{d-1}{2}+s}}}{\sqrt{(d-1)!}}. \]

We use
\[ 2^{d/2} \frac{\abs{\Gamma\paren{\frac{d-1}{2}+s}}}{\sqrt{(d-1)!}} \ll (1+\abs{s-\tfrac{1}{2}})^{\frac{d-2}{2}} \abs{\Gamma\paren{\tfrac{1}{2}+s}}, \]
and for $1 \ge \Re(s)+\frac{1}{2} \ge \delta > 0$, Stirling's formula gives
\[ \abs{\what{\mathcal{W}}_{d,N}(s)} \ll_\delta (N+2)^2 (1+\abs{s-\tfrac{1}{2}})^{N+\frac{d-2}{2}} \abs{\tfrac{1}{2}+s}^{\frac{1}{2}} \exp\paren{-\tfrac{\pi}{2}\abs{\Im(s)}}. \]

Then cutting $\tfrac{\pi}{2}\abs{\Im(s)}$ into four equal pieces (three to remove the leading factors and one left over), \cref{lem:GL2WhittMellinBd}.2 follows from the fact that
\[ \frac{\log^2 x}{\log(2+x\log^2 x)} > 1 \]
for $x > 6$, and
\[ \frac{16}{\pi}\log(N+2) < \frac{4}{\pi} (5+d+2N)\log^2(5+d+2N). \]

\subsection{The Mellin transform, near the imaginary axis}
We now prove \cref{lem:GL2BBd} parts two and three.

Since the $\mathcal{B}_{\varepsilon, m}(a,b)$ function satisfies the symmetry
\[ \mathcal{B}_{\varepsilon, -m}(a,b) = \varepsilon \mathcal{B}_{\varepsilon, m}(a,b), \]
we will assume throughout that $m \ge 0$, and we define $\delta_1,\delta_2,\delta_3\in\set{0,1}$ by $(-1)^{\delta_1}=\varepsilon$, $\delta_2\equiv m+\delta_1\pmod{2}$, and $\delta_3\equiv m\pmod{2}$.

\subsubsection{The recursion formula in $b$}
Writing $\mathcal{B}_{\varepsilon, m}(a,b)$ as a terminating hypergeometric series, we have
\[ \mathcal{B}_{\varepsilon, m}(a,b) = (i m)^{\delta_1} \frac{\Gamma\paren{\frac{\delta_1+a}{2}}\Gamma\paren{\frac{m-\delta_1+b}{2}}}{\Gamma\paren{\frac{m+a+b}{2}}} \pFq32{\tfrac{\delta_1+a}{2},\tfrac{1-m}{2},\delta_1-\tfrac{m}{2}}{\delta_1+\tfrac{1}{2},\tfrac{2+\delta_1-m-b}{2}}{1}. \]

The hypergeometric function satisfies a rather complicated recurrence relation, which simplifies somewhat at $z=1$ \cite[(21)]{WilsonTTR}
\begin{align*}
	a_1 a_2 a_3 F(0)=&b_1(b_2-a_1)(b_2-a_2)(b_2-a_3)(F(1)-F(0))\\
	& \qquad +b_2(b_2-1)(b_1+b_2-a_1-a_2-a_3-1)(F(-1)-F(0)),\\
	F(j) :=& \pFq32{a_1,a_2,a_3}{b_1,b_2+j}{z}.
\end{align*}
This implies for the $\mathcal{B}_{\varepsilon, m}(a,b)$ function that
\begin{align}
\label{eq:Bfunbrecur}
	 \mathcal{B}_{\varepsilon, m}(a,b) =& \frac{4+3a+5b+2ab+2b^2-m^2}{b(1+b)} \mathcal{B}_{\varepsilon, m}(a,b+2)\\
	 & \qquad -\frac{(2+a+b-m)(2+a+b+m)}{b(1+b)} \mathcal{B}_{\varepsilon, m}(a,b+4). \nonumber
\end{align}

\subsubsection{The generic bound}
Directly from \eqref{eq:BetaFunDef} and \cite[(2.27)]{HWI}, we see
\begin{align*}
	\abs{\mathcal{B}_{\varepsilon, m}(a,b)} \le& B\paren{\tfrac{\Re(a)}{2},\tfrac{\Re(b)}{2}},
\end{align*}
where $B(a,b)$ is again the usual beta function.
\Cref{lem:GL2BBd}.2 follows from this and the recursion on $b$ given in \eqref{eq:Bfunbrecur}.

\subsubsection{The exponential decay bound}
Using $\sqrt{\frac{a}{2}}+\sqrt{\frac{b}{2}} \le \sqrt{a+b}$, it's easy to modify the proof of \cite[(2.33)]{HWI} to see
\[ \abs{\mathcal{B}_{\varepsilon, m}(a,b)} \le \paren{2m+\abs{a}+\abs{b}}^{m/2} \abs{\frac{\Gamma\paren{\frac{\delta_1+a}{2}}\Gamma\paren{\frac{\delta_2+b}{2}}}{\Gamma\paren{\frac{m+a+b}{2}}}}. \]

Suppose $10 \ge \Re(a),\Re(b)\ge\eta>0$ and $\abs{\Im(a-b)} > \Max{10,m,\abs{\Im(a+b)}}$, then
\[ \Gamma\paren{\frac{m+a+b}{2}} \gg \Gamma\paren{\frac{\delta_3+a+b}{2}}, \]
and we note that
\[ \abs{\Im(a)}+\abs{\Im(b)}-\abs{\Im(a+b)} = \abs{\Im(a-b)}-\abs{\Im(a+b)}. \]
Then Stirling's formula gives
\[ \abs{\mathcal{B}_{\varepsilon, m}(a,b)} \ll_\eta \abs{6\Im(a-b)}^{\frac{m+\Re(a+b)+\delta_1+\delta_2-2}{2}} \abs{\delta_3+a+b}^{\frac{1}{2}} \exp-\tfrac{\pi}{4}\paren{\abs{\Im(a-b)}-\abs{\Im(a+b)}}. \]

We split $\tfrac{\pi}{4}\abs{\Im(a-b)}$ into four equal parts, and then \cref{lem:GL2BBd}.3 in the region $\Re(b) > 0$ follows from
\[ \frac{\log^2 x}{\log(6x\log^2 x)} > 1 \]
for $x>12$.
The bound on the region $-1<\Re(b) < 0$ follows from the recursion on $b$ given in \eqref{eq:Bfunbrecur} together with $\abs{a}\asymp\abs{b}\asymp\abs{a-b}$.

\bibliographystyle{amsplain}

\bibliography{HigherWeight}

\def\cprime{$'$}
\begin{bibdiv}
\begin{biblist}

\bib{Arthur01}{incollection}{
      author={Arthur, James},
       title={The trace formula for reductive groups},
        date={1983},
   booktitle={Conference on automorphic theory ({D}ijon, 1981)},
      series={Publ. Math. Univ. Paris VII},
      volume={15},
   publisher={Univ. Paris VII, Paris},
       pages={1\ndash 41},
}

\bib{Bailey}{book}{
      author={Bailey, W.~N.},
       title={Generalized hypergeometric series},
      series={Cambridge Tracts in Mathematics and Mathematical Physics, No.
  32},
   publisher={Stechert-Hafner, Inc., New York},
        date={1964},
}

\bib{Baruch01}{article}{
      author={Baruch, Ehud~Moshe},
       title={On {B}essel distributions for quasi-split groups},
        date={2001},
        ISSN={0002-9947},
     journal={Trans. Amer. Math. Soc.},
      volume={353},
      number={7},
       pages={2601\ndash 2614},
}

\bib{BaruchMao01}{article}{
      author={Baruch, Ehud~Moshe},
      author={Mao, Zhengyu},
       title={Bessel identities in the {W}aldspurger correspondence over the
  real numbers},
        date={2005},
        ISSN={0021-2172},
     journal={Israel J. Math.},
      volume={145},
       pages={1\ndash 81},
}

\bib{Val01}{article}{
      author={Blomer, Valentin},
       title={Applications of the {K}uznetsov formula on {$GL(3)$}},
        date={2013},
     journal={Invent. Math.},
      volume={194},
      number={3},
       pages={673\ndash 729},
}

\bib{LargeSieve}{article}{
      author={Blomer, Valentin},
      author={Buttcane, Jack},
       title={Global decomposition of {$GL(3)$} {K}loosterman sums and the
  spectral large sieve},
        date={2018},
     journal={Crelles Journal},
         url={https://doi.org/10.1515/crelle-2017-0034},
        note={\mbox{doi}:10.1515/crelle-2017-0034},
}

\bib{Subconv}{article}{
      author={Blomer, Valentin},
      author={Buttcane, Jack},
       title={On the subconvexity problem for {$L$}-functions on {$GL(3)$}},
        date={arXiv:1504.02667, to appear in Ann. Sci. \'Ec. Norm. Sup\'er.},
}

\bib{GPSSubconv}{article}{
      author={Blomer, Valentin},
      author={Buttcane, Jack},
       title={Subconvexity for {$L$}-functions of non-spherical cusp forms on
  {$GL(3)$}},
        date={arXiv:1801.07611, to appear in Acta Arith.},
}

\bib{BFKMM01}{article}{
      author={Blomer, Valentin},
      author={Fouvry, \'{E}tienne},
      author={Kowalski, Emmanuel},
      author={Michel, Philippe},
      author={Mili\'{c}evi\'{c}, Djordje},
       title={On moments of twisted {$L$}-functions},
        date={2017},
     journal={Amer. J. Math.},
      volume={139},
      number={3},
       pages={707\ndash 768},
}

\bib{BFKMM}{article}{
      author={Blomer, Valentin},
      author={Fouvry, \'{E}tienne},
      author={Kowalski, Emmanuel},
      author={Michel, Philippe},
      author={Mili\'{c}evi\'{c}, Djordje},
       title={On moments of twisted {$L$}-functions},
        date={2017},
        ISSN={0002-9327},
     journal={Amer. J. Math.},
      volume={139},
      number={3},
       pages={707\ndash 768},
}

\bib{BKY01}{article}{
      author={Blomer, Valentin},
      author={Khan, Rizwanur},
      author={Young, Matthew},
       title={Distribution of mass of holomorphic cusp forms},
        date={2013},
        ISSN={0012-7094},
     journal={Duke Math. J.},
      volume={162},
      number={14},
       pages={2609\ndash 2644},
}

\bib{Brugg01}{article}{
      author={Bruggeman, R.~W.},
       title={Fourier coefficients of cusp forms},
        date={1978},
        ISSN={0020-9910},
     journal={Invent. Math.},
      volume={45},
      number={1},
       pages={1\ndash 18},
}

\bib{BFG}{article}{
      author={Bump, Daniel},
      author={Friedberg, Solomon},
      author={Goldfeld, Dorian},
       title={Poincar\'e series and {K}loosterman sums for {${\rm
  SL}(3,{\Z})$}},
        date={1988},
     journal={Acta Arith.},
      volume={50},
      number={1},
       pages={31\ndash 89},
}

\bib{Me01}{article}{
      author={Buttcane, Jack},
       title={On sums of {$SL(3,\mathbb{Z})$} {K}loosterman sums},
        date={2013},
     journal={Ramanujan J.},
      volume={32},
      number={3},
       pages={371\ndash 419},
}

\bib{SpectralKuz}{article}{
      author={Buttcane, Jack},
       title={The spectral {K}uznetsov formula on {$SL(3)$}},
        date={2016},
     journal={Trans. Amer. Math. Soc.},
      volume={368},
      number={9},
       pages={6683\ndash 6714},
}

\bib{HWI}{article}{
      author={Buttcane, Jack},
       title={Higher weight on {${\rm GL}(3)$}. {I}: {T}he {E}isenstein
  series},
        date={2018},
     journal={Forum Math.},
      volume={30},
      number={3},
       pages={681\ndash 722},
}

\bib{HWII}{article}{
      author={Buttcane, Jack},
       title={Higher weight on {$\rm GL(3)$}, {II}: {T}he cusp forms},
        date={2018},
        ISSN={1937-0652},
     journal={Algebra Number Theory},
      volume={12},
      number={10},
       pages={2237\ndash 2294},
}

\bib{ArithKuzI}{article}{
      author={Buttcane, Jack},
       title={The arithmetic {K}uznetsov formula on {$GL(3)$}, {I}: The
  {W}hittaker case},
        date={2019},
     journal={Revista Matem\'atica Iberoamericana},
         url={https://doi.org/10.4171/rmi/1113},
        note={\mbox{doi}:10.4171/rmi/1113},
}

\bib{WeylI}{article}{
      author={Buttcane, Jack},
       title={{K}uznetsov, {P}etersson and {W}eyl on {$GL(3)$}, {I}: The
  principal series forms},
        date={arXiv:1703.09837, to appear in Amer. J. Math.},
}

\bib{WeylII}{article}{
      author={Buttcane, Jack},
       title={{K}uznetsov, {P}etersson and {W}eyl on {$GL(3)$}, {II}: The
  generalized principal series forms},
        date={arXiv:1706.08816},
}

\bib{MeRizwan01}{article}{
      author={Buttcane, Jack},
      author={Khan, Rizwanur},
       title={On the fourth moment of {H}ecke-{M}aass forms and the random wave
  conjecture},
        date={2017},
     journal={Compos. Math.},
      volume={153},
      number={7},
       pages={1479\ndash 1511},
}

\bib{MeFan01}{article}{
      author={Buttcane, Jack},
      author={Zhou, Fan},
       title={{Plancherel Distribution of Satake Parameters of Maass Cusp Forms
  on {$\operatorname{GL}_3$}}},
        date={2018},
        ISSN={1073-7928},
     journal={International Mathematics Research Notices},
      eprint={https://doi.org/10.1093/imrn/rny061},
         url={https://doi.org/10.1093/imrn/rny061},
}

\bib{CHJT}{article}{
      author={Catto, Sultan},
      author={Huntley, Jonathan},
      author={Jorgenson, Jay},
      author={Tepper, David},
       title={On an analogue of {S}elberg's eigenvalue conjecture for {${\rm
  SL}_3(\mathbf Z)$}},
        date={1998},
        ISSN={0002-9939},
     journal={Proc. Amer. Math. Soc.},
      volume={126},
      number={12},
       pages={3455\ndash 3459},
}

\bib{CPS}{book}{
      author={Cogdell, James~W.},
      author={Piatetski-Shapiro, Ilya~I.},
       title={The arithmetic and spectral analysis of {P}oincar\'e series},
      series={Perspectives in Mathematics},
   publisher={Academic Press, Inc., Boston, MA},
        date={1990},
      volume={13},
}

\bib{DLMF}{misc}{
       title={{\it NIST Digital Library of Mathematical Functions}},
         how={http://dlmf.nist.gov/, Release 1.0.22 of 2019-03-15},
         url={http://dlmf.nist.gov/},
        note={F.~W.~J. Olver, A.~B. {Olde Daalhuis}, D.~W. Lozier, B.~I.
  Schneider, R.~F. Boisvert, C.~W. Clark, B.~R. Miller and B.~V. Saunders,
  eds.},
}

\bib{DabFish}{article}{
      author={D\polhk{a}browski, Romuald},
      author={Fisher, Benji},
       title={A stationary phase formula for exponential sums over {$\mathbf
  Z/p^m\mathbf Z$} and applications to {${\rm GL}(3)$}-{K}loosterman sums},
        date={1997},
     journal={Acta Arith.},
      volume={80},
      number={1},
       pages={1\ndash 48},
}

\bib{Duke01}{article}{
      author={Duke, W.},
       title={Hyperbolic distribution problems and half-integral weight {M}aass
  forms},
        date={1988},
     journal={Invent. Math.},
      volume={92},
      number={1},
       pages={73\ndash 90},
}

\bib{DFI02}{article}{
      author={Duke, W.},
      author={Friedlander, J.~B.},
      author={Iwaniec, H.},
       title={Equidistribution of roots of a quadratic congruence to prime
  moduli},
        date={1995},
     journal={Ann. of Math. (2)},
      volume={141},
      number={2},
       pages={423\ndash 441},
}

\bib{Erdelyi2}{book}{
      author={Erd\'{e}lyi, Arthur},
      author={Magnus, Wilhelm},
      author={Oberhettinger, Fritz},
      author={Tricomi, Francesco~G.},
       title={Higher transcendental functions. {V}ol. {II}},
   publisher={Robert E. Krieger Publishing Co., Inc., Melbourne, Fla.},
        date={1981},
        note={Based on notes left by Harry Bateman, Reprint of the 1953
  original},
}

\bib{Gode01}{article}{
      author={Godement, Roger},
       title={A theory of spherical functions. {I}},
        date={1952},
     journal={Trans. Amer. Math. Soc.},
      volume={73},
       pages={496\ndash 556},
}

\bib{Gold01}{book}{
      author={Goldfeld, Dorian},
       title={Automorphic forms and {$L$}-functions for the group {${\rm
  GL}(n,\mathbf R)$}},
      series={Cambridge Studies in Advanced Mathematics},
   publisher={Cambridge University Press, Cambridge},
        date={2006},
      volume={99},
        note={With an appendix by Kevin A. Broughan},
}

\bib{GoldKont}{article}{
      author={Goldfeld, Dorian},
      author={Kontorovich, Alex},
       title={On the determination of the {P}lancherel measure for
  {L}ebedev-{W}hittaker transforms on {${\rm GL}(n)$}},
        date={2012},
     journal={Acta Arith.},
      volume={155},
      number={1},
       pages={15\ndash 26},
}

\bib{GradRyzh}{book}{
      author={Gradshteyn, I.~S.},
      author={Ryzhik, I.~M.},
       title={Table of integrals, series, and products},
     edition={Eighth},
   publisher={Elsevier/Academic Press, Amsterdam},
        date={2015},
        note={Translated from the Russian, Translation edited and with a
  preface by Daniel Zwillinger and Victor Moll, Revised from the seventh
  edition},
}

\bib{KS01}{article}{
      author={Kim, Henry~H.},
       title={Functoriality for the exterior square of {$\mathrm{GL}_4$} and
  the symmetric fourth of {$\mathrm{GL}_2$}},
        date={2003},
     journal={J. Amer. Math. Soc.},
      volume={16},
      number={1},
       pages={139\ndash 183 (electronic)},
        note={With appendix 1 by Dinakar Ramakrishnan and appendix 2 by Kim and
  Peter Sarnak},
}

\bib{KontLeb01}{article}{
      author={Kontorovich, M.~I.},
      author={Lebedev, N.~N.},
       title={On a method of solving certain problems of diffraction theory and
  related problems},
        date={1938},
     journal={Zh. \'{E}ksperiment. Teor. Fiz.},
      volume={8},
       pages={1192\ndash 1206},
        note={in Russian},
}

\bib{Kuz01}{article}{
      author={Kuznetsov, Nikola\u{i}~Vasil'evich},
       title={The {P}etersson conjecture for cusp forms of weight zero and the
  {L}innik conjecture. {S}ums of {K}loosterman sums},
        date={1980},
     journal={Mat. Sb. (N.S.)},
      volume={111(153)},
      number={3},
       pages={334\ndash 383, 479},
}

\bib{LavGrondRath}{article}{
      author={Lavoie, J.~L.},
      author={Grondin, F.},
      author={Rathie, A.~K.},
       title={Generalizations of {W}hipple's theorem on the sum of a
  {${}_3F_2$}},
        date={1996},
        ISSN={0377-0427},
     journal={J. Comput. Appl. Math.},
      volume={72},
      number={2},
       pages={293\ndash 300},
}

\bib{Lebedev01}{article}{
      author={Lebedev, N.~N.},
       title={Sur un formule d'inversion},
        date={1946},
     journal={C. R. (Doklady) Acad. Sci. URSS (N.S.)},
      volume={52},
       pages={655\ndash 658},
}

\bib{MiaWallach01}{article}{
      author={Miatello, R.},
      author={Wallach, N.~R.},
       title={Kuznetsov formulas for real rank one groups},
        date={1990},
        ISSN={0022-1236},
     journal={J. Funct. Anal.},
      volume={93},
      number={1},
       pages={171\ndash 206},
}

\bib{Roelcke01}{article}{
      author={Roelcke, Walter},
       title={\"uber die {W}ellengleichung bei {G}renzkreisgruppen erster
  {A}rt},
        date={1953/1955},
     journal={S.-B. Heidelberger Akad. Wiss. Math.-Nat. Kl.},
      volume={1953/1955},
       pages={159\ndash 267 (1956)},
}

\bib{SearsTitch01}{article}{
      author={Sears, D.~B.},
      author={Titchmarsh, E.~C.},
       title={Some eigenfunction formulae},
        date={1950},
        ISSN={0033-5606},
     journal={Quart. J. Math., Oxford Ser. (2)},
      volume={1},
       pages={165\ndash 175},
}

\bib{Sel01}{article}{
      author={Selberg, Atle},
       title={Harmonic analysis and discontinuous groups in weakly symmetric
  {R}iemannian spaces with applications to {D}irichlet series},
        date={1956},
     journal={J. Indian Math. Soc. (N.S.)},
      volume={20},
       pages={47\ndash 87},
}

\bib{Stade02}{article}{
      author={Stade, Eric},
       title={Archimedean {$L$}-factors on {${\rm GL}(n)\times{\rm GL}(n)$} and
  generalized {B}arnes integrals},
        date={2002},
     journal={Israel J. Math.},
      volume={127},
       pages={201\ndash 219},
}

\bib{Stevens}{article}{
      author={Stevens, Glenn},
       title={Poincar\'e series on {${\rm GL}(r)$} and {K}loostermann sums},
        date={1987},
     journal={Math. Ann.},
      volume={277},
      number={1},
       pages={25\ndash 51},
}

\bib{T02}{book}{
      author={Terras, Audrey},
       title={Harmonic analysis on symmetric spaces and applications. {II}},
   publisher={Springer-Verlag},
     address={Berlin},
        date={1988},
}

\bib{Wallach}{book}{
      author={Wallach, Nolan~Russell},
       title={Real reductive groups. {II}},
      series={Pure and Applied Mathematics},
   publisher={Academic Press, Inc., Boston, MA},
        date={1992},
      volume={132},
}

\bib{Welsh}{article}{
      author={Welsh, Matthew},
       title={Spacing and a large sieve type inequality for roots of a cubic
  congruence},
        date={arXiv:1809.05211},
}

\bib{WilsonTTR}{article}{
      author={Wilson, J.~A.},
       title={Three-term contiguous relations and some new orthogonal
  polynomials},
        date={1977},
       pages={227\ndash 232},
}

\bib{Ye01}{article}{
      author={Ye, Yangbo},
       title={A {K}uznetsov formula for {K}loosterman sums on {${\rm GL}_n$}},
        date={2000},
     journal={Ramanujan J.},
      volume={4},
      number={4},
       pages={385\ndash 395},
}

\end{biblist}
\end{bibdiv}

\end{document}